\documentclass[11pt,letterpaper,reqno]{amsart}
\usepackage[OT2,T1]{fontenc}
\usepackage[utf8]{inputenc}
\usepackage[dvipsnames]{xcolor}
\usepackage{tikz-cd}
\usetikzlibrary{decorations.pathmorphing}
\usepackage{stmaryrd}
\SetSymbolFont{stmry}{bold}{U}{stmry}{m}{n}
\usepackage[pagebackref]{hyperref}
\hypersetup{
    pdfpagemode=FullScreen,
    colorlinks=true,
    linkcolor=red,
    filecolor=blue,
    urlcolor=blue,
    citecolor=ForestGreen,
    linktoc=page}
\usepackage[]{amsrefs}
\usepackage{microtype,amsmath,amsfonts,amssymb,amsthm,mathrsfs,bm}
\usepackage[scr=rsfs,cal=cm]{mathalfa}

\theoremstyle{plain}
\newtheorem{theorem}[]{Theorem}
\newtheorem{theorema}[]{Theorem}

\newtheorem{proposition}[theorem]{Proposition}
\newtheorem{lemma}[theorem]{Lemma}
\newtheorem{claim}[]{Claim}[section]
\newtheorem{corollary}[theorem]{Corollary}

\newtheorem{conjecturea}[]{Conjecture}

\theoremstyle{definition}
\newtheorem{definition}[]{Definition}

\newtheorem{hypothesisa}{Hypothesis}

\newtheorem{remark}[definition]{Remark}

\DeclareRobustCommand{\coprod}{\mathop{\text{\fakecoprod}}}
\newcommand{\fakecoprod}{%
  \sbox0{$\prod$}%
  \smash{\raisebox{\dimexpr.9625\depth-\dp0}{\scalebox{1}[-1]{$\prod$}}}%
  \vphantom{$\prod$}%
}


\usepackage{enumitem}
\numberwithin{theorem}{section}
\numberwithin{definition}{section}
\numberwithin{equation}{section}
\setlist[itemize]{leftmargin=3em}

\title[Mirror symmetry for singular double cover CY 
varieties]{Mirror symmetry for singular double cover Calabi--Yau varieties: quantum test}
\date{\today}

\author[Tsung-Ju Lee]{Tsung-Ju Lee}
\address{Tsung Ju Lee, Department of Mathematics, National Cheng Kung University, Tainan 70101, Taiwan.}
\email{tsungju@gs.ncku.edu.tw}

\author[Bong H.~Lian]{Bong H.~Lian}
\address{Bong H.~Lian, Shanghai Institute for Mathematics and Interdisciplinary Sciences, Yangpu District, Shanghai, 200433, China \& Center for Mathematics and Interdisciplinary Sciences,
Fudan University, Shanghai, 200433, China \& Department of Mathematics,
Brandeis University, Waltham, MA 02453, USA}
\email{lianbong@gmail.com}

\author[Shing-Tung~Yau]{Shing-Tung~Yau}
\address{Shing-Tung~Yau, Department of Mathematics, Harvard University, Cambridge MA 02138, U.S.A.~\& Yau Mathematical Sciences Center, Tsinghua University, Beijing 100084, China}
\email{yau@math.harvard.edu}

\begin{document}
\begin{abstract}
We continue our study on the pairs of singular Calabi--Yau 
varieties arising from double covers over semi-Fano toric manifolds. 
In this paper, we first investigate
singular CY double covers of \(\mathbb{P}^{3}\) branched along 
(1) a union of eight hyperplanes in general position, and (2)
a union of four hyperplanes and a quartic in generation. Our
previous construction produces hypothetical singular mirror partners.
We prove that
they are mirror pairs in the sense that the \(B\)-model of one (variation of Hodge
structure) is equivalent to the \(A\)-model of another (the untwisted part
of the genus zero orbifold Gromov--Witten invariants).
The technique can be generalized and applied to the case when
the nef-partition is trivial. As a byproduct, we also verify 
Morrison's conjecture in certain circumstances.

\end{abstract}
\maketitle 
\tableofcontents

\section{Introduction}
\subsection{Motivations}
Around 1990, inspired by mirror symmetry,
Candelas et al.~\cite{1991-Candelas-de-la-Ossa-Greene-Parkes-a-pair-of-calabi-yau-manifolds-as-an-exactly-soluable-superconformal-theory}
studied a pair of Calabi--Yau (CY) manifolds, 
the quintic and the (orbifold) Fermat quintic threefold found by Greene and Plesser 
\cite{1990-Greene-Plesser-duality-in-calabi-yau-moduli-space},
and predicted the numbers of rational
curves on quintic threefolds in \(\mathbb{P}^{4}\)
in a vicinity of the so-called
maximal unipotent monodromy point in the moduli of 
the (orbifold) Fermat quintic threefold. 
Since then,
mirror symmetry has drawn a lot of attention
and becomes one of the most active 
research areas in mathematics and physics.
In rough terms, a mirror pair is a pair of CY 
manifolds having the property that under an identification, called the \emph{mirror map},
the \(A\)-model correlation function of one is identical to the
\(B\)-model correlation function of another. Here,
the \(A\)-model is taken to be the genus zero Gromov--Witten theory, while
the \(B\)-model refers to variation of Hodge structure (VHS).

In recent work of Hosono, Lian, Takagi, and Yau, strong evidence showed that 
classical mirror symmetry can be extended to certain singular CY varieties.
In \cites{2020-Hosono-Lian-Takagi-Yau-k3-surfaces-from-configurations-of-six-lines-in-p2-and-mirror-symmetry-i,2019-Hosono-Lian-Yau-k3-surfaces-from-configurations-of-six-lines-in-p2-and-mirror-symmetry-ii-lambda-k3-functions}, they revisited
the family of \(K3\) surfaces arising from double covers branched along 
six lines in \(\mathbb{P}^{2}\) in general position, which
were studied by Matsumoto, Sasaki, and Yoshida 
\cites{1988-Matsumoto-Sasaki-Yoshida-the-period-map-of-a-4-parameter-family-of-k3-surfaces-and-the-aomoto-gelfand-hypergeometric-function-of-type-3-6,1992-Matsumoto-Sasaki-Yoshida-the-monodromy-of-the-period-map-of-a-4-parameter-family-of-k3-surfaces-and-the-hypergeometric-function-of-type-3-6} as a higher dimensional analogue of the Legendre family, 
and conjectured that the mirror family of the \(K3\) family is 
given by a certain family of double covers over a del Pezzo surface of degree \(6\),
which is a blow-up of three torus invariant points on \(\mathbb{P}^{2}\) 
(\cite{2019-Hosono-Lian-Yau-k3-surfaces-from-configurations-of-six-lines-in-p2-and-mirror-symmetry-ii-lambda-k3-functions}*{Conjecture 6.3}).
This conjecture has been subsequently investigated and 
tested by Hosono and the authors in many different approaches.
In \cite{2024-Hosono-Lee-Lian-Yau-mirror-symmetry-for-double-cover-calabi-yau-varieties},
they generalized the idea in \cites{2020-Hosono-Lian-Takagi-Yau-k3-surfaces-from-configurations-of-six-lines-in-p2-and-mirror-symmetry-i,2019-Hosono-Lian-Yau-k3-surfaces-from-configurations-of-six-lines-in-p2-and-mirror-symmetry-ii-lambda-k3-functions} and
studied CY varieties arising as double covers of
a semi-Fano projective toric manifold.
Given a semi-Fano projective toric manifold \(X\) together with a nef-parition on \(X\),
we consider a specfic type of families of double covers, called 
\emph{gauge fixed double cover branched along the nef-partition}, and 
proposed a mirror family. Loosely speaking, the mirror family is
also a certain family of gauge fixed double covers branched along a nef-partition, but
the base is now replaced by the Batyrev--Borisov's dual toric manifold
\(X^{\vee}\) and the nef-partition is taken to be the corresponding dual nef-partition.
In the \(K3\) case, the toric base is \(\mathbb{P}^{2}\) and the nef-partition
is taken to be \(-K_{\mathbb{P}^{2}}=h+h+h\), where \(h\) is the hyperplane class. 
In this situation, the Batyrev--Borisov's dual toric variety is exactly 
a blow-up of three torus invariant points on \(\mathbb{P}^{2}\).
It can be shown that when the dimension is less than
or equal to four the pairs constructed are topological mirror pairs, i.e.~
their Hodge diamonds are related by a 90 degree rotation \cite{2024-Hosono-Lee-Lian-Yau-mirror-symmetry-for-double-cover-calabi-yau-varieties}. 
We shall also emphasize that since our singular CY double covers are orbifolds,
the mixed Hodge structure on cohomology groups, say with \(\mathbf{Q}\)
coefficient, is indeed pure and hence the notion of
Hodge numbers is well-defined.

As the topological test is settled,
we now turn to the quantum test. One would like to carry
out the \(A\)-model and the \(B\)-model correlation functions 
and show that they are related under the mirror map. 
In the present circumstance, 
the \(B\) model is taken to be the variation of Hodge
structures for the equisingular family whereas the \(A\)-model 
turns out to be the untwisted part of the genus zero 
orbifold Gromov--Witten theory. 
On \(B\)-side, the period integrals for the equisingular family 
are governed by a GKZ \(A\)-hypergeometric system
with a fractional exponent. This type of GKZ \(A\)-hyergeometric systems has
also been studied by D.~Zhang and the first named author in
\cites{2023-Lee-a-note-on-periods-of-calabi-yau-fractional-complete-intersections,2025-Lee-Zhang-twisted-gkz-hypergeometric-functions-and-relative-homology},
and it turns out that the GKZ \(A\)-hypergeometric system is complete, namely all
the solutions are period integrals.
Mimicking the classical case, we found a close relationship between the principal parts
of the differential operators in the GKZ \(A\)-hypergeometric system and the
cohomology ring of the base of the conjectured mirror CY variety
\cite{2022-Lee-Lian-Yau-on-calabi-yau-fractional-complete-intersections};
this leads to a cohomology-valued \(B\)-series first introduced in 1994 by Hosono,
Lian, and Yau (a.k.a.~the \(I\)-function later) which plays a crucial role in mirror
symmetry. In order
to establish ``mirror theorem,'' we will have to compute
the untwisted part of the genus zero orbifold Gromov--Witten invariants
for singular double covers and compare them with the cohomology-valued \(B\)-series
from the mirror. 

The main purpose of this paper is providing 
further numerical evidence and proving a mirror theorem when the nef-partition
is a trivial partition. In summary, we will compute the 
untwisted part of the genus zero orbifold Gromov--Witten invariants
for singular double covers of \(X\) and show that it is equivalent to the \(B\)-series
from double covers of \(X^{\vee}\).
Besides, we will also investigate CY double covers of \(\mathbb{P}^{3}\) with various
branching locus, including non-trivial nef-partitions, and prove a mirror theorem in those cases.

We should also mention that 
in \cite{2023-Lee-Lian-Romo-non-commutative-resolutions-as-mirrors-of-singular-calabi-yau-varieties}, Romo, the first, and the second named 
authors proposed and investigated a categorical version of 
the mirror correspondence for singular CY double covers. Using gauged linear sigma model (GLSM),
they constructed a non-commutative resolution (NCR) as the \(A\)-side of
homological mirror symmetry. It was shown that the \(A\)-periods given by GLSM hemi-sphere 
partition functions of the NCR on one side agreed with \(B\)-periods (or VHS) of
the corresponding mirror partner family. The categorical
test was further extended in \cite{Lee:2025aa}
to include the classical mirror pairs, among other things,
again using GLSM machinery. 

\subsection{Statements of the main results}
We introduce some notation and then state our main results in this subsection.
Consider a nef-partition \((\Delta,\{\Delta_{i}\}_{i=1}^{r})\)
and its dual nef-partition \((\nabla,\{\nabla_{i}\}_{i=1}^{r})\) in the sense of
Batyrev and Borisov. 
Let \(\mathbf{P}_{\Delta}\) and \(\mathbf{P}_{\nabla}\)
be the toric varieties defined by \(\Delta\) and \(\nabla\). Let
\(X\to\mathbf{P}_{\Delta}\) and \(X^{\vee}\to \mathbf{P}_{\nabla}\) be
maximal projective crepant partial desingularizations (MPCP desingularizations for short hereafter)
of \(\mathbf{P}_{\Delta}\) and \(\mathbf{P}_{\nabla}\).
The nef-partitions on \(\mathbf{P}_{\Delta}\) and \(\mathbf{P}_{\nabla}\)
determine nef-partitions on \(X\) and \(X^{\vee}\). Let 
\(E_{1},\ldots,E_{r}\) and \(F_{1},\ldots,F_{r}\) be the sum of
toric divisors representing nef-partitions on \(X\) and \(X^{\vee}\), respectively.
Throughout this paper, we tacitly make the following assumption.
\begin{hypothesisa}
\label{assumption}
\(X\) and \(X^{\vee}\) are both \emph{smooth}.
\end{hypothesisa}
That is, we assume that both \(\Delta\) and \(\nabla\) 
admit a \emph{regular triangulation}. By a regular triangulation
of \(\Delta\), we mean that a triangulation of \(\Delta\) such that 
each simplex is regular and contains \(\mathbf{0}\) as a vertex.

We now define a \emph{partial gauge fixing} for such a family
and construct a family of \emph{gauge fixed double covers branched along a nef-parition}. 
For each \(1\le j\le r\), we pick a section \(s_{j}\in \mathrm{H}^{0}(X,E_{j})\)
such that 
\begin{equation}
  D_{\infty}\cup\bigcup_{i=1}^{r}\mathrm{div}(s_{j})
\end{equation}
is a strict normal crossing divisor. Here, \(D_{\infty}\) is the 
union of all toric divisors on \(X\).
A double cover with \(D_{\infty}\cup\bigcup_{i=1}^{r}\mathrm{div}(s_{j})\) as the branching divisor is called a
\emph{gauge fixed double cover branched along a nef-parition}. 
By deforming \(s_{j}\), we obtain a family of double covers
parametrized by an open subset
\begin{equation}
V\subset \mathrm{H}^{0}(X,E_{1})\times
\cdots\times\mathrm{H}^{0}(X,E_{r}).
\end{equation}
A parallel construction can be applied on the \(X^{\vee}\) side.
Let \(\mathcal{Y}\to V\), \(\mathcal{Y}^{\vee}\to U\)
be the gauge fixed double cover families
and \(Y\), \(Y^{\vee}\) be the fibers of these families.
Note that both \(Y\) and \(Y^{\vee}\) are orbifolds. 
\begin{conjecturea}
\((Y,Y^{\vee})\) is a mirror pair.
\end{conjecturea}


In this paper, we study the conjecture by 
the quantum test; we will compute the
cohomology-valued \(B\)-series from \(Y^{\vee}\) and 
compare it with the
untwisted part of the genus zero orbifold
Gromov--Witten invariants of \(Y\).

In the first part of this paper, we examine two
explicit examples: double covers
of \(\mathbb{P}^{3}\) whose branching divisor is
\begin{itemize}
\item[(1)] a union 
of eight hyperplanes in general position, or
\item[(2)] a union of four hyperplanes and
a quartic in general position.
  
\end{itemize}
In both cases, we are able to establish the following ``mirror theorem.''

\begin{theorema}[=Theorem \ref{thm:main-theorem-hhhh} and
Theorem \ref{thm:main-theorem-4h}]
The \(A\)-model correlation functions of \(Y\)
are identical with \(B\)-model correlation functions of \(Y^{\vee}\)
under the mirror maps in both cases (1) and (2).
\end{theorema}

The proof follows from a direct calculation. To compute the invariants
of the CY double cover \(Y\), the strategy is to embed the orbifold \(Y\) into another
toric orbifold, where the invariants can be computed
systematically, and then apply a version of quantum hyperplane
section theorem for orbifolds in 
\cite{2010-Tseng-orbifold-quantum-riemann-roch-lefschetz-and-serre}
to achieve our goal.
Thus the main task is to find an appropriate simplicial toric
variety in which the quantum hyperplane section theorem for \(Y\)
can be applied. In case (1), the ambient simplicial 
toric variety is a quotient of \(\mathbb{P}^{7}\) by 
a finite subgroup of its maximal torus and
\(Y\) is a quotient of \(\mathbb{P}^{7}[2,2,2,2]\)
(a smooth complete intersection of four quadrics)
by the same group, whereas in case (2)
the ambient simplicial 
toric variety is a quotient of \(\mathbb{P}(1,1,1,1,4)\) by 
a finite subgroup of its maximal torus and
\(Y\) is a quotient of \(\mathbb{P}(1,1,1,1,4)[8]\)
(a smooth degree 8 hypersurface) by the same group.

\begin{remark}
In case (1), the CY double cover \(Y\) of \(\mathbb{P}^{3}\)
is hence closely related to a CY complete intersection
\(\mathbb{P}^{7}[2,2,2,2]\), which
has been studied in \cites{2010-Caldararu-Distler-Hellerman-Pantev-Sharpe-non-birational-twisted-derived-equivalences-in-abelian-glsms,2013-Sharpe-predictions-for-gromov-witten-invariants-of-noncommutative-resolutions}. In fact,
through a gauged linear sigma model, Sharpe obtained
interesting predictions of ``Gromov--Witten invariants'' in 
\cite{2013-Sharpe-predictions-for-gromov-witten-invariants-of-noncommutative-resolutions}*{\S 4}. And now it is confirmed that these mysterious numbers are genus zero 
untwisted orbifold Gromov--Witten invariants of \(Y\),
the CY double cover of \(\mathbb{P}^{3}\) branched along eight hyperplanes
in general position.
\end{remark}

\begin{remark}
Our calculation for the mirror maps shows that
the \(\Gamma\) factors in the expression of holomorphic periods
are crucial. See also Remark \ref{rmk:gamma-factor}.
\end{remark}

\begin{remark}
In general, there are many embeddings one can potentially use. For instance,
in case (2), we can also embed our singular double cover \(Y\) into an orbifold - 
a quotient of \(\mathbb{P}_{\mathbb{P}^{3}}(\mathbb{C}\oplus\mathbb{L})\)
by a finite subgroup in its maximal torus. 
Here, \(\mathbb{L}\) is the total space of the anti-canonical bundle of 
\(\mathbb{P}^{3}\). The invariants can be
also obtained by manipulating Tseng's quantum hyperplane section theorem for orbifolds.
However, one then must take a non-trivial mirror map to 
obtain the correct series for invariants mainly because the Picard number is reduced
by one after taking the hyperplane sections; the original
two-variable series must be able to be transformed
into a one-variable series. Since we do not have a systematical way
to deal with the change of variable involved, we will not take this approach in this paper. The detail can be found in Appendix
\ref{app:comparison_of_computations}.
See also Remark \ref{rmk:not-p1-bundle}. 
\end{remark}

The construction for case (2) can be generalized
and we are able to prove a mirror theorem when
the nef-partition is a trivial partition, i.e.~\(r=1\). In which
case, we have \(E_{1}=-K_{X}\) and \(F_{1}=-K_{X^{\vee}}\).
We can summarize our second main result as follows.

\begin{theorema}[=Theorem \ref{thm:main}]
For \(r=1\), the cohomology-valued 
\(B\)-series constructed from the period integrals
of \(Y^{\vee}\) computes the genus zero untwisted orbifold 
Gromov--Witten invariants of \(Y\) 
with all insertions from the base \(X\) after a change of variables.
\end{theorema}

Our singular mirror proposal is also related to Morrison's conjecture
which states
that extremal transitions are reversed under mirror
symmetry \cite{1999-Morrison-through-the-looking-glass}. 
Here, an extremal transition 
is a birational contraction from
a smooth CY to a singular one and then followed by a complex
smoothing to another smooth CY.

By its nature, singular CY double cover \(Y\) of \(X\)
has a smoothing \(S\) by deforming the branching divisor.
In this way, \(S\) is a smooth double cover of \(X\)
and hence it is an anti-canonical hypersurface in a certain
semi-Fano toric manifold.
When \(r=1\), we can check that \(Y\)
admits a crepant resolution \(\tilde{Y}\to Y\)
and, more importantly, \(\tilde{Y}\) remains
an anti-canonical hypersurface in a certain toric orbifold
coming from a reflexive polytope. We thus have 
an extremal transition
\begin{equation}
  \begin{tikzcd}
    & &\tilde{Y}\ar[d]\\
    &S\ar[r,rightsquigarrow] &Y
  \end{tikzcd}
\end{equation}
This provides a nice place to test Morrison's conjecture
since mirrors of both \(S\) and \(\tilde{Y}\) are known
due to Batyrev. We will prove
\begin{theorema}[=Theorem \ref{thm:morrison}]
The mirrors \(\tilde{Y}^{\vee}\) and \(S^{\vee}\) 
are connected through an extremal transition. Indeed, we have
a mirror transition
\begin{equation}
  \begin{tikzcd}
    & &S^{\vee}\ar[d]\\
    &\tilde{Y}^{\vee}\ar[r,rightsquigarrow] &Y^{\vee}.
  \end{tikzcd}
\end{equation}
\end{theorema}


\quad \\
\noindent {\bf Acknowledgment}.~The authors would like to express 
our special thanks to our long-time collaborator Shinobu Hosono, 
who has played a critical role in initiating our program of studying 
CY mirror pairs as singular double covers. We also thank Mauricio Romo 
and Leonardo Santilli for their collaboration on studying these mirror 
pairs in the GLSM approach. Research of TJL is supported 
by NSTC 112-2115-M-006-016-MY3. 
Research of BHL is partially supported by SIMIS.

Note added: After the current paper has been completed, 
A.~Harder and S.~Lee \cite{Harder:2025aa} posted a proof of the 
topological mirror duality for double covers CYs in all dimensions, 
as was conjectured by Hosono and us, and 
proved in dimension 3 in \cite{2024-Hosono-Lee-Lian-Yau-mirror-symmetry-for-double-cover-calabi-yau-varieties}. 
The result has also been generalized to larger 
class of singular Galois covers. We also mention that
our result in Section 
\ref{sec:morrison_s_conjecture} has 
a significant 
overlap with results in \cite{Harder:2025aa}.

\section{Preliminaries}
\label{subsection:notation-all}
We begin with some notation and terminologies.
\begin{itemize}
\item Let \(N=\mathbb{Z}^{n}\) be a rank \(n\) lattice and
\(M=\mathrm{Hom}_{\mathbb{Z}}(N,\mathbb{Z})\) be its dual lattice. 
Let \(N_{\mathbb{R}}:=N\otimes_{\mathbb{Z}}\mathbb{R}\) and 
\(M_{\mathbb{R}}:=M\otimes_{\mathbb{Z}}\mathbb{R}\).
\item Let \(\Sigma\) be a fan in \(N_{\mathbb{R}}\) and \(X_{\Sigma}\) be
the toric variety determined by \(\Sigma\). 
Let \(T\subset X_{\Sigma}\) be its maximal torus 
with coordinates \(t_{1},\ldots,t_{n}\).
\item We denote by \(\Sigma(k)\) the set of \(k\)-dimensional cones in \(\Sigma\).
In particular, \(\Sigma(1)\) is the set of \(1\)-cones in \(\Sigma\). 
Similarly, for a cone \(\sigma\in\Sigma\),
we denote by \(\sigma(1)\) the set of \(1\)-cones belonging to \(\sigma\).
By abuse of the notation, we also denote by \(\rho\) 
the primitive generator of the corresponding 
\(1\)-cone.
\item Each \(\rho\) determines a \(T\)-invariant Weil divisor on \(X_{\Sigma}\),
which is denoted by \(D_{\rho}\) hereafter.
Any \(T\)-invariant Weil divisor \(D\) is of the form
\(D=\sum_{\rho\in\Sigma(1)} a_{\rho}D_{\rho}\). The polyhedron of \(D\) is defined to be
\begin{equation*}
\Delta_{D}:=\left\{m\in M_{\mathbb{R}}\colon \langle m,\rho\rangle\ge -a_{\rho}~
\mbox{for all}~\rho\right\}.
\end{equation*}
The integral points \(M\cap\Delta_{D}\) gives rise to a canonical
basis of \(\mathrm{H}^{0}(X_{\Sigma},D)\).
\item A \emph{nef-partition} on \(X_{\Sigma}\) 
is a decomposition of \(\Sigma(1)=\sqcup_{k=1}^{r} I_{k}\)
such that \(E_{k}:=\sum_{\rho\in I_{k}} D_{\rho}\) is nef for each \(k\). Recall that 
a divisor \(D\) is called nef if \(D.C\ge 0\) for any irreducible
complete curve \(C\subset X_{\Sigma}\).
We also have \(E_{1}+\cdots+E_{r}=-K_{X_{\Sigma}}\).
\item A polytope in \(M_{\mathbb{R}}\) is called a \emph{lattice polytope}
if its vertices belong to \(M\). For a lattice polytope \(\Delta\)
in \(M_{\mathbb{R}}\), we denote by \(\Sigma_{\Delta}\) the normal fan of 
\(\Delta\). The toric variety determined by \(\Delta\) is denoted by \(\mathbf{P}_{\Delta}\),
i.e., \(\mathbf{P}_{\Delta}=X_{\Sigma_{\Delta}}\).
\item A \emph{reflexive polytope} \(\Delta\subset M_{\mathbb{R}}\) is a lattice polytope 
containing the origin \(\mathbf{0}\in M_{\mathbb{R}}\) in its 
interior and such that the polar dual 
\(\Delta^{\vee}\) is again a lattice polytope. If
\(\Delta\) is a reflexive polytope, then \(\Delta^{\vee}\) is also a lattice
polytope and satisfies \((\Delta^{\vee})^{\vee}=\Delta\). The normal fan of \(\Delta\)
is the face fan of \(\Delta^{\vee}\) and vice versa.
\end{itemize}

\subsection{The Batyrev--Borisov duality construction}
\label{subsection:b-b-construction}
We briefly recall the construction of the dual nef-partition. The standard references are
\cites{1994-Batyrev-dual-polyhedra-and-mirror-symmetry-for-calabi-yau-hypersurfaces-in-toric-varieties,1996-Batyrev-Borisov-on-calabi-yau-complete-intersections-in-toric-varieties}.
Let \(I_{1},\ldots,I_{r}\) be a nef-partition on \(\mathbf{P}_{\Delta}\).
This gives rise to a Minkowski sum decomposition
\(\Delta=\Delta_{1}+\cdots+\Delta_{r}\), where \(\Delta_{i}=\Delta_{E_{i}}\)
is the section polytope of \(E_{i}\).
Following Batyrev--Borisov, let \(\nabla_{k}\) be the 
convex hull of \(\{\mathbf{0}\}\cup I_{k}\) and
\(\nabla=\nabla_1+\cdots+\nabla_{r}\) be their Minkowski sum.
One can prove that \(\nabla\) is a reflexive polytope in \(N_{\mathbb{R}}\)
whose polar dual is \(\nabla^{\vee}=\mathrm{Conv}(\Delta_{1},\ldots,\Delta_{r})\)
and \(\nabla_1+\cdots+\nabla_{r}\) corresponds to a nef-partition on \(\mathbf{P}_{\nabla}\),
called the \emph{dual nef-partition}.
The corresponding nef toric divisors are denoted by \(F_{1},\ldots,F_{r}\).
Then the section polytope of \(F_{j}\) is \(\nabla_{j}\).

Let \(X\to \mathbf{P}_{\Delta}\) and \(X^{\vee}\to \mathbf{P}_{\nabla}\)
be maximal projective crepant partial (MPCP for short hereafter) resolutions
for \(\mathbf{P}_{\Delta}\) and \(\mathbf{P}_{\nabla}\).
Via pullback, the nef-partitions on \(\mathbf{P}_{\Delta}\) and \(\mathbf{P}_{\nabla}\)
determine nef-partitions on \(X\) and \(X^{\vee}\) and they determine 
the families of Calabi--Yau complete intersections in \(X\) and \(X^{\vee}\) respectively.

Recall that the section polytopes \(\Delta_{i}\) and \(\nabla_{j}\)
correspond to \(E_{i}\) on \(\mathbf{P}_{\Delta}\) and 
\(F_{j}\) on \(\mathbf{P}_{\nabla}\), respectively.
To save the notation, the corresponding nef-partitions and toric divisors 
on \(X\) and \(X^{\vee}\) will be still denoted by \(\Delta_{i}\), \(\nabla_{j}\) and
\(E_{i}\), \(F_{j}\) respectively.

\subsection{Calabi--Yau double covers} 
\label{subsection:cy-double-covers}
We briefly review the construction of Calabi--Yau double covers in
\cite{2024-Hosono-Lee-Lian-Yau-mirror-symmetry-for-double-cover-calabi-yau-varieties}.
Let \(\Delta=\Delta_{1}+\cdots+\Delta_{r}\) and \(\nabla=\nabla_{1}+\cdots+\nabla_{r}\)
be a dual pair of nef-partitions  
representing \(E_{1}+\cdots+E_{r}\) on \(-K_{\mathbf{P}_{\Delta}}\)
and \(F_{1}+\cdots+F_{r}\) on \(-K_{\mathbf{P}_{\nabla}}\) respectively.
Let \(X\) and \(X^{\vee}\) be the MPCP resolution of \(\mathbf{P}_{\Delta}\) and \(\mathbf{P}_{\nabla}\) respectively.
Hereafter, we will simply call the decomposition \(\Delta=\Delta_{1}+\cdots+\Delta_{r}\)
a nef-partition on \(X\) for short with understanding the nef-partition \(E_{1}+\cdots +E_{r}\)
and likewise for the decomposition \(\nabla=\nabla_{1}+\cdots+\nabla_{r}\).
Unless otherwise stated, we assume that 

\begin{center}
{\it \(X\) and \(X^{\vee}\) are both smooth}.
\end{center}
Equivalently, we assume that both \(\Delta\) and \(\nabla\) admit regular triangulations\footnote{
By a regular triangulation, we mean a uni-modular triangulation such that each simplex contains 
the origion as a vertex. A regular triangulation is equivalent to 
a FRST (fine regular star triangulation).}.
From the duality, we have
\begin{equation*}
\mathrm{H}^{0}(X^{\vee},F_{i})\simeq \bigoplus_{\rho\in\nabla_{i}\cap N}\mathbb{C}\cdot t^{\rho}~\mbox{and}~
\mathrm{H}^{0}(X,E_{i})\simeq \bigoplus_{m\in\Delta_{i}\cap M}\mathbb{C}\cdot t^{m}.
\end{equation*}
Here we use the same notation \(t=(t_{1},\ldots,t_{n})\) to 
denote the coordinates on the maximal torus of \(X^{\vee}\) and \(X\).

A double cover \(Y^{\vee}\to X^{\vee}\) has trivial 
canonical bundle if and only if 
the branched locus is linearly equivalent to \(-2K_{X^{\vee}}\).
Let \(Y^{\vee}\to X^{\vee}\) be 
the double cover
constructed from the section \(s=s_{1}\cdots s_{r}\) with
\begin{equation*}
(s_{1},\ldots,s_{r})\in \mathrm{H}^{0}(X^{\vee},2F_{1})\times\cdots\times
\mathrm{H}^{0}(X^{\vee},2F_{r}).
\end{equation*}

We assume that \(s_{i}\in\mathrm{H}^{0}(X^{\vee},2F_{i})\) is 
of the form \(s_{i}=s_{i,1}s_{i,2}\) with \(s_{i,1},s_{i,2}\in \mathrm{H}^{0}(X^{\vee},F_{i})\).
We further assume that \(s_{i,1}\) is the section corresponding to the lattice point
\(\mathbf{0}\in\nabla_{i}\cap N\), i.e., the scheme-theoretic zero of \(s_{i,1}\)
is \(F_{i}\), and that the scheme-theoretic zero of \(s_{i,2}\) is non-singular. 
Deforming \(s_{i,2}\), we obtain a subfamily of double covers branched along
the nef-partition over \(X^{\vee}\) parameterized by an open subset
\begin{equation*}
V\subset \mathrm{H}^{0}(X^{\vee},F_{1})
\times\cdots\times\mathrm{H}^{0}(X^{\vee},F_{r}).
\end{equation*}

\begin{definition}
\label{definition:gauged-fixed-double-cover-family}
Given a decomposition \(\nabla=\nabla_{1}+\cdots+\nabla_{r}\) representing
a nef-partition \(F_{1}+\cdots+F_{r}\) on \(X^{\vee}\),
the subfamily \(\mathcal{Y}^{\vee}\to V\) 
constructed above is called the \emph{gauge fixed double cover family branched along
the nef-partition over \(X^{\vee}\)} or simply the \emph{gauge fixed double cover family} if 
no confuse occurs.
\end{definition}

Given a decomposition \(\nabla=\nabla_{1}+\cdots+\nabla_{r}\) representing
a nef-partition \(F_{1}+\cdots+F_{r}\) on \(X^{\vee}\) as above,
we denote by \(\mathcal{Y}^{\vee}\to V\) the gauge fixed double cover family. 
A parallel construction is applied for the dual 
decomposition \(\Delta=\Delta_{1}+\cdots+\Delta_{r}\) representing 
the dual nef-partition \(E_{1}+\cdots+E_{r}\) 
over \(X\) and this yields
another family \(\mathcal{Y}\to U\),
where \(U\) is an open subset in 
\begin{equation*}
\mathrm{H}^{0}(X,E_{1})\times\cdots\times\mathrm{H}^{0}(X,E_{r}).
\end{equation*}

\subsection{Notation and conventions}
\label{subsection:notation}
Let us fix the notation and conventions we are going to use throughout this note.
We resume the situation and notation in \S\ref{subsection:b-b-construction}.
\begin{itemize}
\item Let \(X\to\mathbf{P}_{\Delta}\) be a MPCP resolution and 
\(\Sigma\) be the fan defining \(X\). We will assume throughout this note that
both \(X\) and \(X^{\vee}\) are \emph{smooth}.
\item Let \(I_{1},\ldots,I_{r}\) be the induced nef-partition on \(X\) as before. 
We label the elements in \(I_{k}\) by \(i_{k,1},\ldots,i_{k,n_k}\) where 
\(n_{k}=\#I_{k}\). We define \(p=n_{1}+\cdots+n_{r}\).
We will write
\begin{equation*}
\Sigma(1)=\left\{\rho_{i,j}\right\}_{1\le i\le r,~1\le j\le n_{i}}.
\end{equation*}
For convenience, we will also write \(D_{i,j}\) for the 
Weil divisor associated with \(\rho_{i,j}\).
\item Let \(\nu_{i,j}:=(\rho_{i,j},\delta_{1,i},\ldots,\delta_{r,i})
\in N\times\mathbb{Z}^{r}\) be the lifting of \(\rho_{i,j}\),
where \(\delta_{i,j}\) is the Kronecker delta.
We additionally put 
\(\nu_{i,0}:=(\mathbf{0},\delta_{1,i},\ldots,\delta_{r,i})
\in N\times\mathbb{Z}^{r}\) for \(1\le i\le r\).
\item We define an order on the set of double indexes 
by declaring \((i,j)\preceq (i',j')\)
if and only if \(i\le i'\) or \(i=i'\) and \(j\le j'\).
Recall that \(\#\{(i,j)\colon 1\le i\le r,~0\le j\le n_{i}\}=p+r\).
There are unique bijections
\begin{align*}
\begin{split}
 J:=\{(i,j)\colon 1\le i\le r,~0\le j\le n_{i}\} &\to \{1,\ldots,p+r\}
\subset (\mathbb{Z},\le ),\\
 I:=\{(i,j)\colon 1\le i\le r,~1\le j\le n_{i}\} &\to \{1,\ldots,p\}
\subset (\mathbb{Z},\le ),
\end{split}
\end{align*}
preserving the order.
\item For a positive integer \(s\) and a matrix 
\(A_{\mathrm{ext}}\in\mathrm{Mat}_{s\times (p+r)}(\mathbb{Z})\) (resp.~ 
\(A\in\mathrm{Mat}_{s\times p}(\mathbb{Z})\)),
we will label the columns of \(A_{\mathrm{ext}}\) by the ordered set \(J\) 
(resp.~the columns of \(A\) by \(I\))
and speak the \((k,l)\)\textsuperscript{th} column of \(A_{\mathrm{ext}}\)
instead of the \((\sum_{1\le i\le k-1}(n_{i}+1)+l+1)\)\textsuperscript{th} 
column of \(A_{\mathrm{ext}}\) 
(resp.~the \((k,l)\)\textsuperscript{th} column of \(A\) instead of the
\((\sum_{1\le i\le k-1}n_{i}+l)\)\textsuperscript{th} column of \(A\)). 
For instance, 
for \(A_{\mathrm{ext}}\in\mathrm{Mat}_{s\times (p+r)}(\mathbb{Z})\), 
the \((1,0)\)\textsuperscript{th} column of \(A_{\mathrm{ext}}\) is the 
\(1\)\textsuperscript{st} column of \(A_{\mathrm{ext}}\).
The \((r,n_{r})\)\textsuperscript{th} column of \(A_{\mathrm{ext}}\) is the 
last column of \(A_{\mathrm{ext}}\).
\item Define the matrices
\begin{align*}
&A:=
\begin{bmatrix}
\nu_{1,1}^{\intercal} & \cdots & \nu_{r,n_{r}}^{\intercal}
\end{bmatrix}\in\mathrm{Mat}_{(n+r)\times p}(\mathbb{Z}),\\
&A_{\mathrm{ext}}:=
\begin{bmatrix}
\nu_{1,0}^{\intercal} & \cdots & \nu_{r,n_{r}}^{\intercal}
\end{bmatrix}\in\mathrm{Mat}_{(n+r)\times(p+r)}(\mathbb{Z}).
\end{align*}
According to our convention, the columns of \(A\) are labeled by \( I\)
and the columns of \(A_{\mathrm{ext}}\) are labeled by \( J\).
We have the following commutative diagram
\begin{equation*}
\begin{tikzcd}
& \mathbb{Z}^{p+r}\ar[d]\ar[r,"A_{\mathrm{ext}}"] & \mathbb{Z}^{n+r}\ar[d]\\
& \mathbb{Z}^{p}\ar[r,"A"] & \mathbb{Z}^{n}.
\end{tikzcd}
\end{equation*}
The left vertical map is given by forgetting the \((i,0)\)\textsuperscript{th}
component for all \(1\le i\le r\). The right vertical map
is given by projecting to the first \(n\) coordinates.
By assumption, \(A_{\mathrm{ext}}\) and \(A\) are surjective.
Let \(L_{\mathrm{ext}}:=\mathrm{ker}(A_{\mathrm{ext}})\) and \(L=\mathrm{ker}(A)\).
We then have
\begin{equation*}
\begin{tikzcd}
&0\ar[r] &L_{\mathrm{ext}}\ar[r]\ar[d] 
& \mathbb{Z}^{p+r}\ar[d]\ar[r,"A_{\mathrm{ext}}"] & \mathbb{Z}^{n+r}\ar[d]\ar[r] &0\\
&0\ar[r] &L\ar[r]
& \mathbb{Z}^{p}\ar[r,"A"] & \mathbb{Z}^{n}\ar[r] &0
\end{tikzcd}
\end{equation*}
where the leftmost vertical arrow is an isomorphism.

\item Each element \(\ell\in \mathbb{Z}^{s}\) can be 
uniquely written as \(\ell^{+}-\ell^{-}\) where
\(\ell^{\pm}\in\mathbb{Z}^{s}_{\ge 0}\) whose supports are disjoint.
\end{itemize}

\subsection{GKZ \texorpdfstring{\(A\)}{A}-hypergeometric systems}
\label{subsection:gkz}
We adopt the notation in \S\ref{subsection:notation}. 
For \(1\le i\le r\), let \(W_i = \mathbb{C}^{n_i+1}\). 
Let \(x_{i,0},\ldots, x_{i,n_i}\) be a fixed coordinate system on the
dual space \({W_i}^{\vee}\). Set
\(\partial_{i,j}=\partial/\partial x_{i,j}\).
Given the matrix \(A_{\mathrm{ext}}\) as above
and a parameter \(\beta \in \mathbb{C}^{n+r}\), 
the \(A\)-hypergeometric ideal \(I({A_{\mathrm{ext}}},\beta)\)
is the left ideal of the Weyl algebra \(\mathscr{D}=\mathbb{C}[x,\partial]\) on
the \emph{dual} vector space \(W^{\vee}:=W_{1}^{\vee}\times\cdots\times
W_{r}^{\vee}\) generated by the following two types of
operators
\begin{itemize}[leftmargin=3em]
\itemsep=3pt
\item The ``box operators'': \(\partial^{\ell^{+}} - \partial^{\ell^{-}}\),
where \(\ell^{\pm}\in \mathbb{Z}_{\geq 0}^{p+r}\) satisfy 
\(A_{\mathrm{ext}}\ell^{+}=A_{\mathrm{ext}}\ell^{-}\).
Here the multi-index convention is used.
\item The ``Euler operators'': \(\mathscr{E}_{k} - \beta_{k}\), where
\(\mathscr{E}_k=\sum_{(i,j)\in J}\langle \nu_{i,j},\mathrm{e}_k\rangle 
x_{i,j}\partial_{i,j}\).
Here \(\mathrm{e}_k=(\delta_{k,1},\ldots,\delta_{k,n+r})\in\mathbb{Z}^{n+r}\).
\end{itemize}
The \(A\)-hypergeometric system \(\mathcal{M}(A_{\mathrm{ext}},\beta)\) 
is the cyclic \(\mathscr{D}\)-module
\(\mathscr{D}\slash I({A_{\mathrm{ext}}},\beta)\).
As shown by
Gel'fand~et.~al.~\cite{1989-Gelfand-Kapranov-Zelevinski-hypergeometric-functions-and-toral-manifolds},
\(\mathcal{M}(A_{\mathrm{ext}},\beta)\) is a holonomic \(\mathscr{D}\)-module.

\begin{remark}
It is shown that the GKZ system \(\mathcal{M}(A_{\mathrm{ext}}, \beta)\) described 
above governs the periods associated with singular double cover 
Calabi--Yau varieties over \(X^{\vee}\)
branched along the \emph{dual} nef-partition \(F_{1}+\cdots+F_{r}\).
\end{remark}

\subsection{Stacky fans and Chen--Ruan cohomology}
The standard references for this subsection are
\cites{2005-Borisov-Chen-Smith-the-orbifold-chow-ring-of-toric-deligne-mumford-stacks,
2015-Coates-Corti-Iritani-Tseng-a-mirror-theorem-for-toric-stacks}.
A \emph{stacky fan} is a triple \(\boldsymbol{\Sigma}=
(N,\Sigma,\rho)\) where \(N\) is a finitely generated
abelian group, \(\Sigma\) is a simplicial fan in \(N_{\mathbb{Q}}:=
N\otimes_{\mathbb{Z}}\mathbb{Q}\) and \(\rho\colon \mathbb{Z}^{p}\to N\)
is a homomorphism. Let \(\{e_{1},\ldots,e_{p}\}\) be the standard basis 
for \(\mathbb{Z}^{p}\). We will denote by \(b\otimes 1\) the
image of \(b\in N\) under the canonical map \(N\to N_{\mathbb{Q}}\).
The data gives rise to an exact sequence
\begin{equation}
  0\to \mathfrak{L}\to \mathbb{Z}^{p}\to N.
\end{equation}
For every \(\sigma\in \Sigma\), we denote by \(\Lambda_{\sigma}\subset 
\mathfrak{L}\otimes_{\mathbb{Z}}\mathbb{Q}\) the elements of the form
\begin{equation}
  \lambda = \sum_{i=1}^{p}\lambda_{i} e_{i}~\mbox{with}~\lambda_{i}\in\mathbb{Z}~\mbox{for}~e_{i}\otimes 1\notin \sigma.
\end{equation}
For a stacky fan \(\boldsymbol{\Sigma}\), we define
\begin{equation}
    \mathrm{Box}(\boldsymbol{\Sigma}):=\bigcup_{\sigma\in\Sigma} \mathrm{Box}(\sigma)
\end{equation}
where for \(\sigma\in \Sigma\) we put
\begin{equation}
  \mathrm{Box}(\sigma)=\left\{b\otimes 1\in N_{\mathbb{Q}}\Bigm\vert
  b\otimes 1=\sum_{e_{i}\otimes 1\in\sigma} a_{i}(e_{i}\otimes 1)~\mbox{for
  some}~0\le a_{i}<1\right\}.
\end{equation}
Set \(\Lambda:=\bigcup_{\sigma\in\Sigma}\Lambda_{\sigma}\). Recall that
the \emph{reduction function} \(v\colon\Lambda\to 
\mathrm{Box}(\boldsymbol{\Sigma})\) is a function defined by
\begin{equation}
  \lambda\mapsto \sum_{i=1}^{p}\lceil \lambda_{i}\rceil\cdot\rho(e_{i})\in N.
\end{equation}
For \(b\in \mathrm{Box}(\boldsymbol{\Sigma})\), we define
\begin{equation}
  \Lambda_{b}:=\{\lambda\in\Lambda~\vert~v(\lambda)=b\}.
\end{equation}

We also review the definition of the Chen--Ruan cohomology 
for an orbifold \(\mathcal{X}=[X\slash G]\) when \(X\) is a smooth variety and \(G\)
is a finite group which is sufficient for our purpose in this note.
For an orfiold \(\mathcal{X}\), we denote by \(|\mathcal{X}|\)
the underlying coarse moduli space.

Recall that for a stack \(\mathcal{X}\), the inertia stack \(\mathcal{IX}\) is the 
fiber product (in the category of \(2\)-category of stacks)
\begin{equation*}
\begin{tikzcd}
\mathcal{IX}\ar[r]\ar[d] &\mathcal{X}\ar[d,"\Delta"]\\
\mathcal{X}\ar[r,"\Delta"] &\mathcal{X}\times\mathcal{X}\\
\end{tikzcd}
\end{equation*}
where \(\Delta\colon\mathcal{X}\to\mathcal{X}\times\mathcal{X}\) is the diagonal map.
For the quotient stack \(\mathcal{X}=[X\slash G]\), its inertia stack is of the form
\begin{equation*}
\mathcal{I}[X\slash G]=\coprod_{(g)\in \mathrm{C}(G)} [X^{g}\slash \mathrm{C}(g)]
\end{equation*}
where \(\mathrm{C}(G)\) is the set of conjugacy classes of \(G\),
\(\mathrm{C}(g)\) is the centralizer of an element \(g\), and
\(X^{g}\) is the fixed part of \(g\).
\begin{definition}
For an orbifold \(\mathcal{X}\), the Chen--Ruan cohomology is defined to be
\begin{equation*}
\mathrm{H}^{\bullet}_{\mathrm{CR}}(\mathcal{X};\mathbb{C}):=
\mathrm{H}^{\bullet}(|\mathcal{IX}|;\mathbb{C}).
\end{equation*}
Here 
the right hand side is the singular cohomology.
\end{definition}
When \(\mathcal{X}=[X\slash G]\), we have
\begin{equation}
\label{eq:cr-twisted}
\mathrm{H}^{\bullet}_{\mathrm{CR}}(\mathcal{X};\mathbb{C})=
\bigoplus_{(g)\in \mathrm{C}(G)} \mathrm{H}^{\bullet}
(\bigl\vert[X^{g}\slash \mathrm{C}(g)]\bigr\vert;\mathbb{C}).
\end{equation}
The components in \eqref{eq:cr-twisted} are referred 
to \emph{twisted sectors} whereas
the distinguished component corresponding to \(e\in G\) is called 
the \emph{untwisted sector}.

We will be focusing on the untwisted sector; it corresponds to
\(\mathbf{0}\in \mathrm{Box}(\boldsymbol{\Sigma})\).

\section{Double covers of \texorpdfstring{\(\mathbb{P}^{3}\)}{P3}
with the nef-partition \texorpdfstring{\(-K_{\mathbb{P}^{3}}=h+h+h+h\)}{-K=h+h+h+h}}
Let us briefly recall the results developed in 
\cite{2024-Hosono-Lee-Lian-Yau-mirror-symmetry-for-double-cover-calabi-yau-varieties}.
The pair of singular Calabi--Yau double covers \((Y,Y^{\vee})\) we have
constructed satisfies the equality \(\chi_{\mathrm{top}}(Y)=(-1)^{n}
\chi_{\mathrm{top}}(Y^{\vee})\), where \(n=\dim Y\). Moreover, when \(n=3\),
we proved that 
\begin{equation*}
h^{p,q}(Y) = h^{3-p,q}(Y^{\vee}),~\forall~0\le p,q\le 3.
\end{equation*}
In other words, \((Y,Y^{\vee})\) is a \emph{topological} mirror pair of Calabi--Yau spaces.

After the ``topological test,'' we now turn to the ``quantum test.''
We study the relationship between enumerative geometry 
(the \(A\)-model) and complex geometry (the \(B\)-model).
Notice that both \(Y\) and \(Y^{\vee}\) are singular; they are orbifolds.
The \(A\) model here is thus taken to be the \emph{genus zero orbifold Gromov--Witten theory}
whereas the \(B\) model is the \emph{equisingular complex deformation theory}.

In this section, we will conduct the ``quantum test'' for 
our gauged fixed double cover branched along the maximal 
nef-partition \(H+H+H+H=-K_{\mathbb{P}^{3}}\) over \(\mathbb{P}^{3}\).

Let \(\Delta\) be the convex hull of 
\begin{align*}
(3,-1,-1),~(-1,3,-1),~(-1,-1,3),~(-1,-1,-1).
\end{align*}
Put \(X=\mathbf{P}_{\Delta}=\mathbb{P}^{3}\) and denote by \(H\) the hyperplane class.
We have the following data.
\begin{itemize}
    \item \(\Delta=\Delta_{1}+\cdots+\Delta_{4}\) is the 
    Minskowski sum decomposition representing the nef-partition \(-K_{X}=H+H+H+H\).
    \item \(\nabla=\nabla_{1}+\cdots+\nabla_{4}\) is the Batyrev--Borisov dual nef-partition. 
\end{itemize}
Let \(X^{\vee}\to\mathbf{P}_{\nabla}\) be any MPCP desingularization.
Since \(\dim X=\dim X^{\vee}=3\), we infer that \(X^{\vee}\) is smooth.
Let \(\mathcal{Y}\to V\) and \(\mathcal{Y}^{\vee}\to U\) be the families of 
Calabi--Yau double covers over \(X\) and \(X^{\vee}\) constructed in 
\S\ref{subsection:cy-double-covers} respectively.
Let \(Y\) and \(Y^{\vee}\) be the fiber of 
\(\mathcal{Y}\to V\) and \(\mathcal{Y}^{\vee}\to U\).
Notice that we have \(h^{1,1}(Y)=h^{2,1}(Y^{\vee})=1\).

In the present case, on the \(X\) side, we have
\begin{eqnarray*}
\rho_{1,1}=(1,0,0),~\rho_{2,1}=(0,1,0),~\rho_{3,1}=(0,0,1)~\mbox{and}~\rho_{4,1}=(-1,-1,-1).
\end{eqnarray*}

\subsection{Picard--Fuchs equations for \texorpdfstring{\(\mathcal{Y}^{\vee}\to U\)}{}}
\label{subsection:pf-equations-cyclic-cy-3-fold}

From the construction, 
the integral points in the section polytopes of \(F_{k}\) 
correspond to the integral points in \(\mathrm{Conv}\{\mathbf{0},\rho_{k,1}\}\).
The GKZ hypergeometric system associated with \(\mathcal{Y}^{\vee}\to U\) is given by 
\begin{equation*}
A_{\mathrm{ext}}=\begin{bmatrix}
  1 & 1 & 0 & 0 & 0 & 0 & 0 &  0\\
  0 & 0 & 1 & 1 & 0 & 0 & 0 &  0\\
  0 & 0 & 0 & 0 & 1 & 1 & 0 &  0\\
  0 & 0 & 0 & 0 & 0 & 0 & 1 &  1\\
  1 & 0 & 0 & 0 & 0 & 0 &-1 &  0\\
  0 & 0 & 1 & 0 & 0 & 0 &-1 &  0\\
  0 & 0 & 0 & 0 & 1 & 0 &-1 &  0\\
\end{bmatrix}~\mbox{and}~
\beta=
\begin{bmatrix}
-1/2\\
-1/2\\
-1/2\\
-1/2\\
0\\
0\\
0\\
\end{bmatrix}
\end{equation*}
The lattice relation given by \( A \) is \(L_{\mathrm{ext}}=
\langle \ell\rangle_{\mathbb{Z}}\) with \(\ell:=(1,-1,1,-1,1,-1,1,-1)\).

From the lattice relation, the box operators is
\begin{equation*}
\label{equation:GKZ-polynomial-operators-P3-mirror}
\Box_{k\ell}=\partial_{x_{1,1}}^{k}\partial_{x_{2,1}}^{k}
\partial_{x_{3,1}}^{k}\partial_{x_{4,1}}^{k}-
\partial_{x_{1,0}}^{k}\partial_{x_{2,0}}^{k}\partial_{x_{3,0}}^{k}\partial_{x_{4,0}}^{k}
=\prod_{i=1}^{4} \partial^{k}_{x_{i,1}}-\prod_{i=1}^{4} \partial^{k}_{x_{i,0}},~
k\in\mathbb{Z}_{\ge0}
\end{equation*}
or with a minus sign if \(k<0\).
Let us consider the case \(k=1\). We have
\begin{align}
\begin{split}
&(\textstyle\prod_{i=1}^{4}x_{i,0})^{1/2}(\textstyle\prod_{i=1}^{4}x_{i,1})
\Box_{\ell}(\textstyle\prod_{i=1}^{4}x_{i,0})^{-1/2}\\
&=\textstyle\prod_{i=1}^{4}\theta_{x_{i,1}}
-z(\textstyle\prod_{i=1}^{4}x_{i,0})^{3/2}
(\textstyle\prod_{i=1}^{4}x_{i,0}\partial_{x_{i,0}})
(\textstyle\prod_{i=1}^{4}x_{i,0})^{-1/2}
\end{split}
\end{align}
where \(z=(x_{1,1}x_{2,1}x_{3,1}x_{4,1})\slash(x_{1,0}x_{2,0}x_{3,0}x_{4,0})\).
Furthermore, from the equality
\begin{equation*}
x_{i,0}^{3/2}\partial_{x_{i,0}}x_{i,0}^{-1/2}=(x_{i,0}\partial_{x_{i,0}}-1/2).
\end{equation*}
Then \eqref{equation:GKZ-polynomial-operators-P3-mirror} becomes
\begin{equation}
\textstyle\prod_{i=1}^{4}\theta_{x_{i,1}}-z
\textstyle\prod_{i=1}^{4}(\theta_{x_{i,0}}-1/2).
\end{equation}
Here \(\theta_{a}=a(\mathrm{d}\slash\mathrm{d}a)\) 
is the logarithmic derivative with respect to \(a\).
Substituting
\begin{equation*}
\theta_{x_{i,0}}=\theta_{z},~\theta_{x_{i,1}}=-\theta_{z}
\end{equation*}
we see that \eqref{equation:GKZ-polynomial-operators-P3-mirror} is transformed into
\begin{equation}
\label{equation:picard-fuchs-equation-for-mirror-p3-cover}
\theta_{z}^{4} - z (\theta_{z}+1/2)^{4}.
\end{equation}

The unique holomorphic series solution to 
\eqref{equation:picard-fuchs-equation-for-mirror-p3-cover} is of the form
\begin{equation}
\label{equation:picard-fuchs-equation-series-sol} 
\sum_{n\ge 0} \frac{\Gamma(n+1/2)^{4}}{\Gamma(1/2)^{4}\Gamma(n+1)^{4}}z^{n}.
\end{equation}

\begin{remark}
\label{remark:2222-in-p7-coodinate}
The equation \eqref{equation:picard-fuchs-equation-for-mirror-p3-cover} has 
been studied in the literature.
Introducing a change of variables \( w=z/256 \), we have \( \theta_{w}=\theta_{z} \) and 
\begin{equation}
\label{equation:picard-fuchs-equation-for-mirror-2-2-2-2-complete-intersection}
\eqref{equation:picard-fuchs-equation-for-mirror-p3-cover} = 
\theta_{w}^{4}-256w(\theta_{w}+1/2)^4,
\end{equation}
which is the Picard--Fuchs equation for the mirror of \( \mathbb{P}^{7}[2,2,2,2]\subset\mathbb{P}^{7}\).
\end{remark}

\subsection{An instanton prediction from mirror symmetry}
\label{subsection:inst-prediction-h-h-h-h}
In this paragraph,
we compute the \(B\) model correlation function (Yukawa coupling), 
the mirror map, and its instanton prediction for the 
one parameter family \(\mathcal{Y}^{\vee}\to U\). 
It follows the result in 
\cite{2013-Sheng-Xu-Zuo-maximal-families-of-calabi-yau-manifolds-with-minimal-length-yukawa-coupling}*{Corollary 2.6}
that \(Y^{\vee}\) admits a crepant resolution
\( \tilde{Y}^{\vee} \) and such a resolution is deformed in family. Let
\(\tilde{\mathcal{Y}}^{\vee}\to \mathcal{Y}^{\vee}\to U\) be the resulting family.
It is shown that
\( h^{2,1}(\tilde{Y}^{\vee})=h^{2,1}(Y^{\vee})=1 \). Let
\begin{equation}
\label{equation:yukawa-coupling}
\left\langle \theta_{z},\theta_{z},\theta_{z}\right\rangle^{\Omega}
:=\int_{\tilde{Y}^{\vee}} \Omega(z)\wedge \theta_{z}^{3}\Omega(z),
~\mbox{\(\Omega\colon\)a local section of 
\(\Omega^{3}_{\tilde{\mathcal{Y}}^{\vee}\slash U}\).
}
\end{equation}
\(\left\langle \theta_{z},\theta_{z},\theta_{z}\right\rangle^{\Omega}\)
is the \(B\) model correlation function,
where the notation \(\theta_{z}\Omega\) means 
differentiating \(\Omega\) with respect to \(\theta_{z}\) 
via Gauss--Manin connection.


By Griffiths transversality, 
\begin{equation*}
\int_{\tilde{Y}^{\vee}} \Omega(z)\wedge \theta_{z}^{2}\Omega_{z}=0.
\end{equation*}
Differentiating the displayed equation twice, we obtain
\begin{equation*}
\int_{\tilde{Y}^{\vee}} \theta_{z}\Omega(z)\wedge \theta_{z}^{3}\Omega(z)+
\theta_{z}\left\langle \theta_{z},\theta_{z},\theta_{z}\right\rangle^{\Omega}=0.
\end{equation*}
By chain rule, we then have
\begin{equation*}
\theta_{z}\left(\int_{\tilde{Y}^{\vee}} \Omega(z)\wedge \theta_{z}^{3}\Omega(z)\right)
-\int_{\tilde{Y}^{\vee}} \Omega(z)\wedge \theta_{z}^{4}\Omega(z)
+\theta_{z}\left\langle \theta_{z},\theta_{z},\theta_{z}\right\rangle^{\Omega}=0;
\end{equation*}
in other words,
\begin{equation*}
2\theta_{z}\left\langle \theta_{z},\theta_{z},\theta_{z}\right\rangle^{\Omega}
-\int_{\tilde{Y}^{\vee}} \Omega(z)\wedge \theta_{z}^{4}\Omega(z)=0.
\end{equation*}
Substituting the last term by the Picard--Fuchs equation 
\eqref{equation:picard-fuchs-equation-for-mirror-2-2-2-2-complete-intersection},
we have derived
\begin{equation*}
\theta_{z}\left\langle \theta_{z},\theta_{z},\theta_{z}\right\rangle^{\Omega}
=\frac{z}{1-z}\left\langle \theta_{z},\theta_{z},\theta_{z}\right\rangle^{\Omega}.
\end{equation*}
We can solve the above equation and get
\begin{equation*}
\left\langle \theta_{z},\theta_{z},\theta_{z}\right\rangle^{\Omega} = \frac{C}{1-z},
\end{equation*}
for some constant \(C\). One can check 
the \emph{normalized Yukawa coupling}
\begin{equation*}
\left\langle \theta_{z},\theta_{z},\theta_{z}\right\rangle
:=\int_{\tilde{Y}^{\vee}} \frac{\Omega(z)}{y_{0}(z)}\wedge \theta_{z}^{3}
\left(\frac{\Omega(z)}{y_{0}(z)}\right)
\end{equation*}
is given by
\begin{equation}
\label{equation:normal-yukawa-coupling}
\left\langle \theta_{z},\theta_{z},\theta_{z}\right\rangle = 
\frac{C}{(1-z)y_{0}(z)^{2}},
\end{equation}
where \(y_{0}(z)\) is the holomorphic series solution 
\eqref{equation:picard-fuchs-equation-series-sol}.

Now we compute the ``mirror map.''
Consider the deformed series
\begin{equation}
\label{eq:gamma-hol-series-deformed}
y_0(z;\rho):=\sum_{n\ge 0} 
\frac{\Gamma(n+\rho+1/2)^4}{\Gamma(1/2)^{4}\Gamma(n+\rho+1)^{4}} z^{n+\rho} 
\end{equation}
and its derivative with respect to \(\rho\)
\begin{equation*}
y_{1}(z):=\left.\frac{\mathrm{d}}{\mathrm{d}\rho}\right|_{\rho=0} y_{0}(z;\rho).
\end{equation*}
Consequently, the ``mirror map'' is given by
\begin{equation}
\label{equation:mirrir-map}
q = \exp\left(2\pi\sqrt{-1} t\right),~t = \frac{1}{2\pi\sqrt{-1}}\frac{y_{1}(z)}{y_{0}(z)}.
\end{equation}
Let us again denote by \( H \) the unique hyperplane class of \( Y \)
coming from \(X\) and \( \langle H,H,H\rangle \) 
be the \(A\) model correlation function.

Mirror symmetry predicts the equality (the ``mirror theorem'')
\begin{equation}
\langle H,H,H\rangle = \left\langle \theta_{z},\theta_{z},\theta_{z}\right\rangle
\end{equation}
under the identification via the mirror map \eqref{equation:mirrir-map} and \( q = \exp (2\pi\sqrt{-1} t)\),
where \(t\) is the coordinate on the K\"{a}hler moduli of \(Y\). 
\(H\) is understood as the operator 
\begin{equation*}
H = 2\pi\sqrt{-1} q\frac{\mathrm{d}}{\mathrm{d} q}
\end{equation*}
and the mirror theorem becomes the equality
\begin{equation}
\langle H,H,H\rangle = \left\langle \theta_{z},\theta_{z},\theta_{z}
\right\rangle(q)\left(2\pi\sqrt{-1}\frac{q}{z}\frac{\mathrm{d}z}{\mathrm{d}q}\right)^{3}.
\end{equation}
Using the classical cup product, one finds \( C = 2 \).
In the present situation, the mirror map is
\begin{equation*}
q = \frac{z}{256} + \frac{z^{2}}{1024} + \frac{221z^{3}}{524288} + \frac{121z^{4}}{524288} 
+ \frac{9924061z^{5}}{68719476736}+\cdots,
\end{equation*}
and the inverse is given by 
\begin{equation*}
z = 256q - 16384 q^{2} + 286720 q^{3} - 9961472 q^{4} - 393334784 q^{5} + \cdots.
\end{equation*}
The \(A\)-model correlation function is
\begin{align}
\begin{split}
\label{eq:a-model-predicted-correlations}
&\langle H,H,H\rangle(q) \\
&= 2 + 64 q + 9792 q^2 + 1404928 q^3 + 205641280 q^4 + 30593496064 q^5 + \cdots.
\end{split}
\end{align}
Consequently, we obtain the following numerical result.
\begin{corollary}
The predicted instanton numbers \(n_{d}\) of \(Y\) for small \(d\) are given by
\begin{align*}
\begin{split}
n_{1} = 64,~n_{2} = 1216,~ n_{3} = 52032,~ n_{4} = 3212992.
\end{split}
\end{align*}
\end{corollary}
\begin{corollary}
\label{cor:prediction-1/8}
The predicted instanton numbers \(n_{d}\) of \(Y\) and those
of \(\mathbb{P}^{7}[2,2,2,2]\subset\mathbb{P}^{7}\), a smooth Calabi--Yau complete intersection
of degree \((2,2,2,2)\) in \(\mathbb{P}^{7}\), differ from an overall factor \(1/8\).
\end{corollary}
\begin{proof}
This follows since their mirrors have the same Picard--Fuchs equation.
See Remark \ref{remark:2222-in-p7-coodinate}.
The instanton predictions then only differ by an overall factor which
is completely determined by the classical intersection numbers.
\end{proof}

\begin{remark}
\label{rmk:gamma-factor}
This computation also shows that 
the Gamma factor in the holomorphic period \(y_{0}(z)\) 
is crucial. Put
\begin{eqnarray*}
a_{\rho}(n):=\frac{\Gamma(n+\rho+1/2)^{4}}{\Gamma(1/2)^{4}\Gamma(n+\rho+1)^{4}}
~\mbox{and}~
b_{\rho}(n):=\frac{\prod_{k=1}^{n}(k+\rho-1/2)^{4}}{\prod_{k=1}^{n}(k+\rho)^{4}}.
\end{eqnarray*}
We then have
\begin{equation*}
a_{\rho}(n) = b_{\rho}(n)A(\rho),~\mbox{where}~
A(\rho)=\frac{\Gamma(\rho+1/2)^{4}}{\Gamma(\rho+1)^{4}}.
\end{equation*}
Notice that 
\begin{equation*}
A(\rho) = 1 -(\log256)\rho+\cdots.
\end{equation*}
Hence the Frobenius method applied to \(\sum_{n\ge 0}a_{\rho}(n)z^{n}\)
and \(\sum_{n\ge 0}b_{\rho}(n)z^{n}\) yields different results.
\(A'(0)=-\log256\) explains the factor \(256\) in Remark \ref{remark:2222-in-p7-coodinate}.
\end{remark}

\begin{remark}
The instanton predictions were also obtained by Sharpe
\cite{2013-Sharpe-predictions-for-gromov-witten-invariants-of-noncommutative-resolutions}
using the technique gauged linear sigma model (GLSM).
\end{remark}


\subsection{An instanton calculation}
In this section, we compute the ``untwisted'' orbifold Gromov--Witten invariants of \(Y\)
through a pre-quotient model constructed in \cites{1988-Terasoma-complete-intersetions-of-hypersurfaces-the-fermat-case-and-the-quadric-case,2013-Gerkmann-Sheng-van-Straten-Zuo-on-the-monodromy-of-the-moduli-space-of-calabi-yau-threefolds-coming-from-eight-planes-in-p3,2007-Dolgachev-Kondo-moduli-of-k3-surfaces-and-complex-ball-quotients}.
Let us briefly review their construction.

\begin{definition}[Hyperplane arrangements]
A set of ordered $m$ hyperplanes in $\mathbb{P}^n$,
denoted by $\mathfrak{A}=(H_1,\ldots,H_m)$,
is called an \emph{$m$-hyperplane arrangement}.
A hyperplane arrangement is said to be
\emph{in general position} if any $n+1$ of them
do not meet.
\end{definition}

Let \([z_1\mathpunct{:}\ldots\mathpunct{:}z_{n+1}]\) 
be homogeneous coordinates on \(\mathbb{P}^n\).
We write
\begin{equation}
	H_i:=\sum_{j=1}^{n+1} a_{ij} z_j,~i=1,\ldots,m.
\end{equation}
To save the notation, we again denote by
\(\mathfrak{A}=(a_{ij})\in\mathrm{Mat}_{m\times(n+1)}(\mathbb{C})\)
the coefficient matrix.
Hyperplane arrangements parameterized by those matrices.
Note that \(\mathfrak{A}\) is in general position
if and only if every \((n+1)\times(n+1)\) 
minor of \(\mathfrak{A}\) is invertible.
Let \(Y\) be the \(2\)-fold cover
over \(\mathbb{P}^{n}\) branched along \(\sum_{i=1}^m H_i\).

We will focus on the case \(m=2(n+1)\).
Given a hyperplane arrangement $\mathfrak{A}$
in general position and the associated coefficient matrix \(\mathfrak{A}\),
we can find \(\mathfrak{B}=(b_{ij})\in\mathrm{Mat}_{(n+1)\times m}(\mathbb{C})\) 
such that they fit into the short exact sequence
\begin{equation}
\begin{tikzcd}[column sep=2em]
	&0 \ar[r] &\mathbb{C}^{n+1}\ar[r,"\mathfrak{A}"] &\mathbb{C}^m \ar[r,"\mathfrak{B}"] &\mathbb{C}^{n+1} \ar[r] &0.
\end{tikzcd}
\end{equation}

Let $[y_1\mathpunct{:}\ldots\mathpunct{:}y_m]$ 
be coordinates on $\mathbb{P}^{m-1}$. Each row of \(\mathfrak{B}\)
defines the equation
\begin{equation}
\label{equation:complete-intersection-equations-projective-space}
	b_{i1} y_1^2+\cdots+b_{im}y_m^2=0,~1\le i\le n+1,
\end{equation}
in \(\mathbb{P}^{m-1}\).
Let $Y'\subset\mathbb{P}^{m-1}$ be the subvariety defined by 
\eqref{equation:complete-intersection-equations-projective-space}.
\begin{lemma}
Assume that $\mathfrak{A}$ is in general position. Then 
\(Y'\) is a smooth complete intersection in $\mathbb{P}^{m-1}$
whose canonical bundle is trivial.
\end{lemma}
\begin{proof}
It suffices to show that the Jacobian matrix $(2b_{ij}y_j)$ is of maximal 
rank; namely $(n+1)$. 
We observe that, under our hypothesis on $\mathfrak{A}$,
every $(n+1)$-by-$(n+1)$ submatrix
of $\mathfrak{B}$ is of full rank. Otherwise, after rearranging the
columns, we may assume the submatrix consisting
of the first $(n+1)$ columns of \(\mathfrak{B}\) is singular. 
Then there exists a non-zero element
\(x=(x_1,\cdots,x_{n+1},0,\ldots,0)
\in\mathrm{ker}(\mathfrak{B})\subset\mathbb{C}^m\). So \(x=\mathfrak{A}(\xi)\) 
for some \(0\ne\xi\in\mathbb{C}^{n+1}\).
But this means the submatrix 
consisting of the last $(n+1)$ rows in \(\mathfrak{A}\) is singular.
We get a contradiction.

Let $a=(a_1,\ldots,a_m)\in Y'$ be non-zero. 
Then $(a_1^2,\ldots,a_m^2)$ belongs 
to $\mathrm{ker}(\mathfrak{B})=\mathrm{im}(\mathfrak{A})$. 
Now $\mathfrak{A}$ is in a general position
implies that at most $n$ coordinates in $a$ can be zero.
Namely at least $m-n=n+2$ coordinates in $a$ are non-vanishing.
Choose any $(n+1)$ from them and let $J$ denote
the corresponding index subset. 
The submatrix $(2b_{ij}y_j)_{1\le i\le n+1,~j\in J}$ has
rank $(n+1)$ and hence the result follows since
every \((n+1)\times (n+1)\) minor in \(\mathfrak{B}\)
is invertible.
The triviality of the canonical bundle follows from adjunction formula.
\end{proof}
The matrix \(\mathfrak{A}\) defines an embedding
$\mathbb{P}^{n}\to\mathbb{P}^{m-1}$. To save the notation,
the embedding will be also denoted by \(\mathfrak{A}\). 
We have a (branched) covering map
\begin{equation}
\Phi\colon\mathbb{P}^{m-1}\to\mathbb{P}^{m-1},~
[y_1\mathpunct{:}\ldots\mathpunct{:}y_m]
\mapsto [y_1^2\mathpunct{:}\ldots\mathpunct{:}y_m^2].
\end{equation}
Consider the diagram
\begin{equation}
\begin{tikzcd}
	&  &\mathbb{P}^{m-1}\ar[d,"\Phi"]\\
	& \mathbb{P}^n \ar[r,"\mathfrak{A}"] &\mathbb{P}^{m-1}
\end{tikzcd}
\end{equation}
The map \(\Phi\) realizes \(Y'\) as a Kummer cover over
$\operatorname{Im}(\mathfrak{A})$ 
branched over the image of $\sum_{i=1}^{m}H_{i}$
under \(\mathfrak{A}\). In fact, \(Y'\) fits the fiber product diagram.
\begin{equation}
\begin{tikzcd}
    &Y'\ar[r]\ar[d]  &\mathbb{P}^{m-1}\ar[d,"\Phi"]\\
    & \mathbb{P}^n \ar[r,"\mathfrak{A}"] &\mathbb{P}^{m-1}
\end{tikzcd}
\end{equation}

Put \(\boldsymbol{\mu}_{2}:=\mathbb{Z}\slash 2\mathbb{Z}\).
We define an action of \(\boldsymbol{\mu}_{2}^{2n+2}\) on \(\mathbb{P}^{2n+1}\) by
\begin{equation*}
g\cdot [y_1\mathpunct{:}\ldots\mathpunct{:}y_{2n+2}]=
\left[(-1)^{g_1}y_1\mathpunct{:}\ldots\mathpunct{:}(-1)^{g_{2n+2}}y_{2n+2}\right],~
g=(g_{1},\ldots,g_{2n+2})\in\boldsymbol{\mu}_{2}^{2n+2}.
\end{equation*}
Notice that the diagonal subgroup acts trivially. Let \(G\)
be the cokernel of the diagonal embedding \(\boldsymbol{\mu}_{2}\to 
\boldsymbol{\mu}_{2}^{2n+2}\).
Then $G$ is the Galois group for the
Kummer cover $Y'\to \mathfrak{A}(\mathbb{P}^{n})$.
Moreover, the map \(\boldsymbol{\mu}_{2}^{2n+2}\to \boldsymbol{\mu}_{2}\) given by
\((g_{1},\ldots,g_{2n+2})\mapsto \sum_{i=1}^{2n+2} g_i\) factors through \(G\).
Let \(G'\) be the kernel of the induced map \(G\to \boldsymbol{\mu}_{2}\); in other words,
\begin{equation*}
G' = \left\{(g_{1},\ldots,g_{2n+2})\in \boldsymbol{\mu}_{2}^{2n+2}
\colon \sum g_{j}\equiv 0\pmod{2}\right\}\slash \boldsymbol{\mu}_{2}
\end{equation*}
where \(\boldsymbol{\mu}_{2}\) is the diagonal subgroup.

\begin{lemma}
We have \(Y\simeq Y'\slash G'\). Hence there exists an isomorphism
of pure polarized $\mathbb{Q}$-Hodge structures
\begin{equation}
\mathrm{H}^{q}(Y,\mathbb{Q})\simeq 
\mathrm{H}^{q}(Y',\mathbb{Q})^{G'}.
\end{equation}
\end{lemma}
\begin{proof}
Since both $Y$ and $Y'\slash G'$ are double cyclic covers
over $\mathbb{P}^n$ branched over $\sum_{i=1}^{2n+2}H_{i}$
and the Picard group of $\mathbb{P}^n$
is torsion free, $Y$ and $Y'\slash G'$ must be isomorphic.
The rest of the statement follows from 
\cite{1957-Grothendieck-sur-quelques-points-dalgebre-homologique}*{Proposition 5.2.4}.
\end{proof}

Specializing to \(n=3\), we see that \(Y'=\mathbb{P}^{7}[2,2,2,2]\) and 
our singular double cover \(Y\) is isomorphic to \(Y'\slash G'\) where
\(G'\) is an abelian group of order \(64\) with exponent \(2\).

Notice that \(Y'\slash G'\) can be regarded as
a complete intersection in \(\mathbb{P}^{7}\slash G'\)
which is a toric variety.
The instantons can be computed by applying the orbifold quantum hyperplane section
theorem developed in
\cite{2010-Tseng-orbifold-quantum-riemann-roch-lefschetz-and-serre}*{Theorem 5.2.3}.
We will prove the following result in the rest of this section.
\begin{theorem}
\label{thm:main-theorem}
The equation \eqref{eq:a-model-predicted-correlations}
is the generating series of the genus zero orbifold Gromov--Witten invariants of \(Y\) 
with all insertions \(H\), where \(H\) is the pullback of the hyperplane class
of \(X\). 
\end{theorem}

\subsubsection{The toric varieties \(\mathbb{P}^{7}\slash G\)
and \(\mathbb{P}^{7}\slash G'\)}

Let \(N=\mathbb{Z}^{7}\). Let \(u_{i}=(\delta_{1,i},\ldots,\delta_{7,i})\), \(1\le i\le 7\), and 
\(u_{8}=(-1,\ldots,-1)\in N\). For each \(i\), we put
\begin{equation*}
\sigma_{i}:=\mathrm{Cone}(u_{1},\ldots,\hat{u}_{i},\ldots,u_{8})\subset N_{\mathbb{R}}.
\end{equation*}
The fan \(\Sigma\) consisting of \(\sigma_{i}\), \(1\le i\le 8\), 
together with all their faces,
defines the toric variety \(\mathbb{P}^{7}\). 
Let \(N''=(2\mathbb{Z})^{7}\subset N\) be a sublattice. Note 
\begin{equation}
    [N:N'']=2^{7}=128.     
\end{equation} 
We can also regard \(\Sigma\)
as a fan in \(N''_{\mathbb{R}}\) rather than in \(N_{\mathbb{R}}\).
In this way, we obtained a toric morphism
\begin{equation*}
\Phi\colon X_{\Sigma,N''}\to X_{\Sigma,N}.
\end{equation*}
Moreover, the Galois group \(G\simeq N\slash N''\) and the map \(\Phi\) gives
rise to an isomorphism \(X_{\Sigma,N''}\slash G\simeq X_{\Sigma,N}\).
We shall remind the reader that \(X_{\Sigma,N}\cong X_{\Sigma,N''}\cong\mathbb{P}^{7}\)
and the map \(\Phi\) is indeed the ``coordinate squaring map.''

Consider another sublattice of \(N\)
\begin{equation}
\label{eq:lattice-inclusions}
N' = \left\{(a_{1},\ldots,a_{7})\in N \Bigm\vert \sum_{i=1}^{7}a_{i}\equiv 0\pmod{2}\right\}.
\end{equation}
Notice that we have inclusions 
\begin{equation}
N''\subset N'\subset N~\mbox{and}~[N:N']=2.
\end{equation} 
Let us explicitly write down an integral basis. Put
\begin{align}
\label{eq:basis-n'}
\begin{split}
v_{1}&=(1,1,0,0,0,0,0)\\
v_{2}&=(0,1,1,0,0,0,0)\\
v_{3}&=(0,0,1,1,0,0,0)\\
v_{4}&=(0,0,0,1,1,0,0)\\
v_{5}&=(0,0,0,0,1,1,0)\\
v_{6}&=(0,0,0,0,0,1,1)\\
v_{7}&=(1,0,0,0,0,0,1).
\end{split}
\end{align}
We see that \(\{v_{1},\ldots,v_{7}\}\) forms an integral basis of \(N'\).
The primitive generator of \(\mathbb{R}_{\ge0} u_{i}\) in \(N'\)
is given by \(2u_{i}\) (rather than \(u_{i}\)) which will be denoted by \(\rho_{i}\) later on.
Under this basis, we can re-write
\begin{align}
\label{eq:one-cones-p3-eight-planes}
\begin{split}
\rho_{1}:=2u_{1}&=(1,-1,1,-1,1,-1,1)\\
\rho_{2}:=2u_{2}&=(1,1,-1,1,-1,1,-1)\\
\rho_{3}:=2u_{3}&=(-1,1,1,-1,1,-1,1)\\
\rho_{4}:=2u_{4}&=(1,-1,1,1,-1,1,-1)\\
\rho_{5}:=2u_{5}&=(-1,1,-1,1,1,-1,1)\\
\rho_{6}:=2u_{6}&=(1,-1,1,-1,1,1,-1)\\
\rho_{7}:=2u_{7}&=(-1,1,-1,1,-1,1,1)\\
\rho_{8}:=2u_{8}&=(-1,-1,-1,-1,-1,-1,-1).
\end{split}
\end{align}

Let us look at their dual lattices. For convenience, we shall identify \(M''\)
with the lattice of ``half-integral'' points
as a subset in \(M_{\mathbb{Q}}\)
\begin{equation}
M''=\left\{\left(\frac{a_1}{2},\ldots,\frac{a_{7}}{2}\right)~\Bigm\vert~a_{i}\in\mathbb{Z}\right\}
\subset M_{\mathbb{Q}}.
\end{equation}
and \(M\) is a sublattice in \(M''\) corresponding to the integral points.
\begin{lemma}
Let \(M'\) be the dual lattice of \(N'\) in \(M_{\mathbb{Q}}\). We have 
\begin{equation*}
M'=\left\{\left(\frac{a_1}{2},\ldots,\frac{a_{7}}{2}\right)\in M''~\Big|~
a_{i}\equiv a_{i+1}~(\mathrm{mod}~2),~\forall i=1,\ldots,6\right\}.
\end{equation*}
\end{lemma}
\begin{proof}
Note that \(\displaystyle\left(\frac{a_1}{2},\ldots,\frac{a_{7}}{2}\right)\in M'\)
if and only if 
\begin{equation*}
\sum_{i=1}^{7} b_{i}\frac{a_{i}}{2}\in \mathbb{Z},~\mbox{for all}~(b_{i})\in N'.
\end{equation*}
By plugging the elements in the basis, we see that for \(i=1,\ldots,6\), 
\begin{equation*}
a_{i}\equiv a_{i+1} \pmod{2}
\end{equation*} 
as claimed.
\end{proof}

We can also view \(\Sigma\) as a fan in \(N_{\mathbb{R}}'\). We have
\begin{equation}
    X_{\Sigma,N'}\simeq \mathbb{P}^{7}\slash G'~\mbox{and}~X_{\Sigma,N}\simeq\mathbb{P}^{7}\slash G
\end{equation}
where 
\begin{equation}
    G' = N'\slash N''~\mbox{and}~G=N\slash N''
\end{equation}
as well as a toric map \(X_{\Sigma,N'}\to X_{\Sigma,N}\).
The following proposition shows that 
\(X_{\Sigma,N'}\) is a double cover over \(X_{\Sigma,N}\)
branched along toric divisors.
\begin{proposition}
Let \(S_{i}:=u_{i}^{\perp}\cap M\) and \(S_{i}':=u_{i}^{\perp}\cap M'\).
The inclusion \(M\to M'\) induces the isomorphism \(S_{i}'\simeq S_{i}\).
\end{proposition}
\begin{proof}
Let us do the case \(i=8\). For \(x=(x_{i})\in S_{8}'\), we have \(x\in M'\) and
\begin{equation}
    \sum_{i=1}^{7} x_{i}=0.
\end{equation}
Let us write \(x_{i}=a_{i}\slash 2\). We have \(a_{i}\equiv 0\pmod{2}\); for
otherwise \(\sum_{i=1}^{7} x_{i}\ne 0\). We thus can conclude \(x\in S_{8}\).

For other \(i\), from \(x\in S_{i}'\), we see that \(x_{i}=0\) and hence \(a_{i}\equiv 0\pmod{2}\).
By definition, \(a_{j}\equiv 0\pmod{2}\) for all \(j\). We again conclude \(x\in S_{i}\).
\end{proof}

The group \(G'\) acts on \(Y'\) and we have \(Y'\slash G'\subset \mathbb{P}^{7}\slash G'\).
Therefore, we can regard \(Y\) as a complete intersection in 
a simplicial toric variety \(X_{\Sigma,N'}\). Notice that 
in the lattice \(N'\), the primitive vector of the 1-cone in \(\Sigma\)
is still given by \(u_{i}\), \(1\le i\le 8\).
In what follows, we 
shall focus on the quotient stack \([\mathbb{P}^{7}\slash G']\).
\(X_{\Sigma,N'}\) is nothing but the coarse moduli space of this stack.

We have the following commutative diagram:
\begin{equation}
\label{eq:comm-diag-quotient}
\begin{tikzcd}
X_{\Sigma,N''}\ar[rd,"q"]\ar[dd,"\Phi"] &\\
& X_{\Sigma,N'}\ar[ld,"p"]\\
X_{\Sigma,N} &
\end{tikzcd}
\end{equation}
We now describe the stacky fan for the later use. Consider the exact 
sequence
\begin{equation}
  0\to \mathbb{L}\to \mathbb{Z}^{8}\xrightarrow{\rho} \mathbb{Z}^{7}=:N'
\end{equation}
where \(\rho\) sends \(e_{i}\) to \(\rho_{i}\). 
Denote by \(N''=\operatorname{Im}(\rho)\). Then
\begin{equation}
    [N':N'']=64.
\end{equation}
Note that the basis \eqref{eq:basis-n'}
gives rise to an embedding of \(N'\) into \(N:=\mathbb{Z}^{7}\)
whose image is a sublattice of index \(2\); these together recover
\eqref{eq:lattice-inclusions}.

On the other hand, the kernel 
\(\mathbb{L}\) is identified with
the diagonal subgroup
\begin{equation}
  \langle (1,1,1,1,1,1,1,1)\rangle_{\mathbb{Z}}\subset \mathbb{Z}^{8}.
\end{equation}
Then \((N',\Sigma,\rho\colon\mathbb{Z}^{8}\to N')\)
is the stacky fan describing \(X_{\Sigma,N'}\).
Moreover, the Mori cone \(\overline{\mathrm{NE}}(X_{\Sigma,N'})\subset \mathbb{R}^{8}\) is
\begin{equation}
  \mathbb{R}_{\ge 0}\cdot(1,1,1,1,1,1,1,1)\subset \mathbb{R}^{8}.
\end{equation}

\subsubsection{Genus zero Gromov--Witten invariants for the 
orbifold \([\mathbb{P}^{7}\slash G']\)}
Toric Deligne--Mumford stacks
can be built from either stacky fans \cite{2005-Borisov-Chen-Smith-the-orbifold-chow-ring-of-toric-deligne-mumford-stacks}
or extended stacky fans \cite{2008-Jiang-the-orbifold-cohomology-ring-of-simplicial-toric-stack-bundles}.
Recall that an \emph{extended stacky fan} is a stacky fan 
\(\boldsymbol{\Sigma}=(N,\Sigma,\rho\colon\mathbb{Z}^{n}\to N)\) together with a map
\begin{equation}
     S\to N_{\Sigma}:=\{c\in N\mid \bar{c}\in |\Sigma|\}
\end{equation} 
from a finite
set \(S\). In practice, one often takes \(S\)
to be a subset of \(\mathrm{Box}(\boldsymbol{\Sigma})\).
For a toric Deligne--Mumford stack, the genus zero orbifold Gromov--Witten
invariants are explicitly computed in
\cite{2015-Coates-Corti-Iritani-Tseng-a-mirror-theorem-for-toric-stacks} using
the combinatorial data of extended stacky fans; after
appropriately choosing \(S\to N_{\Sigma}\), one would be able to compute 
genus zero orbifold Gromov--Witten invariants along twisted sectors.
In short, from the combinatorial data, 
one constructs a cohomology-valued series
(a.k.a.~the \(S\)-extended \(I\)-function)
which was shown to 
compute the genus zero orbifold Gromov--Witten invariants.
If \(S=\emptyset\), we then obtain the 
\emph{non-extended \(I\)-function}, which only
determines the genus zero orbifold Gromov--Witten invariants along 
the \emph{very small parameter space}
\begin{equation}
\mathrm{H}^{2}(\bigl\vert[\mathbb{P}^{7}\slash G']\bigr\vert;\mathbb{C})=
\mathrm{H}^{2}(X_{\Sigma,N'};\mathbb{C})\subset
\mathrm{H}^{2}_{\mathrm{CR}}([\mathbb{P}^{7}\slash G'];\mathbb{C}).  
\end{equation}
This will be enough in our following discussion.
Here, the left hand side is the singular cohomology
of the underlying space.

Let come back to our situation. Let \(\Sigma\) and \(N'\) be 
as before.
Note that 
\begin{equation}
    \mathrm{H}^{2}(X_{\Sigma,N'};\mathbb{C})\cong\mathbb{C}
\end{equation}
is one-dimensional and 
\begin{equation}
    \mathrm{H}_{2}(X_{\Sigma,N'},\mathbb{Z})=\mathbb{Z}\langle\ell\rangle
\end{equation}
where \(\ell\) is the curve class coming from a \emph{wall} in \(\Sigma\),
i.e.~an \(6\)-dimensional cone in \(\Sigma\).
We denote by \(D_{i}\) the toric Weil divisor associated 
with the \(1\)-cone \(\mathbb{R}_{\ge}\rho_{i}\);
it is indeed \(\mathbb{Q}\)-Cartier since \(X_{\Sigma,N'}\) is simplicial.
One can easily prove that 
\begin{equation}
    D_{i}\equiv D_{j}~\mbox{and}~D_{i}.\ell = 1/2~\mbox{for all \(i\)}.     
\end{equation}

\begin{remark}
Note that 
\begin{equation}
    \mathrm{H}^{2}(X_{\Sigma,N'};\mathbb{Z})
\end{equation}
has non-trivial torsion part. Also the divisors \(D_{i}\)
are all \emph{inequivalent} under linear equivalence due to the torsions. 
However, we do have
\begin{equation}
    2D_{i}\sim 2D_{j}~\mbox{for all \(i,j\)}
\end{equation}
and each of them is Cartier.
Passing to \(\mathbb{C}\) coefficients (or \(\mathbb{Q}\) coefficients), all the divisors 
\(D_{i}\) will become linearly equivalent. Let us call \(H\) the image of \(D_{i}\)
in \(\mathrm{H}^{2}(X_{\Sigma,N'};\mathbb{C})\).
One notices that \(H\) is not Cartier, but \(2H\) is Cartier.
\end{remark}

The non-extended \(I\)-function is given by 
\begin{align}
\begin{split}
&B_{[\mathbb{P}^{7}\slash G']}(t;\alpha)\\
&=\alpha\cdot \exp({\textstyle H t\slash \alpha})\sum_{g\in \mathrm{C}(G')}
\sum_{d\in\overline{\mathrm{NE}}_{g}}q^{d}\prod_{j=1}^{8}
\frac{\prod_{\langle d \rangle=\langle m\rangle,~m\le 0} (D_{j}+m\alpha)}{\prod_{\langle d \rangle=\langle m\rangle,~m\le d}(D_{j}+m\alpha)}\mathbf{1}_{g}\\
&=\alpha\cdot \exp(Ht\slash \alpha)\sum_{g\in \mathrm{C}(G')}
\sum_{d\in\overline{\mathrm{NE}}_{g}}q^{d}
\frac{1}{\prod_{\langle d \rangle=\langle m\rangle,~0< m\le d}(H+m\alpha)^{8}}\mathbf{1}_{g}.
\end{split}
\end{align}
Some explanations are in order.
\begin{itemize}
    \item \(\alpha\) stands for the formal variable in this expression. 
    It was called \(z\) in other references, especially in 
    \cite{2015-Coates-Corti-Iritani-Tseng-a-mirror-theorem-for-toric-stacks}.
    Since \(z\) was already used for the coordinate on the moduli space, we
    choose to name the formal variable \(\alpha\).
    \item \(\mathrm{C}(G')\) denotes the set of conjugacy classes of \(G'\).
    \item \(\mathbf{1}_{g}\) is the unit in the cohomology ring of the 
    component of the inertia stack associated with \(g\).
    \item For each \(g\in \mathrm{C}(G')\), 
    the relevant Mori cone \(\overline{\mathrm{NE}}_{g}\) is defined by 
    \begin{equation}
        \overline{\mathrm{NE}}_{g}:=\Lambda_{g}\cap\overline{\mathrm{NE}}(X_{\Sigma,N'})
    \end{equation}
    where \(\overline{\mathrm{NE}}(X_{\Sigma,N'})\) is the classical Mori cone
    of the algebraic variety \(X_{\Sigma,N'}\) and
    \begin{equation}
        \Lambda_{g}:=\{\lambda\in\Lambda\mid v(\lambda)=g\}
    \end{equation}
    where \(\Lambda=\cup_{\sigma\in\Sigma} \Lambda_{\sigma}\).
    See also \cite{2015-Coates-Corti-Iritani-Tseng-a-mirror-theorem-for-toric-stacks}*{Remark 30}.
\end{itemize}
In our case, \(G'\) is abelian and \(\mathrm{C}(G')=G'\) can be
identified with \(\mathrm{Box}(\boldsymbol{\Sigma})\).
Also we note that in the present case
\begin{equation}
    \overline{\mathrm{NE}}_{g}\ne\emptyset~\Leftrightarrow~\mbox{\(g\)
    is the identity in \(G'\)}.
\end{equation}
Moreover, if we denote by \(e\in G'\) the identity, then we can check
\begin{equation}
    \overline{\mathrm{NE}}_{e}=\mathbb{Z}_{\ge 0}\cdot(1,1,1,1,1,1,1,1)\subset \mathbb{Z}^{8}.
\end{equation}
Thus in the formula for \(B_{[\mathbb{P}^{7}\slash G']}(t;\alpha)\),
the index \(d\in \overline{\mathrm{NE}}_{e}\) 
can be identified with non-negative integer
and the formula is reduced to
\begin{equation}
B_{[\mathbb{P}^{7}\slash G']}(t;\alpha)=\alpha\cdot \exp(Ht\slash \alpha)
\sum_{d=0}^{\infty}q^{d}
\frac{1}{\prod_{m=1}^{d}(H+m\alpha)^{8}}\mathbf{1}_{e}.
\end{equation}

\subsubsection{A quantum Lefschetz hyperplane theorem}
To apply the machinery developed in 
\cite{2010-Tseng-orbifold-quantum-riemann-roch-lefschetz-and-serre}*{\S 5}, 
we first verify that the generic stabilizer of
our quotient stack \([\mathbb{P}^{7}\slash G']\) is trivial and 
\(Y'\slash G'\) is defined by a section of a split vector bundle coming from the
coarse moduli space \(X_{\Sigma,N'}\).

Note that \(Y'\slash G'\) is  
a complete intersection in \(\mathbb{P}^{7}\slash G'\) as
an intersection of four sections of the line bundle \(2D_{i}\).
Thus the corresponding \emph{hypergeometric modification} of 
\(B_{[\mathbb{P}^{7}\slash G']}\)
is 
\begin{align}
\label{eq:i-function-modification}
B_{[Y'\slash G']}(t;\alpha)
=\alpha\cdot \exp(Ht\slash \alpha)
\sum_{d=0}^{\infty}q^{d}
\frac{\prod_{m=1}^{2d}(2H+m\alpha)^{4}}{\prod_{m=1}^{d}(H+m\alpha)^{8}}\mathbf{1}_{e}.
\end{align}
One should be aware that the series \eqref{eq:i-function-modification}
is almost identical to 
\begin{equation}
\label{eq:i-function-modification-ci}
B_{Y'}(t;\alpha)
=\alpha\cdot \exp(ht\slash \alpha)
\sum_{d=0}^{\infty}q^{d}
\frac{\prod_{m=1}^{2d}(2h+m\alpha)^{4}}
{\prod_{m=1}^{d}(h+m\alpha)^{8}},
\end{equation} 
the \emph{hypergeometric modification} series
for a \((2,2,2,2)\) Calabi--Yau complete intersection \(Y'\)
in \(\mathbb{P}^{7}\). 
Here \(h\) is the hyperplane class of \(\mathbb{P}^{7}\). The
only difference is the hyperplane classes ``\(h\)'' and ``\(H\).''

Due to the very similar looking appearance of
\eqref{eq:i-function-modification} and \eqref{eq:i-function-modification-ci},
the following corollary follows immediately.
\begin{corollary}
The mirror maps for 
\begin{equation}
(2h)^{4}B_{Y'}(t;\alpha)~\mbox{and}~
(2H)^{4}B_{[Y'\slash G']}(t;\alpha)
\end{equation}
are identical
if we treat \(H\) and \(h\) as formal variables such that \(h^{8}=H^{8}=0\).
\end{corollary}

Next we investigate the ordinary Poincar\'{e} pairing on \(Y'\)
and the orbifold Poincar\'{e} pairing
\(Y'\slash G'\). Let \(h\) and \(H\) be the hyperplane classes on \(\mathbb{P}^{7}\)
and \(\mathbb{P}^{7}\slash G'\) respectively.
By abuse of notation, we shall use the same notation to denote 
the restriction of \(h\) and \(H\)
to \(Y'\) and \(Y'\slash G'\).
From \eqref{eq:comm-diag-quotient}, we see that \(H=p^{\ast}h\) and
\(q^{\ast}H=2h\) (note that \(\Phi^{\ast}h=2h\)). Therefore,
\begin{equation*}
\int_{\left|[Y'\slash G']\right|} H^{3}=\frac{1}{64}\int_{Y'} (2h)^{3}=
\frac{1}{8}\int_{Y'}h^{3}=\frac{1}{8}\int_{\mathbb{P}^{7}} (2h)^{4} h^{3}=2.
\end{equation*}
We see that 
\begin{equation}
\left\{\mathbf{1},H,\frac{H^{2}}{2},\frac{H^{3}}{2}\right\}
\end{equation}
is a symplectic basis
of \(\mathrm{H}^{\bullet}(\bigl\vert[Y'\slash G']\bigr\vert;\mathbb{C})\) with
respect to the orbifold Poincar\'{e} pairing on the coarse moduli 
\(\bigl\vert[Y'\slash G']\bigr\vert\) (the untwisted sector). 

On the other hand, we know that 
\begin{equation}
    \left\{\mathbf{1},h,\frac{h^{2}}{16},\frac{h^{3}}{16}\right\} 
\end{equation}
is a symplectic basis 
of \(\mathrm{H}^{\bullet}(Y';\mathbb{C})\) with respect to the 
ordinary Poincar\'{e} pairing on \(Y'\).
Since \(Y'\) and \(Y'\slash G'\) have identical \(1\)-point 
invariants with insertions from corresponding coarse moduli spaces and 
descendants (a.k.a~the \(J\)-function), we obtain
the following proposition.
\begin{proposition}
\label{prop:1/8}
We have for any \(k=0,\ldots,3\),
\begin{equation*}
\frac{1}{8}\left\langle\frac{h^{k}}{z-\psi},\mathbf{1}\right\rangle_{0,2,d}^{Y'}=
\left\langle\frac{H^{k}}{z-\psi},\mathbf{1}\right\rangle_{0,2,d}^{[Y'\slash G']}.
\end{equation*}
\end{proposition}

\begin{proof}[Proof of Theorem 2.5]
Combining Corollary \ref{cor:prediction-1/8}
and Proposition \ref{prop:1/8}, 
we conclude the proof of Theorem \ref{thm:main-theorem}.
\end{proof}

This proposition explains the peculiar factor \(1/8\)
in Corollary \ref{cor:prediction-1/8}, and hence also
implies the following theorem.
\begin{theorem}
\label{thm:main-theorem-hhhh}
The series \eqref{eq:a-model-predicted-correlations} computes 
the genus zero untwisted orbifold Gromov--Witten invariants of \(Y\), whereby
all the insertions are pullback of cohomology classes
of the base \(\mathbb{P}^{3}\).
\end{theorem}

\section{Double covers of \texorpdfstring{\(\mathbb{P}^{3}\)}{P3}
with the nef-partition \texorpdfstring{\(-K_{\mathbb{P}^{3}}=4h\)}{-K=4h}}
\label{sec:4h}
In the section, we investigate another toy example; this can be generalized
to arbitrary toric bases.
Let again \(\Delta\) be the convex hull of 
\begin{align*}
(3,-1,-1),~(-1,3,-1),~(-1,-1,3),~(-1,-1,-1).
\end{align*}
Let \(X=\mathbf{P}_{\Delta}=\mathbb{P}^{3}\) and \(H\) be the hyperplane class.
Consider the following situation.
\begin{itemize}
    \item Regard \(\Delta=\Delta\) as the 
    Minskowski sum decomposition representing the nef-partition \(-K_{X}=4h\).
    \item The Batyrev--Borisov dual nef-partition \(\nabla\); it
    is just the dual polytope of \(\Delta\). 
\end{itemize}
Let \(X^{\vee}\to\mathbf{P}_{\nabla}\) be a MPCP desingularization
which turns out again to be smooth in the present case.
Let \(\mathcal{Y}\to V\) and \(\mathcal{Y}^{\vee}\to U\) be the families of 
Calabi--Yau double covers over \(X\) and \(X^{\vee}\) constructed in 
\S\ref{subsection:cy-double-covers} respectively.
Let \(Y\) and \(Y^{\vee}\) be the fiber of 
\(\mathcal{Y}\to V\) and \(\mathcal{Y}^{\vee}\to U\).
Notice that we have \(h^{1,1}(Y)=h^{2,1}(Y^{\vee})=1\).
In what follows, we shall drop the subscript (\(r=1\)) for the nef-partition.

In the present case, on the \(X\) side, the primitive
generators of the \(1\)-cones in the fan defining \(X\) are given by
\begin{eqnarray*}
\rho_{1}=(1,0,0),~\rho_{2}=(0,1,0),~\rho_{3}=(0,0,1)~\mbox{and}~\rho_{4}=(-1,-1,-1).
\end{eqnarray*}

\subsection{Picard--Fuchs equations for \texorpdfstring{\(\mathcal{Y}^{\vee}\to U\)}{}}
\label{subsection:pf-equations-cyclic-cy-3-fold-p11114}

From the construction, 
the integral points in the section polytopes of \(F\) 
correspond to the integral points in \(\nabla\).
The GKZ hypergeometric system associated with \(\mathcal{Y}^{\vee}\to U\) is given by 
\begin{equation*}
A_{\mathrm{ext}}=\begin{bmatrix}
1 & 1 & 1 & 1 & 1 \\
0 & 1 & 0 & 0 &-1 \\
0 & 0 & 1 & 0 &-1 \\
0 & 0 & 0 & 1 &-1 \\
\end{bmatrix}~\mbox{and}~
\beta=
\begin{bmatrix}
-1/2\\
0\\
0\\
0\\
\end{bmatrix}
\end{equation*}
The lattice relation given by \( A \) is \(L_{\mathrm{ext}}=
\langle \ell\rangle_{\mathbb{Z}}\) with \(\ell:=(-4,1,1,1,1)\).

From the lattice relation, the box operators is
\begin{equation*}
\label{equation:GKZ-polynomial-operators-P3-mirror-4h}
\Box_{k\ell}=\partial_{x_{1}}^{k}\partial_{x_{2}}^{k}
\partial_{x_{3}}^{k}\partial_{x_{4}}^{k}-
\partial_{x_{0}}^{4k},~
k\in\mathbb{Z}_{\ge0}
\end{equation*}
or with a minus sign if \(k<0\).
Let us consider the case \(k=1\). We have
\begin{align}
\begin{split}
&\textstyle x_{0}^{1/2}(\textstyle\prod_{i=1}^{4}x_{i})
\Box_{\ell}x_{0}^{-1/2}\\
&=\prod_{i=1}^{4}\theta_{x_{i}}
-z\prod_{i=1}^{4}\left(\theta_{x_{0}}-\frac{2i-1}{2}\right)
\end{split}
\end{align}
where \(z=(x_{1}x_{2}x_{3}x_{4})\slash x_{0}^{4}\).
Here \(\theta_{a}=a(\mathrm{d}\slash\mathrm{d}a)\) 
is the logarithmic derivative with respect to \(a\).
Substituting
\begin{equation*}
\theta_{x_{i}}=\theta_{z},~\theta_{x_{0}}=-4\theta_{z}
\end{equation*}
we see that \eqref{equation:GKZ-polynomial-operators-P3-mirror-4h} is transformed into
\begin{equation}
\label{equation:picard-fuchs-equation-for-mirror-p3-cover-4h}
\theta_{z}^{4} - 4^4 z \prod_{i=1}^{4}\left(\theta_{z}+\frac{2i-1}{8}\right).
\end{equation}

The unique holomorphic series solution to 
\eqref{equation:picard-fuchs-equation-for-mirror-p3-cover-4h} is then of the form
\begin{equation}
\label{equation:picard-fuchs-equation-series-sol-4h} 
\sum_{n\ge 0} \frac{\Gamma(4n+1/2)^{4}}{\Gamma(1/2)^{4}\Gamma(n+1)^{4}}z^{n}.
\end{equation}

\begin{remark}
\label{remark:8-in-p11114-coodinate}
The equation \eqref{equation:picard-fuchs-equation-for-mirror-p3-cover-4h} has also 
been studied in the literature.
Introducing a change of variables \( w=z/256 \), we have \( \theta_{w}=\theta_{z} \) and 
\begin{equation}
\label{equation:picard-fuchs-equation-for-mirror-2-2-2-2-complete-intersection-4h}
\eqref{equation:picard-fuchs-equation-for-mirror-p3-cover-4h} = 
\theta_{w}^{4}-65536w\prod_{i=1}^{4}\left(\theta_{w}+\frac{2i-1}{8}\right),
\end{equation}
which is the Picard--Fuchs equation for the mirror of 
the family of anti-canonical Calabi--Yau hypersurfaces in \(\mathbb{P}^{4}(1,1,1,1,4)\).
\end{remark}

\subsection{An instanton prediction from mirror symmetry}
We adapt the notation in \S\ref{subsection:inst-prediction-h-h-h-h}.
Similarly, using the equation
\begin{equation*}
2\theta_{z}\left\langle \theta_{z},\theta_{z},\theta_{z}\right\rangle^{\Omega}
-\int_{\tilde{Y}^{\vee}} \Omega(z)\wedge \theta_{z}^{4}\Omega(z)=0.
\end{equation*}
and substituting the last term by the Picard--Fuchs equation 
\eqref{equation:picard-fuchs-equation-for-mirror-2-2-2-2-complete-intersection-4h},
we get
\begin{equation*}
\theta_{z}\left\langle \theta_{z},\theta_{z},\theta_{z}\right\rangle^{\Omega}
=\frac{256z}{1-256z}\left\langle \theta_{z},\theta_{z},\theta_{z}\right\rangle^{\Omega}.
\end{equation*}
We can solve the above equation and get
\begin{equation*}
\left\langle \theta_{z},\theta_{z},\theta_{z}\right\rangle^{\Omega} = \frac{C}{1-256z},
\end{equation*}
for some constant \(C\). One can check 
the \emph{normalized Yukawa coupling}
\begin{equation*}
\left\langle \theta_{z},\theta_{z},\theta_{z}\right\rangle
:=\int_{\tilde{Y}^{\vee}} \frac{\Omega(z)}{y_{0}(z)}\wedge \theta_{z}^{3}
\left(\frac{\Omega(z)}{y_{0}(z)}\right)
\end{equation*}
is given by
\begin{equation}
\label{equation:normal-yukawa-coupling-4h}
\left\langle \theta_{z},\theta_{z},\theta_{z}\right\rangle = 
\frac{C}{(1-256z)y_{0}(z)^{2}},
\end{equation}
where \(y_{0}(z)\) is the holomorphic series solution 
\eqref{equation:picard-fuchs-equation-series-sol-4h}.

Now we compute the ``mirror map.''
Consider the deformed series
\begin{equation*}
\label{eq:gamma-hol-series-deformed-p11114}
y_0(z;\rho):=\sum_{n\ge 0} 
\frac{\Gamma(4n+4\rho+1/2)^4}{\Gamma(1/2)^{4}\Gamma(n+\rho+1)^{4}} z^{n+\rho} 
\end{equation*}
and its derivative with respect to \(\rho\)
\begin{equation*}
y_{1}(z):=\left.\frac{\mathrm{d}}{\mathrm{d}\rho}\right|_{\rho=0} y_{0}(z;\rho).
\end{equation*}
Consequently, the ``mirror map'' is given by
\begin{equation}
\label{equation:mirrir-map-4h}
q = \exp\left(2\pi\sqrt{-1} t\right),~t = \frac{1}{2\pi\sqrt{-1}}\frac{y_{1}(z)}{y_{0}(z)}.
\end{equation}
Using the classical product, one finds \( C = 2 \) in \eqref{equation:normal-yukawa-coupling-4h}.
In the present case, the mirror map is
\begin{equation*}
q = \frac{z}{256} + \frac{247z^{2}}{1024} + \frac{13368541z^{3}}{524288}+\cdots,
\end{equation*}
and the inverse is given by 
\begin{equation*}
z = 256 q - 4046848 q^2 + 18282602496 q^3 - + \cdots.
\end{equation*}
The \(A\)-model correlation function is
\begin{align}
\begin{split}
\label{eq:a-model-predicted-correlations-4h}
&\langle H,H,H\rangle(q) \\
&= 2 + 29504 q + 1030708800 q^2 + 38440454795264 q^3 + \cdots.
\end{split}
\end{align}
Note that for the classical pairing, we take 
the ample generator of 
\(\mathrm{H}^{2}(Y;\mathbb{Z})\cong\mathbb{Z}\)
which is the pullback of the very ample divisor
\(H\) on \(\mathbf{P}^{3}\).

Consequently, we obtain the following numerical result.
\begin{corollary}
The predicted instanton numbers \(n_{d}\) of \(Y\) for small \(d\) are given by
\begin{align*}
\begin{split}
n_{1} = 29504,~n_{2} = 128834912,~ n_{3} = 1423720546880.
\end{split}
\end{align*}
\end{corollary}
\begin{corollary}
\label{cor:prediction-2}
The predicted instanton numbers \(n_{d}\) of \(Y\) and those
of degree \(8\) hypersurface in the weighted projective space
\(\mathbb{P}(1,1,1,1,4)\) are the same.
\end{corollary}

\subsection{An instanton calculation}
\label{subsec:an_instanton_calculation}
In this subsection, we will compute the genus zero orbifold Gromov--Witten invariants of \(Y\) with insertions from the untwisted sector
through a pre-quotient model \(Y'\) as we did 
in the previous section. 
As we will see, \(Y'\) is a degree \(8\) hypersurface in 
the weighted projective space \(\mathbb{P}(1,1,1,1,4)\).

Let \([z_1\mathpunct{:}\ldots\mathpunct{:}z_4]\) be the homogeneous 
coordinates on \(X=\mathbb{P}^{3}\) and
\(f\) be a degree \(4\) polynomial in \(\mathbb{P}^{3}\)
such that \(\{f=0\}\cup\bigcup_{i=1}^{4}\{z_{i}=0\}\) 
is the branched locus of the double cover \(Y\to X\).
Consider the graph embedding morphism
\begin{equation*}
\Gamma_{f}\colon X\to\mathbb{P}(1,1,1,1,4),~
[z_1\mathpunct{:}\ldots\mathpunct{:}z_4]\mapsto
[z_1\mathpunct{:}\ldots\mathpunct{:}z_4\mathpunct{:}f(z)].
\end{equation*}
This is well-defined since \(f\) is of degree \(4\).
Let \([y_1\mathpunct{:}\ldots\mathpunct{:}y_5]\)
be the homogeneous coordinate on \(\mathbb{P}(1,1,1,1,4)\).
We have a (branched) covering map
\begin{equation}
\Phi\colon\mathbb{P}(1,1,1,1,4)\to\mathbb{P}(1,1,1,1,4),~
[y_1\mathpunct{:}\ldots\mathpunct{:}y_5]
\mapsto [y_1^2\mathpunct{:}\ldots\mathpunct{:}y_5^2].
\end{equation}
Let \(Y'\subset\mathbb{P}(1,1,1,1,4)\)
be the subvariety defined by the equation \(y_{5}^{2}-f(y_{1}^{2},\ldots,y_{4}^{2})\).
It is clear that \(Y'\) is a Calabi--Yau hypersurface. Moreover, we have
\begin{lemma}
\(Y'\) is smooth.
\end{lemma}
Look at the diagram
\begin{equation}
\begin{tikzcd}
	&  &\mathbb{P}(1,1,1,1,4)\ar[d,"\Phi"]\\
	& X \ar[r,"\Gamma_{f}"] &\mathbb{P}(1,1,1,1,4)
\end{tikzcd}
\end{equation}
The map \(\Phi\) realizes \(Y'\) as a Kummer cover over
\(\Gamma_{f}(X)\) branched along
\begin{equation*}
\Gamma_{f}\left(\{f=0\}\cup\bigcup_{i=1}^{4}\{z_{i}=0\}\right).
\end{equation*}
We define an action of \(\boldsymbol{\mu}_{2}^{5}\) on 
\(\mathbb{P}(1,1,1,1,4)\) by
\begin{equation*}
g\cdot [y_1\mathpunct{:}\ldots\mathpunct{:}y_{5}]=
\left[(-1)^{g_1}y_1\mathpunct{:}\ldots\mathpunct{:}(-1)^{g_{5}}y_{5}\right]~
\mbox{where}~g=(g_{1},\ldots,g_{5})\in\boldsymbol{\mu}_{2}^{5}.
\end{equation*}
Notice that the subgroup \(K:=\langle(1,1,1,1,0)\rangle\subset\mu_{2}^{5}\)
acts trivially on \(\mathbb{P}(1,1,1,1,4)\). Let \(G=\boldsymbol{\mu}_{2}^{5}\slash K\).
Then $G$ is the Galois group for the
Kummer cover $Y'\to \mathfrak{A}(X)$.
Moreover, the map \(\boldsymbol{\mu}_{2}^{5}\to \boldsymbol{\mu}_{2}\) given by
\begin{equation*}
(g_{1},\ldots,g_{5})\mapsto \sum_{i=1}^{4} g_i
\end{equation*} 
factors through \(G\).
Let \(G'\) be the kernel of the induced map \(G\to \boldsymbol{\mu}_{2}\); in other words,
\begin{equation*}
G' = \left\{(g_{1},\ldots,g_{5})\in \boldsymbol{\mu}_{2}^{5}
\colon \sum_{j=1}^{4} g_{j}\equiv 0\pmod{2}\right\}
\Big\slash K.
\end{equation*}

\begin{lemma}
\(\mathbb{P}(1,1,1,1,4)\slash G'\to \mathbb{P}(1,1,1,1,4)\) is a double cover
branched along the union of all toric divisors.
\end{lemma}
\begin{proof}
Let \(N=\mathbb{Z}^{4}\) and \(N''=(2\mathbb{Z})^{4}\subset N\)
be the sublattice of index \(16\). Let \(\Sigma\) be a fan in \(N_{\mathbb{R}}\)
defining \(\mathbb{P}(1,1,1,1,4)\); the primitive generators
of the \(1\)-cones in \(\Sigma\) in \(N''\) are 
given by \(\{2u_{1},2u_{2},2u_{3},2u_{4},2u_{5}\}\) where
\begin{align}
\label{eq:all-1-cone-p11114}
\begin{split}
u_{1}&=(1,0,0,0)\\
u_{2}&=(0,1,0,0)\\
u_{3}&=(0,0,1,0)\\
u_{4}&=(0,0,0,1)\\
u_{5}&=(-1,-1,-1,-4).
\end{split}
\end{align}
Now let \(N':=\{(a_{1},\ldots,a_{4})\in N\mid 
\sum a_{i}\equiv 0\pmod{2}\}\).
One can check that
\begin{align}
\begin{split}
v_{1} &= (1,-1,0,0)\\
v_{2} &= (0,1,1,0)\\
v_{3} &= (0,0,1,1)\\
v_{4} &= (1,0,0,1)
\end{split}
\end{align}
form a basis for \(N'\). Note that the presence
of the \(-1\) in \(v_{1}\). This is slightly different
from the previous case mainly because the dimension
of \(\mathbb{P}(1,1,1,1,4)\) is even; the vectors
\begin{align}
\begin{split}
(1,1,0,0)\\
(0,1,1,0)\\
(0,0,1,1)\\
(1,0,0,1)
\end{split}
\end{align}
will not be linearly independent.

It is also easy to check that
under this basis, we have
\begin{align}
\begin{split}
\rho_{1}=2u_{1}&=(1,1,-1,1)\\
\rho_{2}=2u_{2}&=(-1,1,-1,1)\\
\rho_{3}=2u_{3}&=(1,1,1,-1)\\
\rho_{4}=2u_{4}&=(-1,-1,1,1)\\
\rho_{5}=2u_{5}&=(3,1,-3,-5).
\end{split}
\end{align}
Again we have the inclusion relations
\(N''\subset N'\subset N\) and
\begin{equation}
    [N:N']=8,~\mbox{and}~[N':N'']=2.
\end{equation}


We identify \(M''\)
with the lattice of ``half-integral'' points in \(M_{\mathbb{Q}}\)
\begin{equation*}
\left\{\left(\frac{a_1}{2},\ldots,\frac{a_{4}}{2}\right)~\Bigm\vert~a_{i}\in\mathbb{Z}\right\}
\end{equation*}
and \(M\) is a sublattice in \(M''\) corresponding to 
the set of the integral points.
Then 
\begin{equation*}
M'=\left\{\left(\frac{a_1}{2},\ldots,\frac{a_{4}}{2}\right)\in M''~\Big|~
a_{i}\equiv a_{i+1}~(\mathrm{mod}~2),~\forall i=1,\ldots,3\right\}.
\end{equation*}
One easily sees that \(\rho_{i}^{\perp}\cap M = \rho_{i}^{\perp}\cap M'\)
for all \(i\). This implies \(X_{\Sigma,N'}\to X_{\Sigma,N}\) is a 
double cover branched along the union of all toric divisors.
\end{proof}

As before, one can show that
\begin{lemma}
We have \(Y\simeq Y'\slash G'\). Hence there exists an isomorphism
of pure polarized $\mathbb{Q}$-Hodge structures
\begin{equation}
\mathrm{H}^{q}(Y,\mathbb{Q})\simeq 
\mathrm{H}^{q}(Y',\mathbb{Q})^{G'}.
\end{equation}
\end{lemma}

Notice that \(Y'\slash G'\) can be regarded as
a Calabi--Yau hypersurface in \(\mathbb{P}(1,1,1,1,4)\slash G'\).
The instantons can be computed by applying the orbifold quantum hyperplane section
theorem developed in
\cite{2010-Tseng-orbifold-quantum-riemann-roch-lefschetz-and-serre}*{Theorem 5.2.3}.
We will prove the following result in the rest of this section.
\begin{theorem}
\label{thm:main-theorem-4h}
The equation \eqref{eq:a-model-predicted-correlations-4h}
is the generating series of the genus zero untwisted orbifold Gromov--Witten invariants of \(Y\), whereby all the insertions are pullback of cohomology classes
of the base \(\mathbb{P}^{3}\).
\end{theorem}

From now on, for simplicity, we put \(Z'=\mathbb{P}(1,1,1,1,4)\).

\subsubsection{Toric varieties \(Z'\slash G\)
and \(Z'\slash G'\)}

Recall that the group \(G'\) acts on \(Y'\) and we have \(Y'\slash G'\subset Z'\slash G'\).
In the present case, \(Z'=X_{\Sigma,N''}\) and \(Z'\slash G=X_{\Sigma,N}\).
The map \(\Phi\) realizes \(X_{\Sigma,N}\) 
as a quotient \(X_{\Sigma,N''}\slash G\). 
Therefore, we can regard \(Y\) as a Calabi--Yau hypersurface in 
a simplicial toric variety \(X_{\Sigma,N'}\). 

We will be focusing on the quotient stack \([Z'\slash G']\)
and calculate the
genus zero orbifold Gromov--Witten invariants
of its hypersurfaces.
Notice that \(X_{\Sigma,N'}\) is the 
coarse moduli space of this stack.

These data fits the following commutative diagram:
\begin{equation}
\label{eq:comm-diag-quotient-p11114}
\begin{tikzcd}
X_{\Sigma,N''}=Z'\ar[rd,"q"]\ar[dd,"\Phi"] &\\
& X_{\Sigma,N'}=Z'\slash G'\ar[ld,"p"]\\
X_{\Sigma,N}=Z'\slash G &
\end{tikzcd}
\end{equation}

\subsubsection{Genus zero Gromov--Witten invariants for the 
orbifold \([Z'\slash G']\)}

We denote by \(D_{i}\) the toric Weil divisor associated with the \(1\)-cone 
\(\mathbb{R}_{\ge}\rho_{i}\); they
are indeed \(\mathbb{Q}\)-Cartier since \(X_{\Sigma,N'}\) is simplicial.
In the present case, we have
\begin{equation}
\mathrm{H}^{2}(X_{\Sigma,N'};\mathbb{C})=
\mathbb{C}\cdot H~\mbox{and}~\mathrm{H}_{2}(X_{\Sigma,N'};
\mathbb{Z})=\mathbb{Z}\langle\ell\rangle.
\end{equation}
where \(H\) is the image of 
\(D_{i}\) for \(i\ne 4\) under the map
\begin{equation}
  \mathrm{H}^{2}(X_{\Sigma,N'};\mathbb{Q})\to \mathrm{H}^{2}(X_{\Sigma,N'};\mathbb{C})
\end{equation}
and \(\ell\) is the curve class coming from a \emph{wall} in \(\Sigma\).

One can easily check that 
\begin{equation}
  D_{1}\equiv D_{2}\equiv D_{3}\equiv D_{5}~\mbox{and}~D_{4}\equiv 4D_{5}.
\end{equation}
Also one can compute the intersection number \(D_{1}.\ell = 1/4\). 
Now \(4H\) (and hence \(8H\)) is a Cartier divisor.
The non-extended \(I\)-function is given by 
\begin{align}
\begin{split}
&B_{[Z'\slash G']}(t;\alpha)\\
&=\alpha\cdot \exp({\textstyle H t\slash \alpha})\sum_{g\in \mathrm{C}(G')}
\sum_{d\in\overline{\mathrm{NE}}_{g}}q^{d}\prod_{j=1}^{5}
\frac{\prod_{\langle d \rangle=\langle m\rangle,~m\le 0} (D_{j}+m\alpha)}{\prod_{\langle d \rangle=\langle m\rangle,~m\le d}(D_{j}+m\alpha)}\mathbf{1}_{g}\\
&=\alpha\cdot \exp(Ht\slash \alpha)\sum_{g\in \mathrm{C}(G')}
\sum_{d\in\overline{\mathrm{NE}}_{g}}q^{d}
\frac{1}{\displaystyle\prod_{\substack{\langle d \rangle=\langle m\rangle\\0< m\le d}}(H+m\alpha)^{4}
\prod_{\substack{\langle d \rangle=\langle m\rangle\\0< m\le 4d}}(4H+m\alpha)}\mathbf{1}_{g}.
\end{split}
\end{align}
As before, in the above equation, the notation is as follows.
\begin{itemize}
    \item \(\alpha\) is a formal variable.
    \item \(\mathrm{C}(G')\) is the set of conjugacy classes of \(G'\).
    \item \(\mathbf{1}_{g}\) is the unit in the cohomology ring of the 
    component of the inertia stack associated with \(g\).
    \item For each \(g\in \mathrm{C}(G')\), 
    the relevant Mori cone \(\overline{\mathrm{NE}}_{g}\) is defined by 
    \begin{equation}
        \overline{\mathrm{NE}}_{g}:=\Lambda_{g}\cap\overline{\mathrm{NE}}(X_{\Sigma,N'})
    \end{equation}
    where \(\overline{\mathrm{NE}}(X_{\Sigma,N'})\) is the classical Mori cone
    of the algebraic variety \(X_{\Sigma,N'}\) and
    \begin{equation}
        \Lambda_{g}:=\{\lambda\in\Lambda\mid v(\lambda)=g\}
    \end{equation}
    where \(\Lambda=\cup_{\sigma\in\Sigma} \Lambda_{\sigma}\).
\end{itemize}

Let us work out \(\overline{\mathrm{NE}}_{g}\) in this case.
Recall that there is a relation among \(\rho_{i}\)
\begin{equation}
\label{eq:1-cone-p11114}
    \rho_{1}+\rho_{2}+\rho_{3}+4\rho_{4}+\rho_{5}=\mathbf{0}.
\end{equation}
It is unique up to scaling. In \(\mathbb{P}(1,1,1,1,4)\)
there is only one singular cone; namely the cone
generated by
\begin{align}
u_{1},u_{2},u_{3},u_{5}~\mbox{in}~\eqref{eq:all-1-cone-p11114}.
\end{align}
For any \(g\in \mathrm{C}(G')\equiv\mathrm{Box}(\boldsymbol{\Sigma})\), we claim
\begin{equation}
    \Lambda_{g}\ne\emptyset~\Leftrightarrow~\mbox{\(\displaystyle 
    g=\frac{c}{4}(\rho_{1}+\rho_{2}+\rho_{3}+\rho_{5})\in \mathrm{Box}(\boldsymbol{\Sigma})\) for integers \(0\le c\le 3\)}.
\end{equation}
The relation \eqref{eq:1-cone-p11114} shows that 
``\(\Leftarrow\)'' holds.
For the opposite direction, 
note that if \(g\ne e\), i.e.~the corresponding
element in \(\mathrm{Box}(\boldsymbol{\Sigma})\) is non-zero, then
since there is only one singular cone, namely 
\(\sigma:=\operatorname{Cone}\{\rho_{1},\rho_{2},\rho_{3},\rho_{5}\}\),
we have
\begin{equation}
    g\in \mathrm{Box}(\sigma)
\end{equation}
and the result follows.

\subsubsection{A quantum Lefschetz hyperplane theorem}

\(Y'\slash G'\) is an anti-canonical hypersurface in \(Z'\slash G'\).
The corresponding \emph{hypergeometric modification} of \(B_{[Z'\slash G']}\)
is given by
\begin{align}
\begin{split}
\label{eq:i-function-modification-4h}
&B_{[Y'\slash G']}(t;\alpha)
=\alpha\cdot \exp(Ht\slash \alpha)\\
&\times\sum_{g\in \mathrm{C}(G')}
\sum_{d\in\overline{\mathrm{NE}}_{g}}q^{d}
\frac{\displaystyle\prod_{\substack{\langle d \rangle=\langle m\rangle\\0< m\le 8d}}(8H+m\alpha)}{\displaystyle\prod_{\substack{\langle d \rangle=\langle m\rangle\\0< m\le d}}(H+m\alpha)^{4}
\prod_{\substack{\langle d \rangle=\langle m\rangle\\0< m\le 4d}}(4H+m\alpha)}\mathbf{1}_{g}.
\end{split}
\end{align}
Restricting \eqref{eq:i-function-modification-4h} to the 
\emph{untwisted sector}; namely \(\mathbf{1}_{e}\), we obtain
\begin{align}
\label{eq:i-function-modification-untwisted-4h}
\begin{split}
B^{\mathrm{untw}}_{[Y'\slash G']}(t;\alpha)
&=\alpha\cdot \exp(Ht\slash \alpha)\\
&\times\sum_{d\in\mathbb{Z}_{\ge 0}}q^{d}
\frac{\displaystyle\prod_{m=1}^{8d}(8H+m\alpha)}{\displaystyle\prod_{m=1}^{d}(H+m\alpha)^{4}
\prod_{m=1}^{4d}(4H+m\alpha)}\mathbf{1}_{e}
\end{split}
\end{align}
Again the series \eqref{eq:i-function-modification-untwisted-4h}
is almost identical to 
\begin{equation}
\tilde{B}_{Y'}(t;\alpha)
=\alpha\cdot \exp(ht\slash \alpha)
\sum_{d\in\mathbb{Z}_{\ge 0}}q^{d}
\frac{\displaystyle\prod_{m=1}^{8d}(8h+m\alpha)}{\displaystyle\prod_{m=1}^{d}(h+m\alpha)^{4}
\prod_{m=1}^{4d}(4h+m\alpha)}
\end{equation} 
the \emph{hypergeometric modification} 
for the Calabi--Yau hypersurface \(Y'\) in \(Z'\). 
Here \(h\) is the hyperplane class of \(Z'\). Again note that the
only difference is the hyperplane classes ``\(h\)'' and ``\(H\).''

\begin{corollary}
The mirror maps for \(8h\cdot\tilde{B}_{Y'}\)
and \(8H\cdot\tilde{B}^{\mathrm{untw}}_{[Y'\slash G']}\) are identical
if we treat \(H\) and \(h\) as formal variables such that \(h^{5}=H^{5}=0\).
\end{corollary}

Now we turn to investigate the Poincar\'{e} pairing on \(Y'\)
and the orbifold Poincar\'{e} pairing
\(Y'\slash G'\). Let \(h\) and \(H\) be the hyperplane classes on \(Z'\)
and \(Z'\slash G'\) respectively.
By abuse of notation, we shall use the same notation to denote 
the restriction of \(h\) and \(H\)
to \(Y'\) and \(Y'\slash G'\).
From \eqref{eq:comm-diag-quotient-p11114}, we see that \(H=p^{\ast}h\) and
\(q^{\ast}H=2h\) (note that \(\Phi^{\ast}h=2h\)). Therefore,
\begin{equation*}
\int_{\left|[Y'\slash G']\right|} H^{3}=\frac{1}{8}\int_{Y'} (2h)^{3}=
\int_{Y'}h^{3}=\int_{Z'} 8h^{4}=2.
\end{equation*}
We see that 
\begin{equation}
    \left\{\mathbf{1},H,\frac{H^{2}}{2},\frac{H^{3}}{2}\right\}~\mbox{is a symplectic basis
of}~\mathrm{H}^{\bullet}(\bigl\vert[Y'\slash G']\bigr\vert;\mathbb{C})
\end{equation}
with
respect to the orbifold Poincar\'{e} pairing on the coarse moduli 
\(\left|[Y'\slash G']\right|\). On the other hand, we know that 
\(\{\mathbf{1},h,h^{2}/2,h^{3}/2\}\) is a symplectic basis 
of \(\mathrm{H}^{\bullet}(Y';\mathbb{C})\) with respect to the 
Poincar\'{e} pairing on \(Y'\).
We thus proved the following proposition.
\begin{proposition}
\label{prop:1}
We have for any \(k=0,\ldots,3\),
\begin{equation*}
\left\langle\frac{h^{k}}{z-\psi},\mathbf{1}\right\rangle_{0,2,d}^{Y'}=
\left\langle\frac{H^{k}}{z-\psi},\mathbf{1}\right\rangle_{0,2,d}^{[Y'\slash G']}.
\end{equation*}
\end{proposition}
Let us finish the proof of Theorem \ref{thm:main-theorem-4h}
\begin{proof}[Proof of Theorem \ref{thm:main-theorem-4h}]
Combining Corollary \ref{cor:prediction-2}
and Proposition \ref{prop:1}, 
we conclude the proof of Theorem \ref{thm:main-theorem-4h}.
\end{proof}

\begin{remark}
As we will see later, this case (\(r=1\)) fits Batyrev's setup;
the orbifold \(Y'\slash G'\) is a Calabi--Yau hypersurface 
in a certain toric variety defined by a reflexive polytope and
one can apply Batyrev's construction to 
produce a mirror for \(Y'\slash G'\).
\end{remark}

\begin{remark}
\label{rmk:not-p1-bundle}
We could also have embedded our base \(\mathbb{P}^{3}\)
into the projective space bundle
\(\mathbf{P}_{\mathbb{P}^{3}}(\mathbb{C}\oplus\mathbb{L})\)
using the section \(f\), and then constructed
the pre-quotient space \(Y'\) there. Here, \(\mathbb{L}\)
is the total space of the anti-canonical bundle of \(\mathbb{P}^{3}\).
In principle, we are 
able to compute the Gromov--Witten invariants of \(Y'\)
through \(\mathbf{P}_{\mathbb{P}^{3}}(\mathbb{C}\oplus\mathbb{L})\) as well
by a quantum hyperplane section
theorem. Note that \(\mathbf{P}_{\mathbb{P}^{3}}(\mathbb{C}\oplus\mathbb{L})\)
is a smooth semi-Fano toric manifold and \(Y'\) is simply a 
hyperplane section of a convex bundle; calculating Gromov--Witten invariants
of \(Y'\) through 
\(\mathbf{P}_{\mathbb{P}^{3}}(\mathbb{C}\oplus\mathbb{L})\)
seems standard.
However, the downside of this approach is that the 
(non-extended) \(I\)-function \(B(t;\alpha)\) would have 
two independent Novikov variables, because
\(\mathbf{P}_{\mathbb{P}^{3}}(\mathbb{C}\oplus\mathbb{L})\) has Picard rank two, and
a non-trivial change of variable (the mirror map) must be performed
in order to reduce the number of 
the Novikov variables to one, the Picard number of \(Y'\). In general, 
we do not have a precise formula for this.
It is the reason why we insist on working with
the orbifold \(\mathbb{P}(1,1,1,1,4)\). In fact,
\begin{equation}
\mathbf{P}_{\mathbb{P}^{3}}(\mathbb{C}\oplus\mathbb{L})\to\mathbb{P}(1,1,1,1,4)
\end{equation}
is a crepant contraction and \(Y'\) does not intersect
with the exceptional divisor.
\end{remark}


\section{A mirror theorem for Calabi--Yau double covers with
\texorpdfstring{\(r=1\)}{}}
\label{sec:a-mirror-theorem-for-r1}
In this section, we generalize the results in \S\ref{sec:4h}; 
we will prove the mirror theorem for Calabi--Yau
double covers when \(r=1\), i.e.~the case of trivial nef-partition \(E_{1}=-K_{X}\). 
We will treat the general case in a forthcoming paper.

To ease our notation, we will write \(\rho_{j}\equiv \rho_{i,j}\) and
\(\nu_{j}\equiv\nu_{i,j}\) and drop \(i\) in the subscript 
throughout this section (cf.~\S\ref{subsection:notation}). 
Notice that \(n_{1}=p\) in the present case. By duality construction, 
we also have \(\nabla=\Delta^{\vee}\).
Let us review the construction in 
\cite{2024-Hosono-Lee-Lian-Yau-mirror-symmetry-for-double-cover-calabi-yau-varieties}*{Appendix A}.
We will only focus on the double cover case.

\subsection{A toric bundle and its contraction} 
\label{subsection:a-toric-bundle-and-its-construction}

Let \(X\) be a smooth semi-Fano toric variety defined by a fan \(\Sigma\) and 
let \(\mathscr{L}=\mathscr{O}_{X}(-K_{X})\) be the canonical sheaf of \(X\).
Put \(\mathscr{E}=\mathscr{O}_{X}\oplus\mathscr{L}^{\vee}\)
and 
\begin{equation}
  Z=\mathrm{Proj}_{\mathscr{O}_{X}}(\mathrm{Sym}^{\bullet}\mathscr{E})=
  \mathbf{P}_{X}(\mathbb{L}\oplus\mathbb{C})
\end{equation}
to be the projectivization of the rank two vector bundle \(\mathbb{C}\oplus\mathbb{L}\).
Here \(\mathbb{L}\) is the total space of the line bundle \(\mathscr{L}\).
Apparently, \(Z\) is a toric variety and we now describe its toric data.

Let \(\mathrm{e}_{\infty}:=(\mathbf{0},1)\in \bar{N}:=N\times\mathbb{Z}\)
and \(\mathrm{e}_{0}:=(\mathbf{0},-1)\in \bar{N}\). Consider 
\begin{align*}
\begin{split}
\mathcal{S}_1&:=\{\nu_j:=(\rho_j,1)\in \bar{N}~\big|~j=1,\ldots,p\},~\mbox{and}\\
\mathcal{S}_2&:=\{\mathrm{e}_{\infty},\mathrm{e}_{0}\}.
\end{split}
\end{align*}
Any maximal cone \(\tau\in\Sigma(n)\) determines two maximal 
cones in \(\bar{N}\):
\begin{align}\label{maximal:cones-liftings}
\begin{split}
\tau_{0}&=\mathrm{Cone}(\{\nu_j~|~\rho_j\in\tau(1)\}\cup\{\mathrm{e}_{0}\}),~\mbox{and}~\\
\tau_{\infty}&=\mathrm{Cone}(\{\nu_j~|~\rho_j\in\tau(1)\}\cup\{\mathrm{e}_{\infty}\}).
\end{split}
\end{align}
\begin{definition}
Let \(\Sigma_{Z}\) be the collection of \(\tau_{0}\) and \(\tau_{\infty}\) 
as well as all their faces for all \(\tau\in\Sigma(n)\). 
\end{definition}
The following proposition is straightforward.
\begin{proposition}
\(\Sigma_{Z}\) defines the toric variety \(Z\). Furthermore, 
from the construction, the infinite divisor is given by the \(1\)-cone 
\(\mathrm{Cone}\{\mathrm{e}_{\infty}\}\).
\end{proposition}

\begin{proposition}
\label{proposition:H-is-nef}
The divisor \(H:=D_{\mathrm{e}_{\infty}}+
\sum_{j=1}^p D_{\bar\rho_j}\) is base point free. 
\end{proposition}

The next step is to show that \(H\) is a pullback
of a very ample divisor on a toric variety \(Z''\) and describe the 
toric variety and the contraction \(Z\to Z''\) explicity.

Let us recall the construction in 
\cite{2000-Mavlyutov-semi-ample-hypersurfaces-in-toric-varieties}.
Let \(X=X_{\Sigma}\) be an \(n\)-dimensional complete toric variety and \(H\) be a semiample divisor. 
Recall that a Cartier divisor \(H\) is called \emph{\(n\)-semiample} if \(H\) is
generated by global sections and \(H^{n}>0\) where \(n=\dim X\),
or equivalently, \(H\) is generated by global sections and \(\Delta_{H}\)
is of maximal dimension \(n\), or equivalently \(\mathscr{O}_{X}(H)\) is big and nef.
Assume that \(H=\sum_{\rho\in\Sigma(1)} a_{\rho}D_{\rho}\). 
We denote by \(\psi_{H}\) the support function associated with \(H\).
In the present case, \(\psi_{H}\) is convex.
For each \(\sigma\in \Sigma(n)\), we can find an element
\(m_{\sigma}\in M\) such that 
\begin{equation*}
\psi_{H}(u)=\langle u,m_{\sigma}\rangle,~u\in \sigma.
\end{equation*}
The collection \(\{m_{\sigma}\}_{\sigma\in\Sigma(n)}\) is 
called the \emph{Cartier data} of \(H\).
We glue together those maximal dimensional cones in \(\Sigma\) 
having the same \(m_{\sigma}\) and obtain a convex
rational polyhedral cone. In the present case, these cones are 
in fact strongly convex since \(\Delta_H\) has maximal dimension \(n\).
The set of these strongly convex rational polyhedral cones gives rise
to a new fan \(\Sigma_{H}\).
We remark that for each \(r\in\mathbb{Q}_{>0}\), \(rH\) produces the same fan.
Moreover, the fan \(\Sigma\) is a subdivision of \(\Sigma_{H}\). Let 
\(\pi\colon X\to X_{\Sigma_{H}}\) be the corresponding 
toric morphism and
\(\pi_{\ast}\colon A_{n-1}(X)\to A_{n-1}(X_{\Sigma_{H}})\)
be the pushforward map between Chow groups.
\begin{proposition}[\cite{2000-Mavlyutov-semi-ample-hypersurfaces-in-toric-varieties}*{Proposition 1.2}]
Let \(X=X_{\Sigma}\) and \(H\) be an \(n\)-semiample divisor. Then
there exists a unique complete 
toric variety \(X_{\Sigma_{H}}\) 
with a toric birational map \(\pi\colon X_{\Sigma}\to X_{\Sigma_{H}}\) such that
\(\Sigma\) is a refinement of \(\Sigma_{H}\), \(\pi_{\ast}[H]\),
is ample, and \(\pi^{\ast}\pi_{\ast}[H]=[H]\). Moreover,
\(\Sigma_{H}\) is the normal fan of \(\Delta_{H}\);
in other words, \(\mathbf{P}_{\Delta_{H}}=X_{\Sigma_{H}}\).
\end{proposition}

For simplicity, we put \(X''=\mathbf{P}_{\Delta}\).
Recall that \(\eta\colon X\to X''\) is a MPCP
desingularization, i.e.~\(\eta^{\ast}\omega_{X''}^{-1}\simeq \omega_{X}^{-1}\)
with \(\omega_{X''}^{-1}\) being ample.

We give a construction of the contraction \(\phi\colon Z\to Z''\). 
The Cartier data of \(H\) is easy to describe.
\begin{lemma}
Let \(\{m_{\tau}\}_{\tau\in\Sigma(n)}\) be the Cartier data for \(-K_{X}\).
Then the collection of \(\bar{m}_{\tau_{\infty}}:=(\mathbf{0},1)\)
and \(\bar{m}_{\tau_{0}}:=(m_{\tau},0)\) for \(\tau\in\Sigma(n)\) 
gives the Cartier data of \(H\).
\end{lemma}
It follows from the construction in 
\cite{2000-Mavlyutov-semi-ample-hypersurfaces-in-toric-varieties}*{Proposition~1.2}
that there exists a toric map \(\phi\colon Z\to Z'':=\mathbf{P}_{\Delta_{H}}\),
where \(\Delta_{H}\) is the polytope of \(H\).
Moreover, \(H''=\phi_{\ast}[H]\) is an ample divisor on \(Z''\) such that 
\(\phi^{\ast} \phi_{\ast}[H] = [H]\).
It is straightforward to see that \(Z''\) is obtained by
contracting the infinity divisor in \(\mathbf{P}_{X''}(\mathbb{L}''\oplus\mathbb{C})\),
where \(\mathbb{L}''\) is the geometric line bundle of \(\mathscr{O}_{X''}(-K_{X''})\).
Such a contraction exists since \(-K_{X''}\) is ample.
Let us summarize the data in the commutative diagram below.
\begin{equation}
\label{diag:z''}
  \begin{tikzcd}
    &Z=\mathbf{P}_{X}(\mathbb{L}\oplus\mathbb{C})\ar[r]\ar[rr,bend left,"\phi"]\ar[d]
    &\mathbf{P}_{X''}(\mathbb{L}''\oplus\mathbb{C})\ar[d]\ar[r] &Z''\\
    &X\ar[r] &X'' &
  \end{tikzcd}
\end{equation}
The first upper horizontal map is given by the nef divisor \(H\),
which the second one is obtained by contracting the divisor at infinity.
The lower horizontal map is the MPCP desingularization.

\begin{proposition}
\(Z''\) is Fano.
\end{proposition}
\begin{proof}
This follows from the fact that \(-K_{Z''}\sim 2H''\)
which is ample.
\end{proof}


In general \(Z''\) is very singular, possibly non-simplicial. 
It is difficult to compute Gromov--Witten invariants
of \(Y\) from \(Z''\). To facilitate our computation, 
we will construct a partial toric desingularization
\(\psi\colon Z'\to Z''\) in a way such that 
the toric structure of \(Z'\) is ``close'' to that of \(X\)
so that we can embed \(Y\) into \(Z'\) as well.

\subsection{A construction of another toric ambient space \texorpdfstring{\(X'\)}{X'}}

Recall that there
is a canonical projection \(\overline{N}_{\mathbb{R}}\to N_{\mathbb{R}}\).
\begin{definition}
Let \(\Sigma\) be a fan in \(N_{\mathbb{R}}\). 
For \(\sigma\in\Sigma(n)\), we put 
\begin{equation*}
\overline{\sigma}=\mathrm{Cone}(\{(\rho,1)~|~\rho\in\sigma(1)\}\cup\{(\mathbf{0},-1)\})
\subset\overline{N}_{\mathbb{R}}.
\end{equation*}
Let \(\overline{\Sigma}\) be the fan consisting of \(\overline{\sigma}\) and all their faces.
We call \(\overline{\Sigma}\) \emph{the canonical lifting} of \(\Sigma\).
The canonical projection \(\overline{N}\to N\) induces a map of fans 
\(\overline{\Sigma}\to \Sigma\)
under which \(\overline{\sigma}\) maps to \(\sigma\) for any \(\sigma\in\Sigma\).
\end{definition}
Note that collection of the 
maximal cones in \(\overline{\Sigma}\)
is
\begin{equation}
    \{\overline{\tau} \bigm\vert\tau\in\Sigma(n)\}.
\end{equation}

\begin{remark}
If the toric variety \(X_{\Sigma}\) is Gorenstein, then \(X_{\overline{\Sigma}}\)
is the total space of the anti-canonical bundle of \(X_{\Sigma}\).
\end{remark}
Cones in \(\overline{\Sigma}\) are of the forms:
\begin{itemize}
\item[(1)] \(\overline{\tau}\) for some \(\tau\in\Sigma\);
\item[(2)] \(\mathrm{Cone}(\{(\rho,1)\bigm\vert\rho\in\delta(1)\})\) for some \(\delta\in\Sigma\).
\end{itemize}
In particular, a cone in (2) is a face of a cone in (1).

Let \(\mu\in N\cap |\Sigma|\) be a primitive element. 
Denote by \(\Sigma^{\ast}(\mu)\) the star subdivision of \(\Sigma\)
at \(\mu\) (cf.~\cite{2011-Cox-Little-Schenck-toric-varieties}*{\S11.1}). 
Here is an observation.
\begin{lemma}
\label{lemma:star-subd}
For \(\mu\in\nabla\cap N\setminus\{\mathbf{0}\}\), we have
\begin{equation}
\overline{\Sigma^{\ast}(\mu)} = \overline{\Sigma}^{\ast}((\mu,1)).
\end{equation}
In other words,
the star subdivision of the canonical lifting 
\(\overline{\Sigma}\) at \((\mu,1)\)
is equal to the canonical lifting of the star subdivision of 
\(\Sigma\) at \(\mu\).
\end{lemma}
\begin{proof}
The cones in \(\Sigma^{\ast}(\mu)\) are of the following forms.
\begin{itemize}
\item[(a)] \(\sigma\) where \(\mu\notin\sigma\in\Sigma\).
\item[(b)] \(\mathrm{Cone}(\mu,\delta)\in\Sigma^{\ast}(\mu)\) where
\(\mu\notin \delta\in\Sigma\) and \(\{\mu\}\cup\delta\subset\sigma\in\Sigma\).
\end{itemize}
On one hand, the cones in \(\overline{\Sigma^{\ast}(\mu)}\)
are of the following forms:
\begin{itemize}[leftmargin=4em]
\item[(a1)] \(\overline{\sigma}\) where \(\mu\notin\sigma\in\Sigma\).
\item[(a2)] \(\mathrm{Cone}(\{(\rho,1)\bigm\vert\rho\in\delta(1)\})\), where
\(\mu\notin\delta\in\Sigma\).
\item[(b1)] \(\overline{\mathrm{Cone}(\mu,\tau)}\), where
\(\mu\notin \tau\) and \(\{\mu\}\cup\tau\subset\sigma\in\Sigma\).
\item[(b2)] \(\mathrm{Cone}(\{(\mu,1)\}\cup\{(\rho,1)\bigm\vert\rho\in\delta(1)\})\), where
\(\mu\notin \delta\) and \(\{\mu\}\cup\delta\subset\sigma\in\Sigma\).
\end{itemize}
Note that the cones in (a1) and (b1) contain \((\mathbf{0},-1)\),
while the cones in (a2) and (b2) do not.

On the other hand, the cones in \(\overline{\Sigma}^{\ast}((\mu,1))\) are 
a priori of the following forms:
\begin{itemize}[leftmargin=4em]
\item[(1a)] \(\overline{\sigma}\) where \((\mu,1)\notin \overline{\sigma}
\in\overline{\Sigma}\) with \(\sigma\in \Sigma\).
\item[(2a)] \(\tau\) where \((\mu,1)\notin \tau\in\overline{\Sigma}\) and 
\(\tau = \mathrm{Cone}(\{(\rho,1)~|~\rho\in\delta(1)\})\) for some 
\(\delta\in\Sigma\).
\item[(1b)] \(\mathrm{Cone}((\mu,1),\tau)\) where
\((\mu,1)\notin \tau\) with \(\tau=\overline{\delta}\) for
some \(\delta\in\Sigma\) and \(\{(\mu,1)\}\cup\tau\subset\sigma'\in\overline{\Sigma}\).
Note that in this case, \(\sigma'\) must be of the form \(\overline{\sigma}\) as well.
\item[(2b)] \(\mathrm{Cone}((\mu,1),\tau)\) where
\((\mu,1)\notin \tau\) with \(\tau=\mathrm{Cone}(\{(\rho,1)\bigm\vert\rho\in\delta(1)\})\) 
for some \(\delta\in\Sigma\) and \(\{(\mu,1)\}\cup\tau\subset\sigma'\in\overline{\Sigma}\).
\end{itemize}

We will prove that the cones in 
\(\mathrm{(a1)}\), \(\mathrm{(a2)}\), \(\mathrm{(b1)}\), and \(\mathrm{(b2)}\)
correspond to the cones in \(\mathrm{(1a)}\), \(\mathrm{(2a)}\), \(\mathrm{(1b)}\), 
and \(\mathrm{(2b)}\),
respectively.
\quad \\

\noindent {\bf Case I.~\(\mathrm{(1a)}=\mathrm{(a1)}\) and \(\mathrm{(2a)}=\mathrm{(a2)}\)}.

In \(\mathrm{(1a)}\), \((\mu,1)\not\in\overline{\sigma}\) implies that 
\(\mu\notin\sigma\). Indeed, if \(\mu\in\sigma\), we can write
\begin{equation}
\mu = \sum_{\rho\in\sigma(1)} c_{\rho}\rho,~c_{\rho}\ge 0. 
\end{equation}
Since \(\mu\in\nabla\cap N\setminus\{\mathbf{0}\}\) and \(\nabla\) is a 
reflexive polytope, \(\mu\) lies in some facet \(F\) containing \(\tau(1)\).
Suppose \(F\) is defined by a linear functional 
\(f(n)=\langle n,m\rangle=-1\)
for some \(m\in M\). Then
\begin{equation}
-1=f(\mu)=\langle \mu,m\rangle = \sum_{\rho\in\tau(1)}
c_{\rho}\langle \rho,m\rangle= -\sum_{\rho\in\tau(1)} c_{\rho}.
\end{equation}
It follows that \(\sum_{\rho\in\tau(1)} c_{\rho}=1\)
and therefore \((\mu,1)\in \overline{\tau}\). 
We deduce that the cones in \(\mathrm{(1a)}\) must belong to \(\mathrm{(a1)}\).
The converse is obvious since \(\mu\notin\tau\in\Sigma\) implies 
\((\mu,1)\notin\overline{\tau}\in\overline{\Sigma}\).
We conclude that the cones in \(\mathrm{(1a)}\)
and the cones in \(\mathrm{(a1)}\) are the same.
A similar argument shows that the cones in \(\mathrm{(2a)}\)
and \(\mathrm{(a2)}\) are the same.

\noindent {\bf Case II.~\(\mathrm{(1b)}=\mathrm{(b1)}\) and \(\mathrm{(2b)}=\mathrm{(b2)}\)}.

Let us turn to the case \(\mathrm{(b1)}\).
Note that \(\overline{\mathrm{Cone}(\mu,\tau)}=\mathrm{Cone}((\mu,1),\overline{\tau})\)
and that ``\(\mu\notin\tau\Rightarrow (\mu,1)\notin\overline{\tau}\).'' We deduce
that the cones in \(\mathrm{(b1)}\) belong to \(\mathrm{(1b)}\).
Conversely, \((\mu,1)\not\in\delta=\overline{\tau}\) implies that 
\(\mu\notin\tau\) and hence we conclude 
the cones in \(\mathrm{(1b)}\) and the cones in \(\mathrm{(b1)}\) are the same.
Finally, for any cone in \(\mathrm{(2b)}\), we have \(\mu\notin\tau\);
otherwise the same reason in the proof of cases \(\mathrm{(a1)}\)
and \(\mathrm{(1a)}\) implies that \((\mu,1)\in\delta\). 
The image of \(\sigma'\) under the projection \(\overline{N}_{\mathbb{R}}\to N_{\mathbb{R}}\)
gives the cone \(\sigma\) needed in \(\mathrm{(b2)}\). This shows the cones in \(\mathrm{(2b)}\)
belong to \(\mathrm{(b2)}\). The opposite inclusion is clear. 

In conclusion, we proved that the cones in 
\(\mathrm{(a1)}\), \(\mathrm{(a2)}\), \(\mathrm{(b1)}\), and \(\mathrm{(b2)}\)
correspond to \(\mathrm{(1a)}\), \(\mathrm{(2a)}\), \(\mathrm{(1b)}\), and \(\mathrm{(2b)}\),
respectively. This completes the proof.
\end{proof}

We observe that the defining fan \(\Sigma_{Z''}\) of \(Z''\) 
(cf.~\eqref{diag:z''} for definitions)
is the face fan of the upside down pyramid
\begin{equation}
\label{eq:inverted-pyramid}
\mathrm{Conv}(\nabla\times \{1\},(\mathbf{0},-1))\subset \overline{N}_{\mathbb{R}}.
\end{equation}
Let \(\overline{\Sigma}_{\Delta}\) be 
the canonical lifting of \(\Sigma_{\Delta}\), which
is isomorphic to the fan consisting of cones over the lower facets of the 
pyramid \eqref{eq:inverted-pyramid} and all their faces. 
In this case, since \(\Delta\)
is also a reflexive polytope, the canonical lifting \(\overline{\Sigma}_{\Delta}\)
is the fan for the total space of the line bundle \(\mathscr{O}_{\mathbf{P}_{\Delta}}
(-K_{\mathbf{P}_{\Delta}})\).

Recall that the MPCP resolution \(X\to X''\) is obtained 
from \(\Sigma_{\Delta}\)
by a sequence of star subdivisions at some \(\rho_{i}\in N\cap\nabla\setminus \{\mathbf{0}\}\).
The polytope in \eqref{eq:inverted-pyramid} is reflexive and \(\nu_{i}=(\rho_{i},1)\) are integral
points lying on its faces.

Let \(\tilde{\Sigma}_{Z''}\) be 
the fan obtained from \(\Sigma_{Z''}\) by the same sequence 
of star subdivisions at \(\nu_{i}\).
This gives rise to a subdivision \(\tilde{\Sigma}'\) on \(\tilde{\Sigma}_{\Delta}\).
By Lemma \ref{lemma:star-subd}, it is straightforward to see that
\(\tilde{\Sigma}'\) is equal to the fan (see \eqref{maximal:cones-liftings} for notation)
\begin{equation*}
\{\sigma\preceq\tau_{0}\bigm\vert\tau\in\Sigma(n)\}
\end{equation*}
which defines the ``finite part'' of \(Z\). 
It could happen that \(\tilde{\Sigma}_{Z''}\)
is non-simplicial. However, we can always take a simplicialization
to remedy this defect. Let \(\Sigma_{X'}\) be a
simplicialization of \(\tilde{\Sigma}_{Z''}\)
and \(X'\) be the toric variety associated with \(\Sigma_{X'}\). Note that
the simplicialization does not affect the subfan \(\tilde{\Sigma}'\) since it is smooth.
\begin{proposition}
\(K_{X'}\) is Cartier. 
\end{proposition}
\begin{proof}
Note that \(K_{Z''}\) is Cartier and \(X'\to Z''\) is obtained from 
adding some of the integral points in \(\nabla\cap N\setminus\{\mathbf{0}\}\).
It follows that \(\psi\colon X'\to Z''\) is a projective crepant partial resolution 
and in particular \(K_{X'} = \psi^{\ast}K_{Z''}\) is Cartier. 
\end{proof}
\subsection{The graph embedding and the pre-quotient space}
\label{subsec:the_graph_embedding_and_the_pre_quotient_space}
Having constructed a nice ambient toric variety, in this subsection, we
will demonstrate how to construct pre-quotient spaces for Calabi--Yau double
covers and how to embed them into the
toric variety we constructed.

Let \(f\in\mathrm{H}^{0}(X,\mathcal{L})\) be a smooth section. Then \(f\) gives rise
to an embedding
\begin{eqnarray*}
\Gamma_{f}\colon X\to\mathbf{P}_{X}(\mathbb{L}\oplus\mathbb{C}),~x\mapsto [f(x)\mathpunct{:}1]
\end{eqnarray*}
where \([f(x)\mathpunct{:}1]\) denotes the equivalence 
class of the vector \((f(x),1)\in\mathbb{L}\oplus\mathbb{C}\) in the projectivization.

Consider the composition \(\phi\circ\Gamma_{f}\colon X\to Z''\). 
Let \(\psi\colon X'\to Z''\) be a toric partial resolution constructed as above.
We arrive at the following commutative diagram
\begin{eqnarray}
\label{diag:construction-of-x'}
\begin{tikzcd}
&  &  &X'\ar[d,"\psi"]\\
&X\ar[bend left,rru,"g"]\ar[r,"\Gamma_{f}"] & Z\ar[l,bend left,"\pi"]\ar[r,"\phi"] & Z''
\end{tikzcd}
\end{eqnarray}
where \(\pi\colon Z\to X\) is the bundle projection.
We can lift \(\phi\circ\Gamma_{f}\) into \(g\colon X\to X'\) since the subfan  
\(\tilde{\Sigma}'\subset \Sigma_{X'}\) defines the finite part of \(Z\) in which 
\(\Gamma_{f}(X)\) lives.


Recall that \(\overline{N}=N\times\mathbb{Z}\) and \(\Sigma_{X'}\)
is the fan defining \(X'\) in \(\overline{N}'\otimes\mathbb{R}\). Consider a sublattice
\begin{eqnarray*}
\overline{N}'':=2N\times 2\mathbb{Z}\subset \overline{N}.
\end{eqnarray*}
Viewing \(\Sigma_{X'}\) as a fan in \(\overline{N}''\), we obtain 
a \(2^{n+1}\)-sheet covering \(\Phi\colon X'\to X'\)
branched
along the union of toric divisors on \(X'\). 
The Galois group \(G\) of \(\Phi\) is isomorphic to 
\(\overline{N}\slash \overline{N}'\simeq
\boldsymbol{\mu}_{2}^{n+1}\). Let \(Y'\) be the fibred product
\begin{eqnarray}
\label{diag:pre-quotient-construction}
\begin{tikzcd}
&Y'\ar[r]\ar[d] &X'\ar[d,"\Phi"]\\
&X\ar[r,"g"] & X'.
\end{tikzcd}
\end{eqnarray}

By construction, \(Y'\to X\) is a \(2^{n+1}\)-sheet cover
branched along \(\cup_{i=1}^{p} D_{i}\cup\{f=0\}\) and \(Y'\)
is invariant under the \(G\)-action as well. 
\begin{lemma}
\label{lem:y'-is-smooth}
\(Y'\) is smooth.
\end{lemma}
\begin{proof}
Note that \(g(X)\) lies in the smooth part of \(X'\).
It is sufficient to prove the following statement.
Let \(z_{1},\ldots,z_{n+1}\) be coordinates on \(\mathbb{C}^{n+1}\).

\begin{claim}
Assume that \(f\) is a function on \(\mathbb{C}^{n}\) with coordinates \(z_{1},\ldots,z_{n}\)
for which \(\{f=0\}\bigcup\cup_{i=1}^{n}\{z_{i}=0\}\)
is a simple normal crossing divisor.
Then 
\begin{equation}
F(z_{1},\ldots,z_{n+1}):=z_{n+1}^{2}-f(z_{1}^{2},\ldots,z_{n}^{2})  
\end{equation} 
defines 
a smooth subvariety in \(\mathbb{C}^{n+1}=\mathbb{C}^{n}\times\mathbb{C}\) 
whose coordinates are
\(z_{1},\ldots,z_{n},z_{n+1}\).
\end{claim}

Let us prove the claim.
Since \(\{f=0\}\) is smooth, the gradient vector
\begin{equation}
\left(\frac{\partial f}{\partial z_{1}},\ldots,\frac{\partial f}{\partial z_{n}}\right)
\end{equation}
must be non-vanishing on \(\{f=0\}\). 

Suppose on the contrary that \(\mathbf{a}:=(a_{1},\ldots,a_{n+1})\) 
is a singular point of \(\{F=0\}\);
in other words, we have
\begin{equation*}
\begin{cases}
\nabla F(\mathbf{a}) = \mathbf{0}\\
F(\mathbf{a}) = 0.
\end{cases}
\end{equation*}
From the first equality, we see that \(a_{n+1}=0\) and that 
\begin{equation}
\label{eq:gradient-vector}
\left(a_{1}\frac{\partial f}{\partial z_{1}}(a_{1}^{2},\ldots,a_{n}^{2}),
\ldots,a_{n}\frac{\partial f}{\partial z_{n}}(a_{1}^{2},\ldots,a_{n}^{2})\right)=\mathbf{0}.
\end{equation}
Combining with the second equality,
we have \(f(a_{1}^{2},\ldots,a_{n}^{2})=0\). 

Now \eqref{eq:gradient-vector} says that
\begin{equation*}
\frac{\partial f}{\partial z_{i}}(a_{1}^{2},\ldots,a_{n}^{2})\ne 0~\Rightarrow~a_{i}=0.
\end{equation*}
This implies that 
\(\{f=0\}\cup\bigcup_{i=1}^{n}\{z_{i}=0\}\)
is \emph{not} a simple normal crossing divisor. 
To see this, put \(b_{i}=a_{i}^{2}\) so that \(f(\mathbf{b})=0\). Then
at \(\mathbf{b}=(b_{1},\ldots,b_{n})\), we have
\begin{equation}
  (\nabla f)(\mathbf{b})=
  \left(\frac{\partial f(\mathbf{b})}{\partial z_{1}},\ldots,
  \frac{\partial f(\mathbf{b})}{\partial z_{n}}\right).
\end{equation}
The above argument shows that 
\begin{equation}
  \frac{\partial f(\mathbf{b})}{\partial z_{i}}\ne 0~\Rightarrow~b_{i}=a_{i}=0.
\end{equation}
Hence
\begin{equation}
  \mathbf{b}\in\left\{z_{i}=0\Bigm\vert \frac{\partial f(\mathbf{b})}{\partial z_{i}}\ne 0\right\}.
\end{equation}
But then the divisor \(\{f=0\}\bigcup\cup_{i=1}^{n}\{z_{i}=0\}\)
would not be a strictly normal crossing divisor at \(\mathbf{b}\),
contradicting to our assumption.
This completes the proof of the claim and hence the theorem.
\end{proof}

To relate this with our double cover, let us consider another sublattice
\begin{equation*}
\overline{N}':=N\times 2\mathbb{Z}\subset \overline{N}.
\end{equation*}
As before, we identify \(\overline{M}''\) with
\begin{equation*}
\left\{\left(\frac{a_{1}}{2},\ldots,\frac{a_{n+1}}{2}\right)~\Big|~a_{i}\in\mathbb{Z}\right\}
\end{equation*}
and \(\overline{M}\subset \overline{M}''\) corresponds to the subset consisting of integral points.
Under this identification,
\begin{equation*}
\overline{M}':=\left\{\left(a_{1},\ldots,a_{n},\frac{a_{n+1}}{2}\right)\Bigm\vert a_{i}\in\mathbb{Z}\right\}.
\end{equation*}
\(\overline{M}\) is an index 2 subgroup in \(\overline{M}'\). Let 
\(G':=\overline{N}'\slash\overline{N}''\).
Obviously \(G'\) is an index \(2\) subgroup in \(G=\overline{N}\slash\overline{N}''\).
We claim
\begin{proposition}
The covering \(X'\slash G'\to X'\slash G\cong X'\) is a double cover 
branched along the union of all toric divisors.
\end{proposition}
\begin{proof}
Clearly, \(X'\slash G\cong X'\) and \(X'\slash G'\to X'\slash G\) is a double cover.
Now we prove that the branched locus is the union of all toric divisors.

Let \(\nu\) be the primitive generator of a \(1\)-cone in \(\Sigma_{X'}\). 
We claim that \(\nu^{\perp}\cap \overline{M}'=\nu^{\perp}\cap \overline{M}\).
Let us write \(\nu = (v_{1},\ldots,v_{n},1)\). For \(x=(x_{1},\ldots,x_{n},x_{n+1}/2)
\in \nu^{\perp}\cap \overline{M}'\),
we have
\begin{equation*}
\sum_{i=1}^{n} v_{i}x_{i}+\frac{x_{n+1}}{2}\in \mathbb{Z}.
\end{equation*}
This implies that \(x_{n+1}\in 2\mathbb{Z}\) and 
therefore \(x\in \nu^{\perp}\cap \overline{M}\)
as claimed.
\end{proof}

\begin{remark}
We can relate this construction with the previous one for
\(\mathbb{P}(1,1,1,1,4)\) as follows. First of all,
in the present case \(X=X''=\mathbb{P}^{3}\) 
and \(\mathbb{L}\) is the total space of \(\mathscr{O}_{X}(4)\)
(cf.~diagram \eqref{diag:z''}). By contracting
the divisor at infinity, \(Z''=\mathbb{P}(1,1,1,1,4)\). In this case,
since no desingularization is needed and \(Z''\) is already simplicial, 
we see that \(X'=Z''\).

To compare our current construction with the one in \S\ref{sec:4h}, we can use the
integral linear transformation
\begin{equation}
  \begin{bmatrix}
    1 & -1 & 0 & 1\\
    -1 & 1 & 1 & 0\\
    0 & 1 & 1 & 0\\
    0 & 0 & 0 & 0
  \end{bmatrix}
\end{equation}
to relate their \(1\)-cones. In fact, it takes
\(\rho_{1},\ldots,\rho_{4}\) in \eqref{eq:comm-diag-quotient-p11114} (as column vectors) to
\begin{equation}
  \begin{bmatrix}
    2\\
    0\\
    0\\
    1
  \end{bmatrix},
  \begin{bmatrix}
    0\\
    2\\
    0\\
    1
  \end{bmatrix},
  \begin{bmatrix}
    0\\
    0\\
    2\\
    1
  \end{bmatrix},
  \begin{bmatrix}
    0\\
    0\\
    0\\
    -1
  \end{bmatrix}.
\end{equation}
\end{remark}

From the proposition, we see that \(Y'\slash G'\simeq Y\).
\begin{definition}
The variety \(Y'\) is called the \emph{pre-quotient} space 
of \(Y\).
\end{definition}
Furthermore, we can prove
\begin{proposition}
\(Y'\) is a smooth Calabi--Yau hypersurface in \(X'\).
\end{proposition}
\begin{proof}
Note that \(\Sigma_{X'}(1)=\{\nu_{i}\bigm\vert i=1,\ldots,p\}\cup\{(\mathbf{0},-1)\}\).
The map \(g\colon X\to X'\) realizes \(X\) as a hypersurface in \(X'\) defined
by a section of the bundle 
\begin{equation*}
\mathscr{O}_{X'}\left(\textstyle\sum_{i=1}^{p} D_{\nu_{i}}\right).
\end{equation*}
It follows that \(Y'\) is the zero locus of a section of the sheaf
\(\mathscr{O}_{X'}\left(2\textstyle\sum_{i=1}^{p} D_{\nu_{i}}\right)\) which
is the anti-canonical bundle since \(D_{\mathrm{e}_{0}}\sim \sum_{i=1}^{p} D_{\nu_{i}}\).
Lastly, lemma \ref{lem:y'-is-smooth} ensures that \(Y'\) is smooth.
\end{proof}
\subsection{GKZ systems and cohomology-valued \texorpdfstring{\(B\)}{B}-series}
\label{subsection:gkz-r=1-notation}
Given a double cover \(Y\to X\) branched along the nef-partition \(E_{1}=-K_{X}\),
the pre-quotient space \(Y'\) we have constructed 
is a \emph{smooth} Calabi--Yau hypersurface in a semi-Fano simplicial toric variety \(X'\).

Let us fix the following notation which will be used in the rest of the paper.
\begin{itemize}
\item Let \(\nu_{j}=(\rho_{j},1)\in N\times\mathbb{Z}\), \(1\le j\le p\), and
\(\nu_{p+1}=(\mathbf{0},-1)\in N\times\mathbb{Z}\) be the \(1\)-dimensional cones in the fan defining \(X'\).

\item Let 
\begin{equation*}
\bar{A}=\begin{bmatrix}
\mathbf{0} &\rho_{1}^{\intercal} & \cdots & \rho_{p}^{\intercal} & \mathbf{0}\\
0& 1 & \cdots & 1 & -1 \\
1& 1 & \cdots & 1 & 1 \\
\end{bmatrix}\in\mathrm{Mat}_{(n+2)\times(p+2)}(\mathbb{Z})
\end{equation*}
and 
\begin{equation*}
\bar{\beta}=
\begin{bmatrix}
\mathbf{0}\\
0\\
-1
\end{bmatrix}\in\mathbb{C}^{n+2}.
\end{equation*}
The GKZ hypergeometric system \(\mathcal{M}(\bar{A},\bar{\beta})\) 
governs the periods of the Batyrev--Borisov mirror family of \(Y'\subset X'\).
Let \(\{x_{i}\}_{i=0}^{p+1}\) be the coordinates for the
GKZ hypergeometric system \(\mathcal{M}(\bar{A},{\beta})\) corresponding
to the columns of \(\bar{A}\).
Let \(\bar{\gamma}=\begin{bmatrix}\mathbf{0},0,-1\end{bmatrix}^{\intercal}
\in\mathbb{C}^{p+2}\). We have \(\bar{A}(\bar{\gamma})=\bar{\beta}\) 

\item Let
\begin{equation*}
A_{\mathrm{ext}}=\begin{bmatrix}
\mathbf{0} & \rho_{1}^{\intercal} & \cdots & \rho_{p}^{\intercal}\\
1 & 1 & \cdots & 1 \\
\end{bmatrix}\in\mathrm{Mat}_{(n+1)\times(p+1)}(\mathbb{Z})
\end{equation*}
and 
\begin{equation*}
\beta=
\begin{bmatrix}
\mathbf{0}\\
-1/2
\end{bmatrix}\in\mathbb{C}^{n+1}.
\end{equation*}
The GKZ hypergeometric system \(\mathcal{M}({A_{\mathrm{ext}}},{\beta})\)
governs the periods of \(\mathcal{Y}^{\vee}\to U\), 
the \emph{gauged fixed double cover branched along the dual
nef-partition \(F_{1}=-K_{X^{\vee}}\) over \(X^{\vee}\)}.
Let \(\{w_{i}\}_{i=0}^{p}\)
be the coordinates for \(\mathcal{M}(A_{\mathrm{ext}},\beta)\) 
corresponding to the columns of \(A_{\mathrm{ext}}\).
Let \(\gamma=\begin{bmatrix}\mathbf{0},-1/2\end{bmatrix}^{\intercal}\in\mathbb{C}^{p+1}\).
We have \(A_{\mathrm{ext}}(\gamma)=\beta\).

\item Let \(\bar{L}=\mathrm{ker}(\bar{A})\). Note that 
\(L_{\mathrm{ext}}\simeq \bar{L}\simeq L\)
where \(L=\mathrm{ker}(A)\) and \(\bar{L}\) are defined in \S\ref{subsection:notation}.
\item Let \(\{\ell^{(1)},\ldots,\ell^{(p-n)}\}\) be a \(\mathbb{Z}\)-basis
of \(L_{\mathrm{ext}}\). We assume that the cone generated by 
\(\{\ell^{(1)},\ldots,\ell^{(p-n)}\}\) is smooth and contains the Mori cone of \(X\)
under the isomorphism \(L_{\mathrm{ext}}\simeq L\).
For an element \(\ell\in L_{\mathrm{ext}}\), the corresponding element under this isomorphism
is denoted by \(\bar{\ell}\). 
Explicitly,
\begin{equation}
\label{eq:mori-cone-correspondence}
\ell=(\ell_{0},\ldots,\ell_{p})
\leftrightarrow\bar{\ell}=(\bar{\ell}_{0},\ldots,\bar{\ell}_{p+2})=
(2\ell_{0},\ell_{1},\ldots,\ell_{p},-\ell_{0}).
\end{equation}
In particular,
the corresponding basis in \(\bar{L}\) is 
\(\{\bar{\ell}^{(1)},\ldots,\bar{\ell}^{(p-n)}\}\).

\item For each \(j=1,\ldots,p-n\), let
\begin{equation*}
z_{j} = \prod_{i=0}^{p+1} x_{i}^{\ell^{(j)}}~\mbox{and}~
\bar{z}_{j} = \prod_{i=0}^{p} w_{i}^{\bar{\ell}^{(j)}}
\end{equation*}
be the ``torus invariant'' coordinates for the GKZ systems.

\item Let \(\xi\in\mathrm{H}^{2}(Z,\mathbb{Z})\) be the 
\(1\)\textsuperscript{st} Chern class of the relative ample line bundle \(\pi\colon Z\to X\).
\item Let \(D_{1},\ldots,D_{p}\) be the divisors associated with
\(\rho_{1},\ldots,\rho_{p}\) on \(X\).
Let \(\bar{D}_{1},\ldots,\bar{D}_{p+1}\) be 
the divisors associated with
\(\nu_{1},\ldots,\nu_{p+1}\) on \(X'\).
\end{itemize}

\begin{lemma}
The matrices \(A_{\mathrm{ext}}\) and \(\bar{A}\) give rise to the following commutative diagram
\begin{equation*}
\begin{tikzcd}
& \mathbb{Z}^{p+2} \ar[r,"\bar{A}"]\ar[d,"q"] &\mathbb{Z}^{n+2}\ar[d,"q'"]\\
& \mathbb{Z}^{p+1} \ar[r,"A_{\mathrm{ext}}"] &\mathbb{Z}^{n+1}
\end{tikzcd}
\end{equation*}
where \(q\) is the projection given by 
forgetting the \((p+1)\)\textsuperscript{th} coordinate 
(according to our convention, this corresponds to the column
\(\begin{bmatrix}\mathrm{0} & -1 & 1\end{bmatrix}^{\intercal}\)
in \(\bar{A}\).)
and \(q'\) is the projection given by forgetting the \((n+1)\)\textsuperscript{st}
coordinate. Moreover, \(A_{\mathrm{ext}}\) and \(\bar{A}\) are surjective
and \(q\) induces an isomorphism \(\bar{L}\simeq L_{\mathrm{ext}}\).
\end{lemma}
\begin{proof}
The proof is elementary and hence omitted.
\end{proof}
We can state the main theorem in this section.
\begin{theorem}
\label{thm:same-hol-series}
Under a suitable identification, the unique holomorphic series solutions 
(cf.~\cite{1996-Hosono-Lian-Yau-gkz-generalized-hypergeometric-systems-in-mirror-symmetry-of-calabi-yau-hypersurfaces}*{Equation (3.5)}) to  
\(\mathcal{M}(A_{\mathrm{ext}},\beta)\) and \(\mathcal{M}(\bar{A},\bar{\beta})\) are identical.
\end{theorem}
\begin{proof}
On one hand, for \(\mathcal{M}(\bar{A},\bar{\beta})\),
the equation 
\cite{1996-Hosono-Lian-Yau-gkz-generalized-hypergeometric-systems-in-mirror-symmetry-of-calabi-yau-hypersurfaces}*{Equation (3.5)}
becomes
\begin{equation*}
\bar{\Phi}^{\bar{\gamma}}(x) = \sum_{\bar{\ell}\in \bar{L}}
\frac{\Gamma(1-\bar{\ell}_{0})
(-1)^{\bar{\ell}_{0}}}
{\prod_{j=1}^{p+1}\Gamma(\bar{\ell}_{j}+1)}x^{\bar{\ell}}.
\end{equation*}
Under the correspondence \eqref{eq:mori-cone-correspondence}, if \(\ell\)
is the corresponding element in \(L_{\mathrm{ext}}\), the coefficient 
in \(\bar{\Phi}^{\bar{\gamma}}(x)\)
becomes
\begin{equation}
\label{eq:gamma-function-computations}
\frac{\Gamma(1-2\ell_{0})
(-1)^{2\ell_{0}}}
{\Gamma(1-\ell_{0})\prod_{j=1}^{p}
\Gamma(\ell_{j}+1)}=
\frac{\Gamma(1-2\ell_{0})}
{\Gamma(1-\ell_{0})\prod_{j=1}^{p}
\Gamma(\ell_{j}+1)}
\end{equation}
We can easily compute
\begin{equation*}
\eqref{eq:gamma-function-computations} = 
\frac{\Gamma(1/2-\ell_{0})
2^{-2\ell_{0}}}
{\Gamma(1/2)\prod_{j=1}^{p}
\Gamma(\ell_{j}+1)}.
\end{equation*}
On the other hand, for \(\mathcal{M}(A_{\mathrm{ext}},\beta)\),
the equation 
\cite{1996-Hosono-Lian-Yau-gkz-generalized-hypergeometric-systems-in-mirror-symmetry-of-calabi-yau-hypersurfaces}*{Equation (3.5)}
is
\begin{equation*}
\Phi^{\gamma}(w) = \sum_{\ell\in L_{\mathrm{ext}}}
\frac{\Gamma(1/2-\ell_{0})
(-1)^{\ell_{0}}}
{\Gamma(1/2)\prod_{j=1}^{p}
\Gamma(\ell_{j}+1)}w^{\ell}.
\end{equation*}
If we introduce the change of variables 
\begin{align*}
\begin{cases}
w_{j}=x_{j},~\mbox{for}~j=1,\ldots,p,\\ 
w_{0}=-x_{0}^{2}/4x_{p+1},
\end{cases}
\end{align*}
we see that \(\bar{\Phi}^{\bar{\gamma}}(x)=\Phi^{\gamma}(w)\).
\end{proof}

To keep our presentation concise, 
we shall only recall the definition of the 
cohomology-valued \(B\)-series in the present situation.
In short, it is
a cohomology-valued series constructed from the unique holomorphic
period around the large complex structure limit point
in the moduli of \(Y^{\vee}\) by replacing the components
of the lattice relation vector \(\ell\) with cohomology classes.

\begin{definition}[Cf.~\cite{2022-Lee-Lian-Yau-on-calabi-yau-fractional-complete-intersections}*{\S3}]
The \emph{cohomology-valued \(B\)-series for the singular CY 
double cover \(Y^{\vee}\)}
is a cohomology-valued series defined by
\begin{equation}
  B_{X}^{\gamma}(w):=\left(\sum_{\ell\in A_{\mathrm{ext}}}\mathcal{O}_{\ell}^{\gamma}
  w^{\ell+\gamma}\right)\exp\left(\sum_{i=0}^{p}(\log w_{i})D_{i}\right)
\end{equation}
where
\begin{equation}
\mathcal{O}_{\ell}^{\gamma}:=
\frac{\Gamma(1/2-D_{0})
(-1)^{\ell_{0}}}
{\Gamma(1/2)\prod_{j=1}^{p}
\Gamma(D_{j}+1)}\in \mathrm{H}^{\bullet}(X;\mathbb{C})
\end{equation}
and \(D_{0}=-\sum_{i=1}^{p}D_{i}\).
\end{definition}

Let us also recall the cohomology-valued
\(B\)-series for classical CY hypersurfaces in toric varieties.
\begin{definition}[Cf.~\cite{1996-Hosono-Lian-Yau-gkz-generalized-hypergeometric-systems-in-mirror-symmetry-of-calabi-yau-hypersurfaces}]
The \emph{cohomology-valued \(B\)-series for the CY 
hypersurface \(Y'\) in \(X'\)}
is a cohomology-valued series defined by
\begin{equation}
  B_{X'}^{\bar{\gamma}}(x):=\left(\sum_{\bar{\ell}\in \bar{L}}
  \mathcal{O}_{\bar{\ell}}^{\bar{\gamma}}
  x^{\bar{\ell}+\bar{\gamma}}\right)
  \exp\left(\sum_{i=0}^{p+1}(\log x_{i})\bar{D}_{i}\right)
\end{equation}
where
\begin{equation}
\mathcal{O}_{\bar{\ell}}^{\bar{\gamma}}= \sum_{\bar{\ell}\in \bar{L}}
\frac{\Gamma(1-\bar{D}_{0})
(-1)^{\bar{\ell}_{0}}}
{\prod_{j=1}^{p+1}\Gamma(\bar{D}_{j}+1)}x^{\bar{\ell}}
\in \mathrm{H}^{\bullet}(X';\mathbb{C})
\end{equation}
and \(\bar{D}_{0}=-\sum_{i=1}^{p+1}\bar{D}_{i}\).
\end{definition}

\begin{remark}
The cohomology-valued \(B\)-series only encodes the untwisted part
of the genus zero orbifold Gromov--Witten invariants of \(X'\slash G'\).
\end{remark}

\subsection{A mirror theorem}
In this subsection, we prove a version of mirror theorem for our Calabi--Yau
double cover with \(r=1\). We retain the 
notation in \S\ref{subsection:gkz-r=1-notation}.
Let us summarize what we have achieved so far. 
Recall that we have a 
commutative diagram \eqref{diag:pre-quotient-construction}.
\begin{eqnarray}
\begin{tikzcd}
&  &  &X'\ar[d,"\psi"]\\
&X\ar[bend left,rru,"g"]\ar[r,"\Gamma_{f}"] & Z\ar[l,bend left,"\pi"]\ar[r,"\phi"] & Z''
\end{tikzcd}
\end{eqnarray}
In this diagram, 
\begin{itemize}
\item \(Z=\mathbf{P}_{X}(\mathbb{L}\oplus\mathbb{C})\), \(\pi\colon Z\to X\) is the bundle projection,
and \(\Gamma_{f}\colon X\to Z\) is the embedding induced by \(f\in\mathrm{H}^{0}(X,\mathcal{L})\),
where \(\mathcal{L}=\mathscr{O}_{X}(-K_{X})\).
\item \(\phi\colon Z\to Z''\) is the toric contraction of 
associated with the divisor \(\sum_{i=1}^{p}D_{\nu_{i}}+D_{\mathrm{e}_{\infty}}\).
\item \(\psi\colon X'\to Z''\) is a partial crepant resolution. Note that 
\(X'\) contains an open toric subvariety which is isomorphic to \(Z\setminus D_{\mathrm{e}_{\infty}}\).
\item \(g\) is a lifting of \(\phi\circ\Gamma_{f}\). 
\end{itemize}

Using the map \(g\) and the covering map \(\Phi\) constructed in 
\S\ref{subsection:a-toric-bundle-and-its-construction}, we can form
a fibred product and obtain a smooth subvariety \(Y'\) in \(X'\),
which is called a pre-quotient space.
Moreover, there exists an index \(2\) subgroup \(G'\subset G\) of the
Galois group of the covering \(\Phi\) such that \(Y'\slash G'\subset X'\slash G'\).
They fit the following commutative diagram
\begin{eqnarray}
\begin{tikzcd}
&Y'\ar[rr,"g'"]\ar[ld,"q"]\ar[dd] & &X'\ar[dd,"\Phi"]\ar[ld,"q"]\\
Y'\slash G'\ar[rd,"p"]\ar[rr,crossing over,"t" near end] & & X'\slash G'\ar[rd,"p"]\\
&X\ar[rr,"g"] & & X'.
\end{tikzcd}
\end{eqnarray}
In the above diagram, by abuse of notation, the restriction of \(q\) and \(p\)
to \(Y'\) and \(Y'\slash G'\)
are again denoted by \(q\) and \(p\). Also notice that \(Y\cong Y'\slash G'\).

Now we have two cohomology-valued \(B\)-series;
they are coming from the same holomorphic series by Theorem \ref{thm:same-hol-series} 
but they take value in different cohomology rings.
More precisely, for our double cover \(Y\to X\), the corresponding \(B\) series
takes value in \(\mathrm{H}^{\bullet}(X;\mathbb{C})\). From
our pre-quotient space construction \(Y\cong Y'\slash G'\), 
the corresponding \(B\)-series takes value in 
\(\mathrm{H}^{\bullet}_{\mathrm{CR}}([X'\slash G'];\mathbb{C})\)
whose untwisted part is equal to 
\(p^{\ast}\mathrm{H}^{\bullet}(X';\mathbb{C})\).

We begin with the following observation.
\begin{lemma}
The pullback map \(g^{\ast}\colon \mathrm{H}^{\bullet}(X';\mathbb{C})\to 
\mathrm{H}^{\bullet}(X;\mathbb{C})\) is surjective.
Consequently, \(p^{\ast}\mathrm{H}^{\bullet}(X;\mathbb{C})
=(g\circ p)^{\ast}\mathrm{H}^{\bullet}(X';\mathbb{C})\)
and 
\begin{equation}
  g^{\ast}B_{X'}^{\bar{\gamma}}(x)=B_{X}^{\gamma}(w)
\end{equation}
under the change of variable 
\begin{align*}
\begin{cases}
w_{j}=x_{j},~\mbox{for}~j=1,\ldots,p,\\ 
w_{0}=-x_{0}^{2}/4x_{p+1},
\end{cases}
\end{align*}
\end{lemma}
\begin{proof}
As a ring, \(\mathrm{H}^{\bullet}(X;\mathbb{C})\)
is generated by toric divisors on \(X\). It is thus sufficient
to prove that \(\mathrm{Im}(g^{\ast})\) contains all toric divisors.

Denote by \(V\) the (open) smooth toric variety defined by \(\tilde{\Sigma}'\).
By construction, \(V = Z\setminus D_{\mathrm{e}_{\infty}}\) 
with \(Z=\mathbf{P}_{X}(\mathbb{L}\oplus\mathbb{C})\) and the lifting
\(g\) is given by 
\begin{equation*}
X \xrightarrow{\Gamma_{f}} V(\subset Z) = V(\subset X'). 
\end{equation*}
The \(1\)-cone \(\mathbb{R}_{\ge 0}\cdot \nu_{i}\), regarded as a \(1\)-cone in \(\Sigma_{X'}\),
gives a \(\mathbb{Q}\)-Cartier divisor on \(X'\), which is Cartier on \(V\). 
Under the pullback \(g^{\ast}\), this line bundle is exactly the same as the 
line bundle associated with \(\rho_{i}\). 

For the last part, we notice that 
\(g^{\ast}\bar{D}_{i}=D_{i}\) for \(1\le i\le p\),
\(g^{\ast}\bar{D}_{p+1}=-D_{0}\) and \(g^{\ast}\bar{D}_{0}=2D_{0}\).
This concludes the proof.
\end{proof}

\begin{corollary}
There is an isomorphism 
\begin{equation*}
(g\circ p)^{\ast}\mathrm{H}^{\bullet}(X';\mathbb{C})
\simeq t^{\ast}p^{\ast}\mathrm{H}^{\bullet}(X';\mathbb{C})
\end{equation*}
of subrings in \(\mathrm{H}^{\bullet}(Y;\mathbb{C})\).
\end{corollary}

According to the Corollary, we can compare two cohomology-valued 
\(B\) series. Under the identification of these cohomology groups,
we deduce that
\begin{equation*}
p^{\ast}B_{X}^{\gamma} = t^{\ast}B_{[X'\slash G']}^{\bar{\gamma}}.
\end{equation*}
Therefore,
\begin{equation*}
t_{!}p^{\ast}B_{X}^{\gamma}=
t_{!}t^{\ast}B_{[X'\slash G']}^{\bar{\gamma}}
=B_{[X'\slash G']}^{\bar{\gamma}}\cup\mathrm{c}_{1}(-K_{X'\slash G'}).
\end{equation*}
We can summarize the result into the following theorem.
\begin{theorem}
\label{thm:main}
The cohomology-valued \(B\)-series \(B_{X}^{\gamma}\) for the 
singular double cover 
computes the genus zero 
untwisted orbifold Gromov--Witten invariants with
insertions from the base \(X\) after a change of variables.
\end{theorem}

\begin{remark}
As mentioned in earlier, it is important to understand 
double cover singular CYs in the context homological mirror symmetry (HMS). 
To this end, we must find an appropriate formulation of the \(A\)-side 
and \(B\)-side categories, followed by constructing an equivalence 
between them. The first attempt was made in 
\cite{2023-Lee-Lian-Romo-non-commutative-resolutions-as-mirrors-of-singular-calabi-yau-varieties} and \cite{Lee:2025aa}. 
The main idea was to construct a non-commutative resolution (NCR) 
using GLSM and matrix factorization theories. The mirror symmetry 
test was based on a comparison of two period sheaves. 
However a full fledge categorical equivalence has yet to be understood. To this end, Kawamata \cite{2002-Kawamata-d-equivalence-and-k-equivalence} 
has introduced the notion of orbifold 
$\mathrm{D}^{b}\mathrm{Coh}(Y)$ for any orbifold $Y$. 
We are interested in the case when $Y$ is a 
CY orbifold given by a double cover of a smooth toric variety branched 
along a nef-partition. Since our singular double cover CYs 
are of these kinds, this may be a strong candidate for the \(B\)-side 
category for HMS. For the \(A\)-side category, the above-mentioned NCR construction seems promising as a possible candidate. However the symplectic structure of the orbifold mirror $Y$ is not at all manifest in the NCR construction. For this purpose, one might consider two possible approaches. One approach is to use the notion of wrapped Fukaya categories \cite{MR2602848} to incorporate singularity information of $Y$ via Milnor fibers of weighted homogeneous polynomials. Another approach is to use an equivariant version of Fukaya category of a smooth symplectic CY pre-quotient $Y'$, equipped with a global finite abelian group action $G'$ such that $Y=Y'/G'$. We hope to pursue this line of attack on the HMS problem in a subsequent study.
\end{remark}

\section{Morrison's conjecture}
\label{sec:morrison_s_conjecture}
In \cite{1999-Morrison-through-the-looking-glass}, Morrison 
conjectured that extremal transitions are reversed under mirror
symmetry. An extremal transition 
is a birational contraction from
a smooth CY to a singular one and followed by a complex
smoothing to another smooth CY.

The aim of this section is to test Morrison's conjecture using
our singular CY double covers with \(r=1\). It turns out that
in this case our singular mirror proposal fits the picture well.
See also \cite{2015-Fredrickson-extremal-transitions-from-nested-reflexive-polytopes}
for extremal transitions from nested reflexive polytope.
Let us retain the notations in \S\ref{sec:a-mirror-theorem-for-r1},
especcially those in \S\ref{subsection:a-toric-bundle-and-its-construction}.

\subsection{Extremal transitions from polytopes}

\begin{definition}
Let \(P\) be a polytope in \(M_{\mathbb{R}}\) or 
\(N_{\mathbb{R}}\) and \(k\in\mathbb{N}\) be a positive integer.
Denote by \(\operatorname{Vert}(P)\) the set of vertices of \(P\).
We define \(P_{k}\) to be the convex hull of 
\begin{equation}
  (ku,1)~\mbox{with}~u\in \operatorname{Vert}(P)~\mbox{and}~(\mathbf{0},-1).
\end{equation}
\end{definition}

Let \(\Delta\subset M_{\mathbb{R}}\) be a reflexive polytope and 
\(\nabla=\Delta^{\vee}\subset N_{\mathbb{R}}\) be the dual polytope. 
\begin{lemma}
Then \(\Delta_{2}\) is a reflexive polytope whose dual 
polytope is \(\nabla_{1}\subset \overline{N}_{\mathbb{R}}\times\mathbb{R}\),
the convex hull of
\begin{equation}
(v,-1)~\mbox{with}~v\in \operatorname{Vert}(\nabla)~\mbox{and}~(\mathbf{0},1).
\end{equation}
\end{lemma}
\begin{proof}
We will show that the defining equation for any facet (a codimension one face)
of \(\Delta_{2}\) is of the form
\begin{eqnarray}
  \langle -,\bar{n}\rangle = -1
\end{eqnarray}
where \(\bar{n}\in \overline{N}=N\times\mathbb{Z}\) and \(\langle -,-\rangle\)
is the canonical pairing between \(\overline{M}\) and \(\overline{N}\).

We observe that if \(F\) is a facet of \(\Delta\), then
the convex hull 
\begin{equation}
  \mathrm{Conv}(\{(\rho,1)~\vert~\rho\in F\}\cup \{(\mathbf{0},-1)\})
\end{equation}
of \(F\) and \((\mathbf{0},-1)\) is a facet of \(\Delta_{2}\). Conversely
any facet of \(\Delta_{2}\) other than the base, i.e.~the 
convex hull of \((2u,1)~\mbox{with}~u\in \operatorname{Vert}(\Delta)\), is of this form.
Let \(\langle -,n\rangle=-1\) be the defining equation for
a facet of \(\Delta\). Here we slightly abuse of notation; 
we use the same symbol \(\langle -,-\rangle\)
to denote the canonical pairing between \(M\) and \(N\).
Then the facet \(\mathrm{Conv}(F\cup \{(\mathbf{0},1)\})\) is defined by
\begin{eqnarray}
  \langle -,(n,1)\rangle=-1.
\end{eqnarray}
Indeed, we have \(\langle (\mathbf{0},-1),(n,1)\rangle=-1\) and
\begin{equation}
  \langle (2u,1),(n,1)\rangle=2\langle u,n\rangle+1 =-1.
\end{equation}
It is also obvious that the base is defined by the equation
\begin{eqnarray}
  \langle -,(\mathbf{0},-1)\rangle=-1.
\end{eqnarray}
Now it follows from the equations that \(\Delta_{2}^{\vee}\)
is the convex hull of
\begin{eqnarray}
  (v,1)~\mbox{with}~v\in \operatorname{Vert}(\nabla)~\mbox{and}~(\mathbf{0},-1).
\end{eqnarray}
The proof is now completed.
\end{proof}

Similarly, we have the following result.
\begin{lemma}
\(\nabla_{2}\) is a reflexive polytope whose dual 
polytope is
\(\Delta_{1}\subset \overline{M}_{\mathbb{R}}\times\mathbb{R}\),
the convex hull of
\begin{equation}
(v,1)~\mbox{with}~v\in \operatorname{Vert}(\Delta)~\mbox{and}~(\mathbf{0},-1).
\end{equation}
\end{lemma}

For a given pair of reflexive polytopes \((\Delta,\nabla)\) with 
\(\Delta\subset M_{\mathbb{R}}\) and \(\nabla\subset N_{\mathbb{R}}\),
we obtain four different polytopes \(\Delta_{1}\), \(\Delta_{2}\),
and \(\nabla_{1}\), \(\nabla_{2}\) in relevant vector spaces.
We have also inclusions
\begin{equation}
  \Delta_{1}\subset \Delta_{2}~\mbox{and}~\nabla_{1}=\Delta_{2}^{\vee}
  \subset\nabla_{2}=\Delta_{1}^{\vee}.
\end{equation}

\subsection{Relations with singular CY double covers}
From the inclusions
\begin{equation}
   \Delta_{1}\subset \Delta_{2}~\mbox{and}~\Delta_{2}^{\vee}=\nabla_{1}
\end{equation} 
we obtain a degeneration of CY double covers \(Y\) of \(X\).
Recall that \(X\to\mathbf{P}_{\Delta}\) is a MPCP desingularization,
which is smooth under the Hypothesis \ref{assumption}.
Note that \(\Delta_{2}\) is the section polytope 
of the anti-canonical divisor of
\begin{equation*}
Z:=\mathbf{P}_{X}(\mathbb{L}\oplus\mathbb{C}).
\end{equation*}
Here, \(\mathbb{L}\) is the total space of the line bundle
whose sheaf of sections is \(\omega_{X}^{-1}\).
Note that \(\mathrm{H}^{0}(Z,\mathscr{O}_{Z}(D_{\mathrm{e}_{0}}))\)
is one-dimensional.

The next lemma is straightforward.
\begin{lemma}
The general sections of \(-K_{Z}\) given by integral
points in \(\Delta_{1}\subset\Delta_{2}\)
are of the form
\begin{equation}
  y^{2}-\left(\prod_{i=1}^{p}\pi^{\ast}s_{i,1}\right)\pi^{\ast}f
\end{equation}
where \(y\) is a general section
of the basepoint free sheaf \(\mathscr{O}_{Z}(D_{\mathrm{e}_{0}})\),
\(f\in\mathrm{H}^{0}(Z,\omega_{Z}^{-1})\)
is a general section, and \(\pi\colon Z\to X\)
is the bundle projection.
The equation defines a gauge fixed double cover family of \(X\) branched along
the nef-partition \(-K_{X}\).
\end{lemma}

According to the results in \S\ref{sec:a-mirror-theorem-for-r1}, the family
obtained in the preceding lemma
is a anticanonical family in the toric variety \(X'\slash G'\) whose
anticanonical polytope is on the nose \(\nabla_{2}\).
Based on the resolution procedure in 
\cite{2013-Sheng-Xu-Zuo-maximal-families-of-calabi-yau-manifolds-with-minimal-length-yukawa-coupling}*{\S2.2}, one can construct
a crepant resolution by a sequence of blow-ups of \(Z\) whose
centers are given by intersections of two toric divisors;
they are the zero section of \(\mathbb{L}\to X\) 
(regard as a divisor in \(Z\)) and a
toric divisors pullback from \(X\).
In this manner, we are able to construct a 
crepant resolution \(\tilde{Y}\to Y\),
which turns out to be a MPCP desingularization
of the Fano toric variety \(\mathbf{P}_{\Delta_{1}}\).

We then obtain an extremal transition
\begin{equation}
  \begin{tikzcd}
    & &\tilde{Y}\ar[d]\\
    &S\ar[r,rightsquigarrow] &Y
  \end{tikzcd}
\end{equation}
where \(S\) is a \emph{smooth} double cover over \(X\).

Batyrev's mirror construction produces
\begin{itemize}
  \item a mirror \(\tilde{Y}^{\vee}\) 
  of \(\tilde{Y}\) (\(\tilde{Y}^{\vee}\) is an anti-canonical
  hypersuface in 
  a MPCP desingularization of \(\mathbf{P}_{\nabla_{2}}\)
  while \(\tilde{Y}\) is an anti-canonical 
  hypersurface in a MPCP desingularization of 
  \(\mathbf{P}_{\Delta_{1}}\));
  \item a mirror \(S^{\vee}\) of \(S\)
  (\(S\) is an anti-canonical hypersurface in 
  a MPCP desingularization of
  \(\mathbf{P}_{\Delta_{2}}\) while \(S^{\vee}\)
  is an anti-canonical 
  hypersurface in a MPCP desingularization of 
  \(\mathbf{P}_{\nabla_{1}}\))
\end{itemize} 
Our singular mirror \(Y^{\vee}\) connects \(S^{\vee}\) 
and \(\tilde{Y}^{\vee}\) in a nice way; they form an extremal
transition
\begin{equation}
    \begin{tikzcd}
    &\tilde{Y}^{\vee}\ar[d,rightsquigarrow] &\\
    &Y^{\vee} &S^{\vee}\ar[l]
  \end{tikzcd}
\end{equation}
on the dual side; there exists a contraction from 
\(S^{\vee}\) to \(Y^{\vee}\) and a complex smoothing 
from \(Y^{\vee}\) to \(\tilde{Y}^{\vee}\).

This can be directly seen by applying the same construction
on the dual side. As a byproduct, we proved the following
theorem regarding Morrison's conjecture.

\begin{theorem}
\label{thm:morrison}
Let \((\Delta,\nabla)\) be a pair of reflexive polytopes
in revelant vector spaces. Then
Morrison's conjecture holds for anti-canonical 
CY hypersurfaces in MPCP desingularizations
of \(\mathbf{P}_{{\Delta}_{1}}\), \(\mathbf{P}_{{\Delta}_{2}}\),
\(\mathbf{P}_{{\nabla}_{1}}\), and \(\mathbf{P}_{{\nabla}_{2}}\).
\end{theorem}

\begin{remark}
For \(r\ge 2\), under Hypothesis \ref{assumption} for \(X\),
a crepant resolution \(\tilde{Y}\to Y\) 
of the singular CY double cover \(Y\) still exists
\cite{2013-Sheng-Xu-Zuo-maximal-families-of-calabi-yau-manifolds-with-minimal-length-yukawa-coupling}. However, it is not clear to us whether
\(\tilde{Y}\) remains an anti-canonical hypersurface in 
a suitable toric variety. Nevertheless, it provides
a nice model to study the crepant transformation conjecture (CTC)
and find explicit relations between ordinary Gromov--Witten invariants
of \(\tilde{Y}\) and orbifold Gromov--Witten 
invariants of \(Y\).
We regard \(Y\) as a subvariety in \(\mathbb{L}\), the total
space of the anti-canonical bundle of \(X\). The
resolution algorithm provided in 
\cite{2013-Sheng-Xu-Zuo-maximal-families-of-calabi-yau-manifolds-with-minimal-length-yukawa-coupling} is a sequence of blow-ups of \(\mathbb{L}\)
along smooth subvarieties. 
One expects to generalize Lai's result 
\cite{2009-Lai-GW-invariants-of-blow-ups-along-submanifolds-with-convex-normal-bundles} appropriately and compare the invariants of 
\(\tilde{Y}\) and \(Y\) directly. 
\end{remark}

\appendix

\section{Comparison of computations}
\label{app:comparison_of_computations}
Let \(Y\) be the CY double cover of \(\mathbb{P}^{3}\) branched
along four hyperplanes and a quartic in general position.
In this section, we explicitly compute
the genus zero orbifold Gromov--Witten invariants of \(Y\) using 
another embedding $Y\subset Z$ where $Z$ has \(2\) 
K\"{a}hler moduli. We show a (tricky) way to specialize them to the 
K\"{a}hler moduli of \(Y\) and recover the untwisted Gromov--Witten 
invariants we computed earlier by our mirror theorem. 
The comparison shows that embedding into spaces with high 
dimensional K\"{a}hler moduli can make such computation complicated.

Let us retain the notation from \S\ref{subsec:an_instanton_calculation}.
As in \S\ref{subsec:the_graph_embedding_and_the_pre_quotient_space}, 
the section \(f\) gives rise to the ``graph embedding''
\begin{eqnarray}
  \Gamma_{f}\colon \mathbb{P}^{3}\to \mathbf{P}_{\mathbb{P}^{3}}(\mathbb{L}\oplus\mathbb{C})
\end{eqnarray}
where \(\mathbb{L}\) is the total space of 
the line bundle whose sheaf of sections is \(\mathscr{O}_{\mathbb{P}^{3}}(4)\).
Let us describe the toric data for \(\mathbf{P}_{\mathbb{P}^{3}}(\mathbb{L}\oplus\mathbb{C})\). 
The \(1\)-cones are given by
\begin{align}
\begin{split}
    &u_{1}:=(1,0,0,-1),\\
    &u_{2}:=(0,1,0,-1),\\
    &u_{3}:=(0,0,1,-1),\\
    &u_{4}:=(-1,-1,-1,-1),\\
    &u_{5}:=(0,0,0,1),\\
    &u_{6}:=(0,0,0,-1).
\end{split}
\end{align}
In this presentation, the divisor of the zero section of \(\mathbb{L}\)
corresponds to \(u_{5}\), whereas the divisor at infinity
corresponds to \(u_{6}\).

Put \(Z=\mathbf{P}_{\mathbb{P}^{3}}(\mathbb{L}\oplus\mathbb{C})\) as before
and denote by \(D_{i}\) the toric divisor corresponding to \(u_{i}\).
Then we have \(\mathrm{H}^{2}(Z;\mathbb{C})=\mathbb{C}h\oplus \mathbb{C}\xi\). Here
\(h\) is the cohomology class of \(D_{1}\) and \(\xi\) is the cohomology 
class of \(D_{6}\). Then
\begin{equation}
    \begin{cases}
    D_{1}=D_{2}=D_{3}=D_{4}=h,\\
    D_{5}=\xi+4h,\\
    D_{6}=\xi.
    \end{cases}
\end{equation}
Now we study the Mori cone. It is known that 
the Mori cone of \(Z\) is generated by primitive relations
\begin{equation}
    \{u_{1},u_{2},u_{3},u_{4}\},~\mbox{and}~\{u_{5},u_{6}\}.
\end{equation}
Denote by \(\ell_{1}\) and \(\ell_{2}\) the corresponding extremal curves.
Then one can compute
\begin{equation}
    \begin{cases}
    h.\ell_{1}=1,\\
    \xi.\ell_{1}=-4,\\
    \end{cases}~\mbox{and}~
    \begin{cases}
    h.\ell_{2}=0,\\
    \xi.\ell_{2}=1.
    \end{cases}
\end{equation}
The Gromov--Written invariants of \(Y'\) (a CY hypersurface in \(Z\)) 
can be calculated by
the cohomology-valued series (the \(I\)-function)
\begin{align}
\begin{split}
    &\alpha e^{(t_{1}h+t_{2}\xi)/\alpha}(2\xi+8h)\\
    &\times\sum_{d_{1},d_{2}\ge 0}\frac{\prod_{m=1}^{2d_{2}}
    (2\xi+8h+m\alpha)
    q_{1}^{d_{1}}q_{2}^{d_{2}}}{\prod_{m=1}^{d_{1}}(h+m\alpha)^{4}\prod_{m=1}^{d_{2}}(\xi+4h+m\alpha)
    \prod_{m=1}^{d_{2}-4d_{1}}(\xi+m\alpha)}.
\end{split}
\end{align}
We note that 
\begin{eqnarray}
    \xi\cdot (\xi+4h)=0,
\end{eqnarray}
so the summation is reduced to 
\begin{align}
\begin{split}
    &\alpha e^{(t_{1}h+t_{2}\xi)/\alpha}(2\xi+8h)\\
    &\times\sum_{d_{2}\ge 
    4d_{1}\ge 0}\frac{\prod_{m=1}^{2d_{2}}
    (2\xi+8h+m\alpha)
    q_{1}^{d_{1}}q_{2}^{d_{2}}}{\prod_{m=1}^{d_{1}}(h+m\alpha)^{4}\prod_{m=1}^{d_{2}}(\xi+4h+m\alpha)
    \prod_{m=1}^{d_{2}-4d_{1}}(\xi+m\alpha)}.
\end{split}
\end{align}
Now restricting this series to \(Y'\) and observing that \(\xi\vert_{Y'}\)
is trivial, we obtain
\begin{equation}
    \alpha e^{(t_{1}h)/\alpha}(8h)\sum_{d_{2}\ge 
    4d_{1}\ge 0}\frac{\prod_{m=1}^{2d_{2}}
    (8h+m\alpha)
    q_{1}^{d_{1}}q_{2}^{d_{2}}}{\prod_{m=1}^{d_{1}}(h+m\alpha)^{4}\prod_{m=1}^{d_{2}}(4h+m\alpha)
    \prod_{m=1}^{d_{2}-4d_{1}}(m\alpha)}.
\end{equation}
In order to compute the invariants, a change of variable is needed.
Let us denote by \((Q_{1},Q_{2}):=m(q_{1},q_{2})\) the mirror map. The inverse of the mirror
map is given by
\begin{align}
\begin{split}
    q_{1}&=Q_1 - 16Q_1Q_2 + 96Q_1Q_2^2 - 256Q_1Q_2^3 +256Q_1Q_2^4\\ 
    &- 15808Q_1^2Q_2^4 + 252928Q_1^2Q_2^5 - 1517568Q_1^2Q_2^6 \\
    &+ 4046848Q_1^2Q_2^7 - 4046848Q_1^2Q_2^8 + \cdots,\\
    q_{2}&=Q_{2}.
\end{split}
\end{align}
Under this map, we obtain the series
\begin{align}
\begin{split}
\mathbf{1} + &(14752Q_1Q_2^4 + 128838600Q_1^2Q_2^8\\ 
&+ \frac{19220227397632}{9}Q_1^3Q_2^{12}
+ 46386112081796274Q_1^4Q_2^{16}\\
& + \frac{29242279664078082314752}{25}Q_1^5Q_2^{20}+\cdots)h^{2}\\
 + &(-59008Q_1Q_2^4 - 257677200Q_1^2Q_2^8\\ 
&- \frac{76880909590528}{27}Q_1^3Q_2^{12}
- 46386112081796274Q_1^4Q_2^{16}\\ 
&- \frac{116969118656312329259008}{125}Q_1^5Q_2^{20}+\cdots)h^{3}.
\end{split}
\end{align}
So we obtain the generating series
\begin{align}
    \begin{split}
    &2+29504Q_1Q_2^4 + 128838600Q_1^2Q_2^8\\ 
    &+ \frac{38440454795264}{27}Q_1^3Q_2^{12} 
    + 23193056040898137Q_1^4Q_2^{16} \\
    &+ \frac{58484559328156164629504}{125}Q_1^5Q_2^{20}+\cdots.
    \end{split}
\end{align}
Now let \(Q:=Q_{1}Q_{2}^{4}\). We obtain the desired series
\begin{align}
    \begin{split}
    &2+29504Q + 128838600Q^{2} + 
    \frac{38440454795264}{27}Q^{3} \\
    &+ 23193056040898137Q^{4} 
    + \frac{58484559328156164629504}{125}Q^{5}+\cdots.
    \end{split}
\end{align}
We can argue as in Proposition \ref{prop:1} and conclude that 
this is the generating series for the genus zero untwisted orbifold
Gromov--Witten
invariants for \(Y'\slash G'\).
Note that this is equivalent to \eqref{eq:a-model-predicted-correlations-4h};
there is a factor \(d^{3}\) in 
the degree \(d\) term due to the divisor axiom.

This is not surprising since we are computing the invariants for
the same variety. However, this indicates that
the calculation seems very complicated if the K\"{a}hler moduli is higher-dimensional.




\end{document}